\documentclass[11pt]{article}
\pdfoutput=1 % for the arXiv version

\usepackage[english]{babel}

\usepackage{geometry}
\usepackage{amsmath,bm}
\numberwithin{equation}{section} % change the numbering of equations to (#Section.#EquationInSection)

\usepackage[margin = 1em]{subfig}

\usepackage{enumitem}

\usepackage{graphicx} \graphicspath{{./}{Figures/}}

\usepackage{amsthm}

\newtheoremstyle{italic}% Name
{5pt}% Space above
{5pt}% Space below
{\itshape}% Body font
{}% Indent amount
{}% Theorem head font
{}% Punctuation after theorem head
{.5em}% Space after theorem head
{\bfseries{\thmname{#1}~\thmnumber{#2}.}\thmnote{~\textnormal{(#3)}}}% Theorem head spec (can be left empty, meaning ‘normal’)

\newtheoremstyle{italic_borrowed}% Name
{5pt}% Space above
{5pt}% Space below
{\itshape}% Body font
{}% Indent amount
{}% Theorem head font
{}% Punctuation after theorem head
{.5em}% Space after theorem head
{\bfseries{\thmname{#1}~\thmnumber{#2}.}\thmnote{~\textnormal{[#3]}}}% Theorem head spec (can be left empty, meaning ‘normal’)

\newtheoremstyle{upright}% Name
{5pt}% Space above
{5pt}% Space below
{\upshape}% Body font
{}% Indent amount
{\bfseries}% Theorem head font
{}% Punctuation after theorem head
{.5em}% Space after theorem head
{\bfseries{\thmname{#1}~\thmnumber{#2}.}\thmnote{~\textnormal{(\textit{#3}\textrm{)}}}}% Theorem head spec (can be left empty, meaning ‘normal’)

\theoremstyle{italic}
\newtheorem{theorem}{Theorem}[section]
\newtheorem{lemma}[theorem]{Lemma}
\newtheorem{proposition}[theorem]{Proposition}
\newtheorem{corollary}[theorem]{Corollary}

\theoremstyle{italic_borrowed}

\theoremstyle{upright}

\newtheorem{remark}[theorem]{Remark}

\usepackage{dsfont}
\usepackage{amsfonts}

\usepackage{xparse}

\usepackage{mathtools}

\newcommand\blfootnote[1]{%
  \begingroup%
  \renewcommand\thefootnote{}\footnote{#1}%
  \addtocounter{footnote}{-1}%
  \endgroup%
}%

\newcommand{\unit}{u}% complex unit root
\newcommand{\defi}{\coloneqq}% definition
\newcommand{\ifed}{\eqqcolon}% definition other way around

\newcommand{\modu}[1]{\operatorname{mod}(#1)} % modulus
\newcommand{\dinf}{\textup{d}} % upright d

\newcommand{\Nat}{\mathbb{N}} % natural numbers starting at 1
\newcommand{\Int}{\mathbb{Z}} % integer numbers
\newcommand{\Real}{\mathbb{R}} % real numbers
\newcommand{\Complex}{\mathbb{C}} % complex numbers

\NewDocumentCommand \E { m g }{%
    \IfNoValueTF{#2}% Check if second argument is given
        {\mathbb{E}[#1]}% If argument is not given, then use only the first argument
        {\mathbb{E}_{#1}[#2]} % If both arguments are given, use both
    }%
\NewDocumentCommand \Esc { m g }{%
    \IfNoValueTF{#2}% Check if second argument is given
        {\mathbb{E}\left[#1\right]}% If argument is not given, then use only the first argument
        {\mathbb{E}_{#1}\!\!\left[#2\right]} % If both arguments are given, use both
    }%
\NewDocumentCommand \Efxd { m g }{%
    \IfNoValueTF{#2}% Check if second argument is given
        {\mathbb{E}\Bigl[#1\Bigr]}% If argument is not given, then use only the first argument
        {\mathbb{E}_{#1}\!\Bigl[#2\Bigr]} % If both arguments are given, use both
    }%
\NewDocumentCommand \Prob { m g }{%
    \IfNoValueTF{#2}% Check if second argument is given
        {\mathbb{P}(#1)}% If argument is not given, then use only the first argument
        {\mathbb{P}_{#1}(#2)} % If both arguments are given, use both
    }%
\NewDocumentCommand \Probfxd { m g }{%
    \IfNoValueTF{#2}% Check if second argument is given
        {\mathbb{P}\Bigl(#1\Bigr)}% If argument is not given, then use only the first argument
        {\mathbb{P}_{#1}\!\Bigl(#2\Bigr)} % If both arguments are given, use both
    }%
\newcommand{\bld}[1]{\mathbf{#1}} % bold notation for vectors
\newcommand{\oneb}{\bld{1}} % column vector of ones
\NewDocumentCommand \eb { m g }{%
    \IfNoValueTF{#2}% Check if second argument is given
        {\bld{e}_{#1}}% If argument is not given, then use only the first argument
        {\bld{e}^{(#1)}_{#2}} % If both arguments are given, use both
    }%
\NewDocumentCommand \zerob { g }{%
    \IfNoValueTF{#1}% Check if optional argument is given
        {\bld{0}}% If optional argument is not given, then use no arguments
        {\bld{0}^{(#1)}} % If optional arguments are given, use it
    }%

\newcommand{\la}{\lambda} % arrival rate
\newcommand{\al}{\alpha}
\newcommand{\be}{\beta}

\newcommand{\fixedlengthrightarrow}[1]{\xrightarrow{\mathmakebox[4.5em]{#1}}}

\newcommand{\A}[1]{A_{#1}} % matrix A
\newcommand{\B}[1]{B_{#1}} % matrix B
\newcommand{\I}{I} % identity matrix I
\newcommand{\M}[1]{M^{(#1)}} % matrix M with 1 element, shown in superscript
\newcommand{\Lo}{L} % lower diagonal matrix L

\newcommand{\pb}[1]{\bld{p}(#1)} % column vector p
\newcommand{\p}[1]{p(#1)} % single element of the column vector p
\newcommand{\qb}[1]{\bld{q}(#1)} % column vector q
\newcommand{\q}[1]{q(#1)} % single element of the column vector q

\newcommand{\pbmodi}[1]{\hat{\bld{p}}(#1)} % column vector p for the modified model
\newcommand{\pmodi}[1]{\hat{p}(#1)} % single element of the column vector p for the modified model
\newcommand{\qbmodi}[1]{\hat{\bld{q}}(#1)} % column vector q for the modified model
\newcommand{\qmodi}[1]{\hat{q}(#1)} % single element of the column vector q for the modified model

\newcommand{\Dpos}[1]{D_{\scriptscriptstyle +}(#1)} % matrix used in computation of solution to positive inner equations
\newcommand{\Dneg}[1]{D_{\scriptscriptstyle -}(#1)} % matrix used in computation of solution to negative inner equations

\newcommand{\ibpos}[1]{\bld{i}_{\scriptscriptstyle +}(#1)} % column eigenvector positive interior
\newcommand{\ipos}[1]{i_{\scriptscriptstyle +}(#1)} % element of the column eigenvector for the positive interior
\newcommand{\ibneg}[1]{\bld{i}_{\scriptscriptstyle -}(#1)} % column eigenvector negative interior
\newcommand{\ineg}[1]{i_{\scriptscriptstyle -}(#1)} % element of the column eigenvector for the negative interior

\newcommand{\Func}[1]{\Psi(#1)} % function used in the initial solution
\newcommand{\funcparam}{\psi} % variable name for the roots (also need it separately)
\newcommand{\func}[2]{\funcparam_{\scriptscriptstyle #1}(#2)} % roots used in the initial solution

\newcommand{\hb}{\bld{h}} % column vector h
\newcommand{\h}[1]{h(#1)} % element of the column vector h

\newcommand{\ch}{\eta} % coefficient for the horizontal compensation
\newcommand{\cv}{\nu} % coefficient for the vertical compensation

\newcommand{\Lboundal}{c_1}
\newcommand{\Lboundbe}{c_2}
\newcommand{\Lrootpos}{v}
\newcommand{\Lrootneg}{w}
\newcommand{\Lal}{\Lrootpos_{\scriptscriptstyle -}}
\newcommand{\Lalt}{\Lrootneg_{\scriptscriptstyle -}}
\newcommand{\Lbe}{\Lrootpos_{\scriptscriptstyle +}}
\newcommand{\Lbet}{\Lrootneg_{\scriptscriptstyle +}}
\newcommand{\Lcoeffpos}{\gamma}
\newcommand{\Lcoeffneg}{\theta}
\newcommand{\Lcv}{\Lcoeffpos_{\cv}}
\newcommand{\Lcvt}{\Lcoeffneg_{\cv_{s + 1}}}
\newcommand{\Lchpos}{\Lcoeffpos_{\ch}}
\newcommand{\Lchtpos}{\Lcoeffpos_{\ch_{s + 1}}}
\newcommand{\Lchneg}{\Lcoeffneg_{\ch}}
\newcommand{\Lchtneg}{\Lcoeffneg_{\ch_{s + 1}}}

\DeclareRobustCommand{\patternNE}{\includegraphics{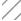}}
\DeclareRobustCommand{\patternhor}{\includegraphics{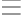}}
\DeclareRobustCommand{\patternver}{\includegraphics{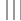}}

\title{Steady-state analysis of shortest expected delay routing}%
\author{Jori Selen\footnotemark[1] \footnotemark[2], Ivo J.B.F. Adan\footnotemark[1] \footnotemark[2] \footnotemark[3], Stella Kapodistria\footnotemark[2], Johan S.H. van Leeuwaarden\footnotemark[2]}%

\begin{document}%

\maketitle%

\renewcommand{\thefootnote}{\fnsymbol{footnote}}%
\footnotetext[1]{Department of Mechanical Engineering, Eindhoven University of Technology}%
\footnotetext[2]{Department of Mathematics and Computer Science, Eindhoven University of Technology}%
\footnotetext[3]{Department of Industrial Engineering \& Innovation Sciences, Eindhoven University of Technology}%
\blfootnote{E-mail address: {\tt j.selen@tue.nl}}%
\renewcommand{\thefootnote}{\arabic{footnote}} \setcounter{footnote}{0}%

\begin{abstract}%
We consider a queueing system consisting of two non-identical exponential servers, where each server has its own dedicated queue and serves the customers in that queue FCFS. Customers arrive according to a Poisson process and join the queue promising the shortest expected delay, which is a natural and near-optimal policy for systems with non-identical servers. This system can be modeled as an inhomogeneous random walk in the quadrant. By stretching the boundaries of the \textit{compensation approach} we prove that the equilibrium distribution of this random walk can be expressed as a series of product-forms that can be determined recursively. The resulting series expression is directly amenable for numerical calculations and it also provides insight in the asymptotic behavior of the equilibrium probabilities as one of the state coordinates tends to infinity.
\end{abstract}%

%%%%%%%%%%%%%%%%%%%%%%%%%%%%%%%%%%%%%%%%%%%%%%%%%%%%%%%
%%%%%%%%%%%%%%%%%%%%%%%%%%%%%%%%%%%%%%%%%%%%%%%%%%%%%%%
%%%%%%%%%%%%%%%%%%%%% NEW SECTION %%%%%%%%%%%%%%%%%%%%%
%%%%%%%%%%%%%%%%%%%%%%%%%%%%%%%%%%%%%%%%%%%%%%%%%%%%%%%
%%%%%%%%%%%%%%%%%%%%%%%%%%%%%%%%%%%%%%%%%%%%%%%%%%%%%%%

\section{Introduction}%
\label{sec:introduction}%

In this paper we analyze the performance of a system with two servers under the shortest expected delay (SED) routing policy. This routing policy assigns an arriving customer to the queue that has the shortest expected delay (sojourn time), where delay refers to the waiting time plus the service time. This policy arises naturally in various application areas, and poses considerable mathematical challenges.

In particular, we focus on a queueing system with two servers, where each server has its own queue with unlimited buffer capacity. All service times are independent and exponentially distributed, but the two servers have different service rates, i.e.~respectively 1 and $s$. In both queues customers are served in FCFS-order. Customers arrive according to a Poisson process with rate $\la$ and upon arrival join one of the two queues according to the following mechanism: Let $q_1$ and $q_2$ be the number of customers in queue 1 and 2, respectively, including a possible customer in service. For an arriving customer, the expected delay in queue 1 is $q_1 + 1$ and in queue 2 is $(q_2 + 1)/s$. The SED routing policy assigns an arriving customer to queue 1 if $q_1 + 1 < (q_2 + 1)/s$ and to queue 2 if $q_1 + 1 > (q_2 + 1)/s$. In case the expected delays in both queues are equal, i.e.~$q_1 + 1 = (q_2 + 1)/s$, the arriving customer joins queue 1 with probability $q$ and queue 2 with probability $1 - q$. Once a customer has joined a queue, no switching or jockeying is allowed. The service times are assumed independent of the arrival process and the customer decisions. This system is stable if and only if, see e.g.~\cite[Theorem~1]{AllocationPolicies_Coombs2003},
\begin{equation}%
\rho \defi \la / (1 + s) < 1. \label{eqn:stability_condition}
\end{equation}%
We refer to this specific queueing system as the SED system. The SED system can be modeled as a two-dimensional Markov process on the states $(q_1,q_2)$ with the transition rate diagram as in Figure~\ref{fig:trd_two-dimensional}. From the  transition rate diagram it is evident that this state description leads to an inhomogeneous random walk in the quadrant, making an exact analysis extremely difficult. The inhomogeneous behavior occurs along the line $s(q_1 + 1) = q_2 + 1$ and it can be expected that the solution structure of the stationary probabilities above and below this line will be different. Moreover, for $s \neq 1$, this line divides the state space in unequal proportions, further increasing the complexity of an exact analysis.

\begin{figure}%
\centering%
\includegraphics{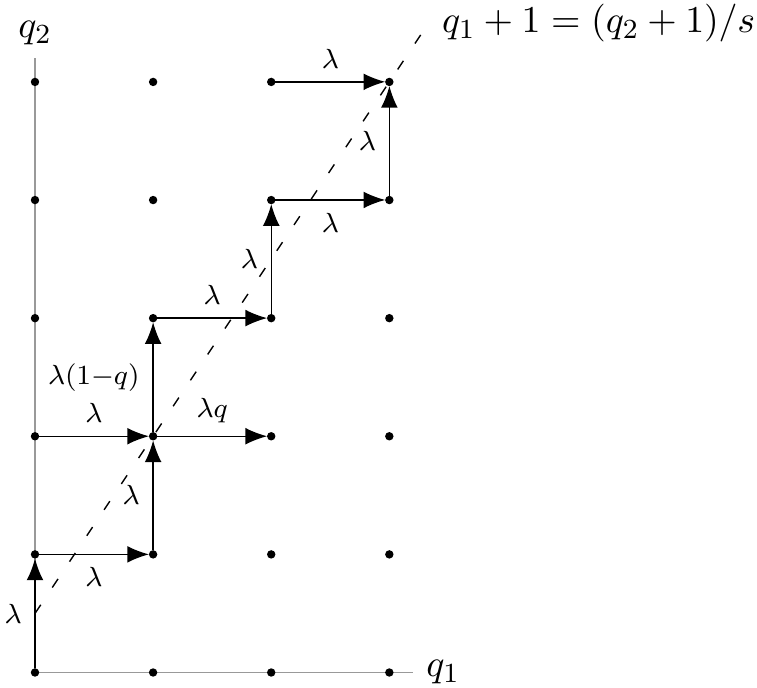} % PDF figure
\caption{Transition rate diagram of the Markov process on the state space $(q_1,q_2)$ with $s = \frac{3}{2}$. Only transitions corresponding to arrivals that reach or cross the dashed line are shown.}%
\label{fig:trd_two-dimensional}
\end{figure}%

Under the assumption of identical service rates, SED routing becomes join the shortest queue (JSQ) routing which is known to minimize the mean expected delay \cite{DecidingCounterExamples_Whitt1986}. If the service rates of the two servers are different, then JSQ is not optimal. As Whitt \cite{DecidingCounterExamples_Whitt1986} points out, if the system if fully observable and we are a priori knowledgeable of the service times of waiting customers, for example, by looking at the shopping baskets in a supermarket, then the natural choice is to join the queue promising the smallest sojourn time in the system, instead of the shorter of the queues. However, this type of information is too detailed and might not always be available. In particular, there are many situations in which we have limited knowledge of the service times of the waiting customers. Such systems include the teller waiting lines, production facilities, communication networks, etc. If the only available information is that of the number of waiting customers at each queue, then it is quite natural, although not necessarily optimal, to choose the shortest queue routing \cite{CompApproachSymmetric_Adan1990,CompApproachAsymmetric_Adan1991,BadLuck_Blanc2009,AsymmetricalShortest_Cohen1998,TwoParallelQueues_Flatto1977,JSQ_Stability_Foley2001, JSQ_WebServerFarms_Gupta2007,ShortestQueue_Halfin1985,TwoParallelQueues_Kingman1961,GeometricDecay_Li2007,InfiniteServer_Yao2005,JSQ_ErlangLoss_Yao2008}. However, if on top of the number of waiting customers we can estimate the expected service times at the two queues, then the natural choice is to route the customers according to SED, rather than JSQ.

SED routing seems to be a natural choice, but in practice this choice is not always optimal, since it does not minimize the mean stationary delay \cite{DynamicRouting_Ephremides1980,DecidingCounterExamples_Whitt1986}. For the two non-identical server setting, Hajek~\cite{TwoInteractingServiceStations_Hajek1984} solves a Markov decision problem and proves that the optimal routing policy is of the threshold-type. However, SED routing exhibits relatively good performance at both ends of the range of system utilizations, but performs slightly worse than other policies at medium utilizations, see \cite{LoadSharing_Banawan1992}. Furthermore, Foschini~\cite{SEDR_HeavyTraffic_Foschini1977} has shown that SED routing is asymptotically optimal and results in complete resource pooling in the heavy traffic limit.

State-dependent routing policies such as JSQ and SED are typically difficult to analyze. For the stationary behavior of queueing systems with SED routing very little is known. The difficulty in analyzing this type of models is evident from the analysis of the JSQ policy for two identical parallel exponential servers \cite{TwoParallelQueues_Haight1958}. The first major steps towards its analysis were made in \cite{TwoParallelQueues_Flatto1977,TwoParallelQueues_Kingman1961}, using a uniformization approach that established that the generating function of the equilibrium distribution is meromorphic. Thus, a partial fraction expansion of the generating function of the joint stationary queue length distribution would in principle lead to a representation of the equilibrium distribution as an infinite linear combination of product forms. For an extensive treatment of the JSQ system using a generating function approach, the interested readers is referred to \cite{BoundaryValueProblemsQueueing_Cohen2000,RandomWalksQP_Fayolle1999}. An alternative approach that is not based on generating functions is the \textit{compensation approach} \cite{CompApproachSymmetric_Adan1990}. This approach directly solves the equilibrium equations and leads to an explicit solution. The essence of the approach is to first characterize the product forms satisfying the equilibrium equations for states in the inner region of the two-dimensional state space and then to use the product forms in this set to construct a linear combination of product forms which also satisfies the boundary conditions. The construction is based on a compensation argument: after introducing the starting product form, new product forms are added so as to alternately compensate for the errors on the two boundaries.

The compensation approach has been developed in a series of papers \cite{CompApproachMultipleWaitingLines_Adan2001,CompApproachSymmetric_Adan1990,CompApproachAsymmetric_Adan1991,CompApproachTwoDimensional_Adan1993} and aims at a direct solution for a class of two-dimensional random walks on the lattice of the first quadrant that obey the following conditions:
\begin{enumerate}[label = \textup{(\roman*)}]%
\item Step size: only transitions to neighboring states.
\item Forbidden steps: no transitions from interior states to the North, North-East, and East.
\item Homogeneity: the same transitions occur according to the same rates for all interior points, and similarly for all points on the horizontal boundary, and for all points on the vertical boundary.
\end{enumerate}%
The approach exploits the fact that the balance equations in the interior of the quarter plane are satisfied by linear (finite or infinite) combinations of product forms, that need to be chosen such that the equilibrium equations on the boundaries are satisfied as well. As it turns out, this can be done by alternatingly compensating for the errors on the two boundaries, which eventually leads to an infinite series of product forms. The SED queueing system in this paper is more complicated than the JSQ and other classical queueing systems, see e.g.~\cite{CompApproachMultipleWaitingLines_Adan2001,CompApproachErlangServers_Adan1996,CompApproachSymmetric_Adan1990,CompApproachAsymmetric_Adan1991,CompApproachTwoDimensional_Adan1993}, since the two-dimensional random walk that describes the SED system exhibits \textit{inhomogeneous} behavior in the interior of the quadrant. In this paper, we show that the compensation approach can nevertheless be further developed to overcome the obstacles caused by the inhomogeneous behavior of the random walk. This leads to a solution for the stationary distribution in the form of a \textit{tree of product forms}.

The only other work in this direction is \cite{CompApproachErlangServers_Adan1996}, which considers SED routing for two identical parallel servers with Poisson arrivals and Erlang distributed service times. The crucial difference with our setting is that we do not focus on generalizing service times, but instead consider servers with different service rates.

The remainder of the paper is organized as follows. In Section~\ref{sec:model_description} we introduce the model in detail and describe the equilibrium equations. We discuss the compensation approach and its methodological extensions together with our contribution in Section~\ref{sec:stretching_boundaries_compensation_approach}. Some numerical results are presented in Section~\ref{sec:numerical_results}. Section~\ref{sec:applying_the_compensation_approach} applies the compensation approach to determine the equilibrium distribution of the SED system as a series of product-form solutions. Finally, we present some conclusions in Section~\ref{sec:conclusion}.

%%%%%%%%%%%%%%%%%%%%%%%%%%%%%%%%%%%%%%%%%%%%%%%%%%%%%%%
%%%%%%%%%%%%%%%%%%%%%%%%%%%%%%%%%%%%%%%%%%%%%%%%%%%%%%%
%%%%%%%%%%%%%%%%%%%%% NEW SECTION %%%%%%%%%%%%%%%%%%%%%
%%%%%%%%%%%%%%%%%%%%%%%%%%%%%%%%%%%%%%%%%%%%%%%%%%%%%%%
%%%%%%%%%%%%%%%%%%%%%%%%%%%%%%%%%%%%%%%%%%%%%%%%%%%%%%%

\section{Equilibrium equations}%
\label{sec:model_description}%

The Markov process associated with the SED system has an inhomogeneous behavior in the interior of the quadrant, specifically, along the line $s(q_1 + 1) = q_1 + 1$, see Figure~\ref{fig:trd_two-dimensional}. In this section we transform the two-dimensional state space $(q_1,q_2)$ to a half-plane with a finite third dimension. For this state description, we show that the theoretical framework of the compensation approach can be extended and in this way we determine the equilibrium distribution of the SED system.

We will henceforth assume that $s$ is a positive integer number. The service rate $s$ could also be chosen to be rational and the analysis would be similar, but notationally more difficult. We further elaborate on this point in Remark~\ref{rem:rational_s}.

In queue 2 we count the number of groups of size $s$ and denote it as $j$, i.e. $j = \lfloor q_2/s \rfloor$, and we denote the number of remaining customers as $r$, i.e. $r = \modu{q_2,s}$. Clearly, a single group in queue 2 requires the same expected amount of work as a single customer in queue 1. The total number of customers in queue 2 is thus $js + r$ and for an arriving customer the expected delay in queue 2 is $j + (r + 1)/s$. In terms of these variables, SED routing works as follows: if $q_1 + 1 < j + (r + 1)/s$ the arriving customer joins queue 1 and if $q_1 + 1 > j + (r + 1)/s$ the arriving customer joins queue 2. In case the expected delays in both queues are equal, i.e.~$q_1 + 1 = j + (r + 1)/s$, the arriving customer joins queue 1 with probability $q$ and queue 2 with probability $1 - q$.

For convenience we introduce the length of the shortest queue $m = \min(q_1,j)$ and the difference between queue 2 and queue 1, i.e.~$n = j - q_1$. Using this notation, the SED system is formulated as a three-dimensional Markov process with state space $\{(m,n,r) \mid m \in \Nat_0, ~ n \in \Int, ~ r = 0,1,\ldots,s-1 \}$. Under the stability condition \eqref{eqn:stability_condition} the equilibrium distribution exists. Let $p(m,n,r)$ denote the equilibrium probability of being in state $(m,n,r)$ and let
\begin{equation}%
\pb{m,n} = (\p{m,n,0},\p{m,n,1},\ldots,\p{m,n,s-1})^T
\end{equation}%
with $\bld{x}^T$ the transpose of a vector $\bld{x}$. Throughout the paper, we use bold lowercase letters or numbers for vectors and uppercase Latin letters for matrices. For convenience, we have listed all state variables and their interpretation in Table~\ref{tbl:variables}.

\begin{table}%
\centering%
\begin{tabular}{ccl}%
Variable & Expression & Interpretation \\
\hline
$q_1$ &  & Number of customers (groups of size 1) in queue 1 \\
$q_2$ &  & Number of customers in queue 2 \\
$\lfloor q_2/s \rfloor$ & & Number of groups of size $s$ in queue 2 \\
$m$   & $m = \min(q_1,\lfloor q_2/s \rfloor)$ & Minimum number of groups in queue 1 and 2 \\
$n$   & $n = \lfloor q_2/s \rfloor - q_1$ & Difference between number of groups in queue 2 and 1 \\
$r$   & $r = \modu{q_2,s}$ & Number of customers in queue 2 that are not in a group
\end{tabular}%
\caption{Interpretation of the state variables of the Markov process.}
\label{tbl:variables}%
\end{table}%

The transition rates are given by
\begin{align*}%
(m,n,r) &\fixedlengthrightarrow{(1+s)\rho} \begin{cases}%
(m+1,n-1,r), & m \ge 0, ~ n > 0, ~ r = 0,1,\ldots,s-1, \\
(m,n,r+1), & m \ge 0, ~ n \le 0, ~ r = 0,1,\ldots,s-2, \\
(m+1,n+1,0), & m \ge 0, ~ n < 0, ~ r = s-1,
\end{cases} \\
(m,0,s-1) &\fixedlengthrightarrow{(1+s)\rho(1-q)} (m,1,0), ~ m \ge 0, \\
(m,0,s-1) &\fixedlengthrightarrow{(1+s)\rho q} (m,-1,s-1), ~ m \ge 0,
\intertext{corresponding to arrivals, and}
(m,n,r) &\fixedlengthrightarrow{1} \begin{cases}%
(m-1,n+1,r), & m > 0, ~ n \ge 0, ~ r = 0,1,\ldots,s-1, \\
(m,n+1,r), & m \ge 0, ~ n < 0, ~ r = 0,1,\ldots,s-1,
\end{cases} \\
(m,n,r) &\fixedlengthrightarrow{s} \begin{cases}%
(m,n,r-1), & m \ge 0, ~ n \in \Int, ~ r = 1,2,\ldots,s-1, \\
(m,n-1,s-1), & m \ge 0, ~ n > 0, ~ r = 0, \\
(m-1,n-1,s-1), & m > 0, ~ n \le 0, ~ r = 0,
\end{cases}%
\end{align*}%
corresponding to service completions. Figure~\ref{fig:state_space_trd}(a) displays the transition rate diagram for the three-dimensional state space. The transition rates are described by the matrices $\A{x,y}$ in the positive quadrant and $\B{x,y}$ in the negative quadrant, where the pair $(x,y)$ indicates the step size in the $(m,n)$-direction. Let $\I$ be the $s \times s$ identity matrix, $\M{x,y}$ be an $s \times s$ binary matrix with element $(x,y)$ equal to one and zeros elsewhere, and $\Lo$ an $s \times s$ subdiagonal matrix with elements $(x,x-1), ~ x = 1,2,\ldots,s-1$ equal to one and zeros elsewhere. For consistency with the indexing of the vector $\pb{m,n}$, indexing of a matrix starts at 0. The transition rate matrices take the form
\begin{align*}%
\A{1,-1} &= (1 + s)\rho\I,                 &\A{0,1}  &= (1 + s)\rho(1-q) \M{0,s - 1}, \\
\A{-1,1} &= \B{0,1} = \I,                &\B{1,1}  &= (1 + s)\rho \M{0,s - 1}, \\
\A{0,-1} &= \B{-1,-1} = s\M{s-1,0},      &\B{0,-1} &= (1 + s)\rho q \M{s - 1,s - 1}, \\
\A{0,0}  &= -(1 + s)(\rho + 1)\I + s\Lo^T, &\B{0,0}  &= \A{0,0} + (1 + s)\rho\Lo.
\end{align*}%

\begin{figure}%
\centering%
\subfloat[Three-dimensional $(m,n,r)$ transition rate diagram. Note that the third dimension is perpendicular and gives rise to $s$-layers.]{%
\includegraphics{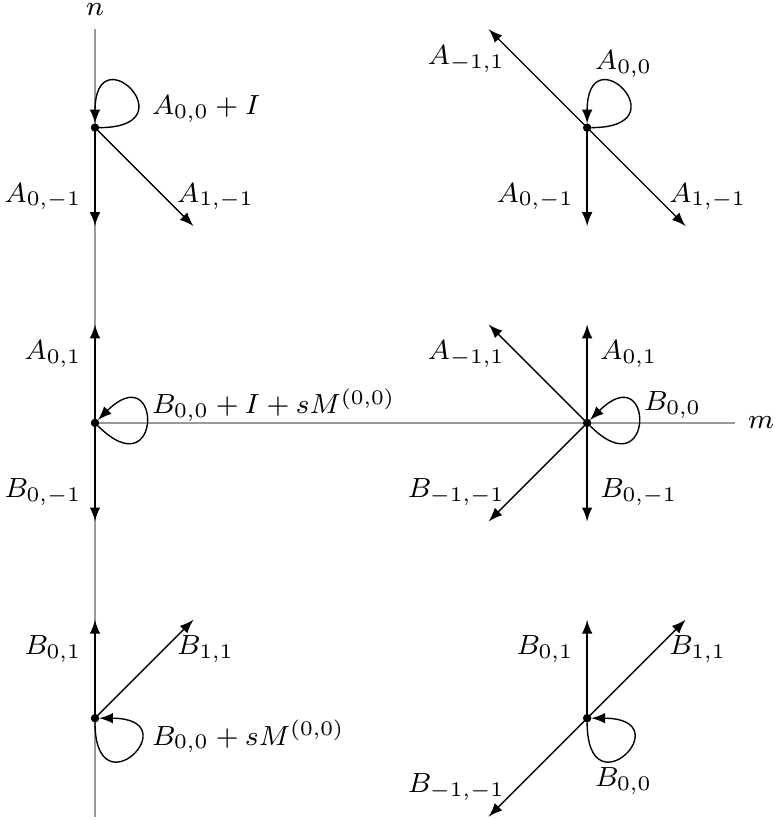}% PDF figure
}%
%\subfloat[Partitioning of the state space of the Markov process: interior (or inner) states {\protect\tikz \protect\draw[pattern = my north east lines, line space = 4pt, pattern color = black!40, draw = none] (0,0) rectangle (0.2,0.2);}; horizontal states {\protect\tikz \protect\draw[pattern = my horizontal lines, line space = 2pt, pattern color = black!40, draw = none] (0,0) rectangle (0.2,0.2);}; and vertical states {\protect\tikz \protect\draw[pattern = my vertical lines, line space = 2pt, pattern color = black!40, draw = none] (0,0) rectangle (0.2,0.2);}.]{%
\subfloat[Partitioning of the state space of the Markov process: interior (or inner) states \patternNE; horizontal states \patternhor; and vertical states \patternver.]{%
\includegraphics{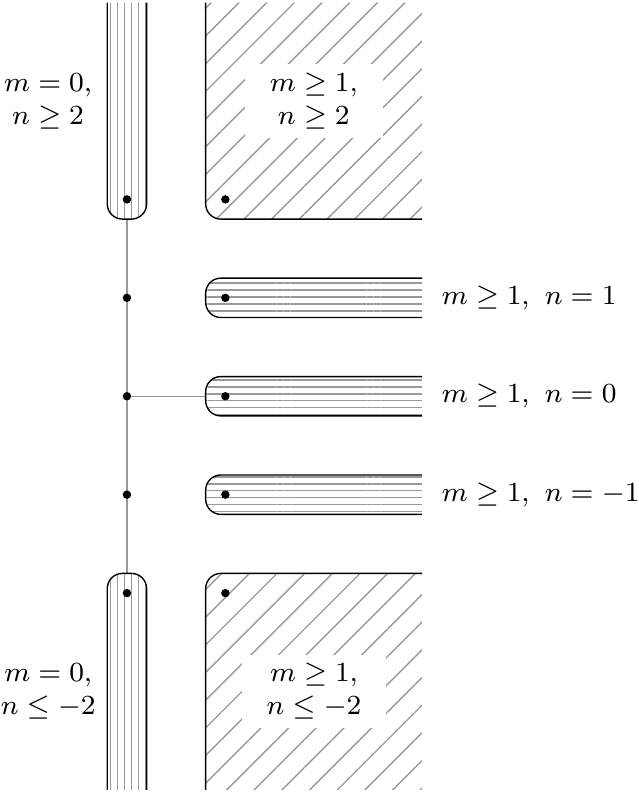}% PDF figure
}%
\caption{Transition rate diagram and state space partitioning.}
\label{fig:state_space_trd}
\end{figure}%

The equilibrium equations can be written in matrix-vector form. We partition the state space as illustrated in Figure~\ref{fig:state_space_trd}(b). For the interior of the positive and negative quadrant we have the following inner equations
\begin{align}%
\A{0,0} \pb{m,n} + \A{1,-1} \pb{m - 1,n + 1} + \A{0,-1} \pb{m,n + 1} \notag \\
+ \A{-1,1} \pb{m + 1,n - 1} &= \zerob, \quad m \ge 1, ~ n \ge 2, \label{eqn:equilibrium_eqs_I+} \\
\B{0,0} \pb{m,n} + \B{1,1} \pb{m - 1,n - 1} + \B{0,1} \pb{m,n - 1} \notag \\
+ \B{-1,-1} \pb{m + 1,n + 1} &= \zerob, \quad m \ge 1, ~ n \le -2. \label{eqn:equilibrium_eqs_I-}
\intertext{The equilibrium equations corresponding to the states on the horizontal axis, or directly adjacent to the horizontal axis, are referred to as the horizontal boundary equations and are given by}
\A{0,0} \pb{m,1} + \A{1,-1} \pb{m - 1,2} + \A{0,-1} \pb{m,2} \notag \\
+ \A{-1,1} \pb{m + 1,0} + \A{0,1} \pb{m,0} &= \zerob, \quad m \ge 1, ~ n = 1, \label{eqn:equilibrium_eqs_H+}\\
\B{0,0} \pb{m,-1} + \B{1,1} \pb{m - 1,-2} + \B{0,1} \pb{m,-2} \notag \\
+ \B{-1,-1} \pb{m + 1,0} + \B{0,-1} \pb{m,0} &= \zerob, \quad m \ge 1, ~ n = -1, \label{eqn:equilibrium_eqs_H-}\\
\B{0,0} \pb{m,0} + \A{1,-1} \pb{m - 1,1} + \B{1,1} \pb{m - 1,-1} \notag \\
+ \A{0,-1} \pb{m,1} + \B{0,1} \pb{m,-1} &= \zerob, \quad m \ge 1, ~ n = 0. \label{eqn:equilibrium_eqs_H}
\intertext{The vertical boundary equations are}
(\A{0,0} + \I)\pb{0,n} + \A{0,-1} \pb{0,n + 1} + \A{-1,1} \pb{1,n - 1} &= \zerob, \quad m = 0, ~ n \ge 2, \label{eqn:equilibrium_eqs_V+} \\
(\B{0,0} + s\M{0,0}) \pb{0,n} + \B{0,1} \pb{0,n - 1} + \B{-1,-1} \pb{1,n + 1} &= \zerob, \quad m = 0, ~ n \le -2. \label{eqn:equilibrium_eqs_V-}
\end{align}%
Finally, for the three remaining boundary states near the origin, we have
\begin{align}%
(\A{0,0} + \I)\pb{0,1} + \A{0,-1}\pb{0,2} + \A{-1,1} \pb{1,0} + \A{0,1} \pb{0,0} &= \zerob, \\
(\B{0,0} + s\M{0,0}) \pb{0,-1} + \B{0,1} \pb{0,-2} + \B{-1,-1} \pb{1,0} + \B{0,-1} \pb{0,0} &= \zerob,\\
(\B{0,0} + \I + s\M{0,0})\pb{0,0} + \A{0,-1} \pb{0,1} + \B{0,1} \pb{0,-1} &= \zerob.
\end{align}%
\begin{remark}[Rational $s$]\label{rem:rational_s}%
A system with a rational service rate $s = \frac{s_2}{s_1}$ can also be analyzed. In that case, one needs to consider a system with two servers and service rates $s_1$ and $s_2$. Similar to our analysis at the start of Section~\ref{sec:model_description}, one denotes the number of groups of size $s_1$ in queue 1 as $i$ and the number of groups of size $s_2$ in queue 2 as $j$. Then, let $r_n \in \{0,1,\ldots,s_n - 1\}$ denote the number of remaining customers in queue $n = 1,2$. Based on the aforementioned construction, a single group in either queue 1 or 2 requires the same expected amount of work. Lastly, set $m = \min(i,j)$, $n = j - i$ and the third finite dimension is a lexicographical ordering of the states $(r_1,r_2) \in \{0,1,\ldots,s_1 - 1\} \times \{0,1,\ldots,s_2 - 1\}$. This state space description leads to a transition rate diagram that has a similar structure as the one seen in Figure~\ref{fig:state_space_trd}(a). In this sense, a system with a rational service rate can be analyzed using the approach described in this paper.
\end{remark}

%%%%%%%%%%%%%%%%%%%%%%%%%%%%%%%%%%%%%%%%%%%%%%%%%%%%%%%
%%%%%%%%%%%%%%%%%%%%%%%%%%%%%%%%%%%%%%%%%%%%%%%%%%%%%%%
%%%%%%%%%%%%%%%%%%%%% NEW SECTION %%%%%%%%%%%%%%%%%%%%%
%%%%%%%%%%%%%%%%%%%%%%%%%%%%%%%%%%%%%%%%%%%%%%%%%%%%%%%
%%%%%%%%%%%%%%%%%%%%%%%%%%%%%%%%%%%%%%%%%%%%%%%%%%%%%%%

\section{Evolution of the compensation approach and our contribution}%
\label{sec:stretching_boundaries_compensation_approach}%

In this section we use the abbreviations: vertical boundary (VB); vertical compensation step (VCS); horizontal boundary (HB); and horizontal compensation step (HCS).

The compensation approach is used for the direct determination of the equilibrium distribution of Markov processes that satisfy the three conditions mentioned in Section~\ref{sec:introduction}. The key idea is a compensation procedure: the equilibrium distribution can be represented as a series of product-form solutions, which is generated term by term starting from an initial solution, such that each term compensates for the error introduced by its preceding term on one of the boundaries of the state space. In this section, we motivate why the SED system requires a fundamental extension of the compensation approach. We do so by first describing the evolution of the compensation approach through a series of models and present for each model the corresponding methodological contribution. Finally, we describe the extension required for the SED system.

The compensation approach was pioneered by Adan et al.~\cite{CompApproachSymmetric_Adan1990}, for a queueing system with two identical exponential servers, both with rate 1, and JSQ routing. Such a queueing system can be modeled as a Markov process with states $(q_1,q_2) \in \Nat_0^2$, where $q_i$ is the number of customers at queue $i$, including a customer possibly in service. By defining $m = \min(q_1,q_2)$ and $n = q_2 - q_1$, one transforms the state space from an inhomogeneous random walk in the quadrant to a random walk in the half plane that is homogeneous in each quadrant. Since the two quadrants are mirror images of each other, it is not needed to determine the equilibrium probabilities in both quadrants; it suffices to do so in the positive quadrant. The transition rate diagram of the Markov process is shown in Figure~\ref{fig:JSQ_various_models}(a).

\begin{figure}%
\centering%
\subfloat[Identical servers \cite{CompApproachSymmetric_Adan1990}]{%
\includegraphics{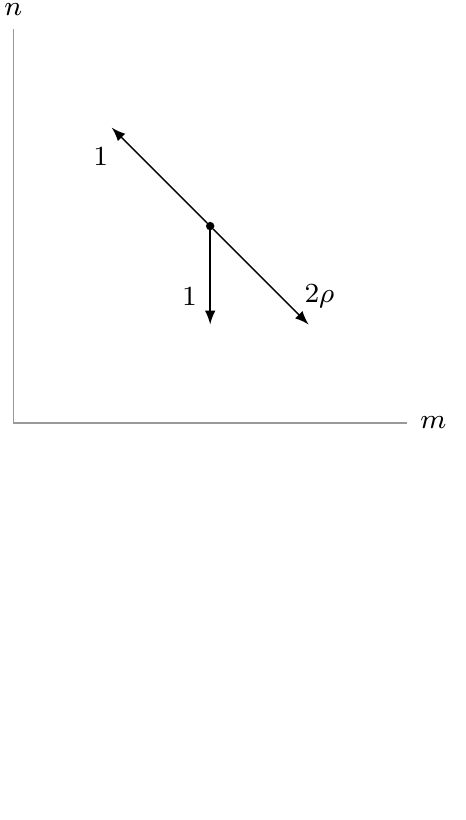}% PDF figure
}%
\subfloat[Non-identical servers \cite{CompApproachAsymmetric_Adan1991}]{%
\includegraphics{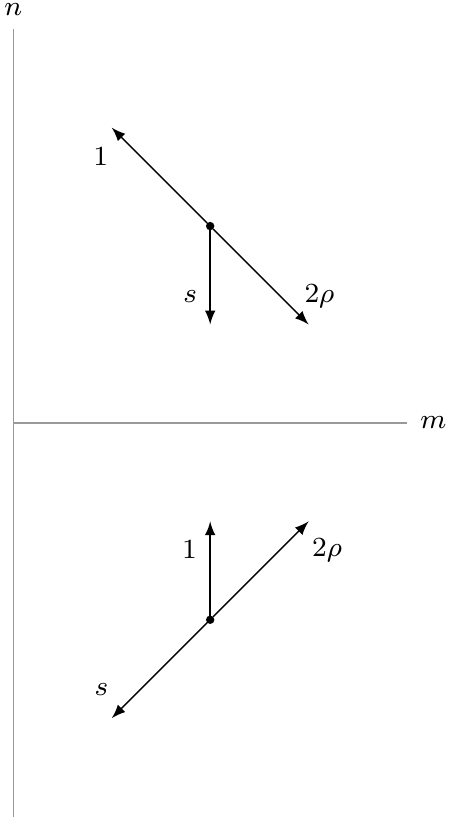}% PDF figure
}%
\subfloat[Identical servers and Erlang arrivals with $s$ phases ($s$-layered transition rate diagram) \cite{CompApproachErlangArrivals_Adan2013}]{%
\includegraphics{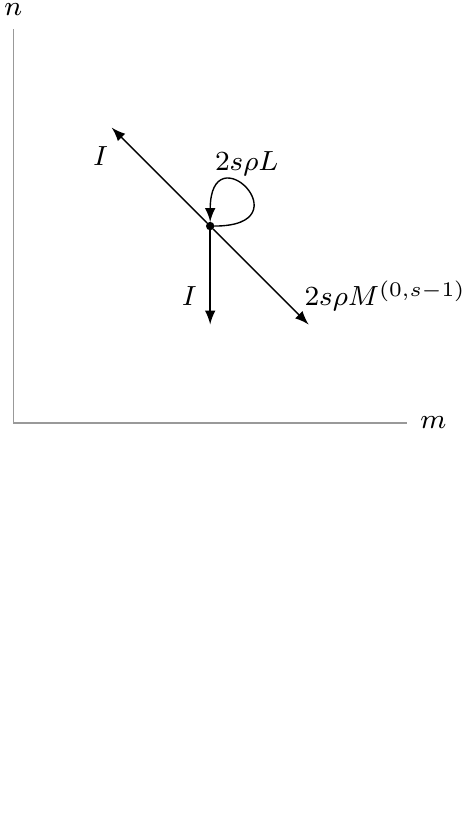}% PDF figure
}%
\caption{Simplified transition rate diagrams on the state space $(m,n)$ for JSQ systems with two servers.}%
\label{fig:JSQ_various_models}%
\end{figure}%

For the symmetric JSQ model in \cite{CompApproachSymmetric_Adan1990}, the initial solution is of the form $\ch_0 \al_0^m \be_0^n$, where $\ch_0$ is a coefficient and $\al_0$ and $\be_0$ are the compensation parameters that satisfy the kernel equation
\begin{equation}%
\al \be 2(\rho + 1) = \be^2 2\rho + \al \be^2 + \al^2, \label{eqn:kernel_equation_JSQ_identical}
\end{equation}%
which is obtained by substituting $\al^m \be^n$ in the equilibrium equations of the interior of the positive quadrant and dividing by common powers. This initial solution satisfies the equilibrium equations in the interior and on the HB (there is only one such solution). In order to compensate for the error on the VB, one adds the compensation term $\cv_0 \al_1^m \be_0^n$ such that $\ch_0 \al_0^m \be_0^n + \cv_0 \al_1^m \be_0^n$ satisfies the equilibrium equations in the interior and on the VB. The compensation parameter $\al_1$ with $\al_1 < \be_0$ is generated from \eqref{eqn:kernel_equation_JSQ_identical} for a fixed $\be = \be_0$. The coefficient $\cv_0$ satisfies a linear equation and is a function of $\ch_0$, $\al_0$, $\be_0$, and $\al_1$. The resulting solution violates the equilibrium equations on the HB. Hence, one adds another compensation term $\ch_1 \al_1^m \be_1^n$ such that $\cv_0 \al_1^m \be_0^n + \ch_1 \al_1^m \be_1^n$ satisfies the equilibrium equations in the interior and on the HB, where $\be_1$ and $\ch_1$ are determined in a similar way as for the VCS. Repeating the compensation steps leads to a series expression for the equilibrium probabilities that satisfies all equilibrium equations:
\begin{equation}%
p(m,n) = C \sum_{l = 0}^\infty \ch_l \al_l^m \be_l^n + C \sum_{l = 0}^\infty \cv_{l + 1} \al_{l + 1}^m \be_l^n, \quad m \ge 0, ~ n \ge 1,
\end{equation}%
where $C$ is the normalization constant. Figure~\ref{fig:alpha_beta_JSQ_identical} displays the way in which the compensation parameters are generated.

\begin{figure}%
\centering%
\includegraphics{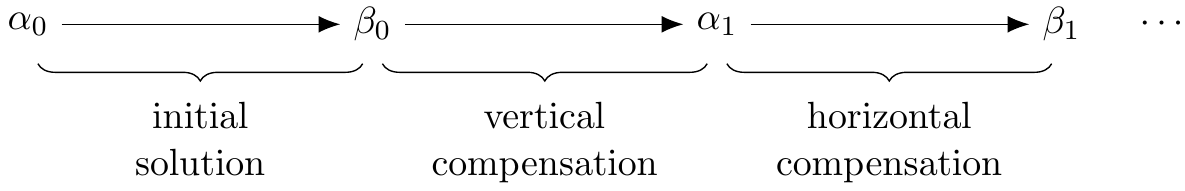}% PDF figure
\caption{For a system with two identical servers and JSQ routing, the compensation approach generates in each compensation step a single compensation term.}%
\label{fig:alpha_beta_JSQ_identical}%
\end{figure}%

The first extension is presented in \cite{CompApproachAsymmetric_Adan1991} for the asymmetric JSQ model, i.e.~the servers are now assumed to be non-identical with speeds 1 and $s$. The symmetry argument used earlier does not hold anymore and one needs to consider the complete half-plane, see Figure~\ref{fig:JSQ_various_models}(b) for the transition rate diagram. Note that the half-plane consists of two quadrants with different transition rates that are coupled on the horizontal axis. The approach in this case is an extension of the approach introduced in \cite{CompApproachSymmetric_Adan1990}. In a VCS, one compensates solutions that satisfy the positive inner equations on the positive VB as well as solutions that satisfy the negative inner equations on the negative VB. Two kernel equations (one for each quadrant) are used to generate the $\al$'s, and the coefficients satisfy different linear equations. For a HCS, each product-form solution that satisfies the positive inner equations, generates a single $\be$ for the positive quadrant and a single $\be$ for the negative quadrant. Accordingly, a product-form solution that satisfies the negative inner equations is compensated on the HB. Thus, in this case the generation of compensation parameters has a \textit{binary tree structure}, see Figure~\ref{fig:alpha_beta_JSQ_non-identical}.

\begin{figure}%
\centering%
\includegraphics{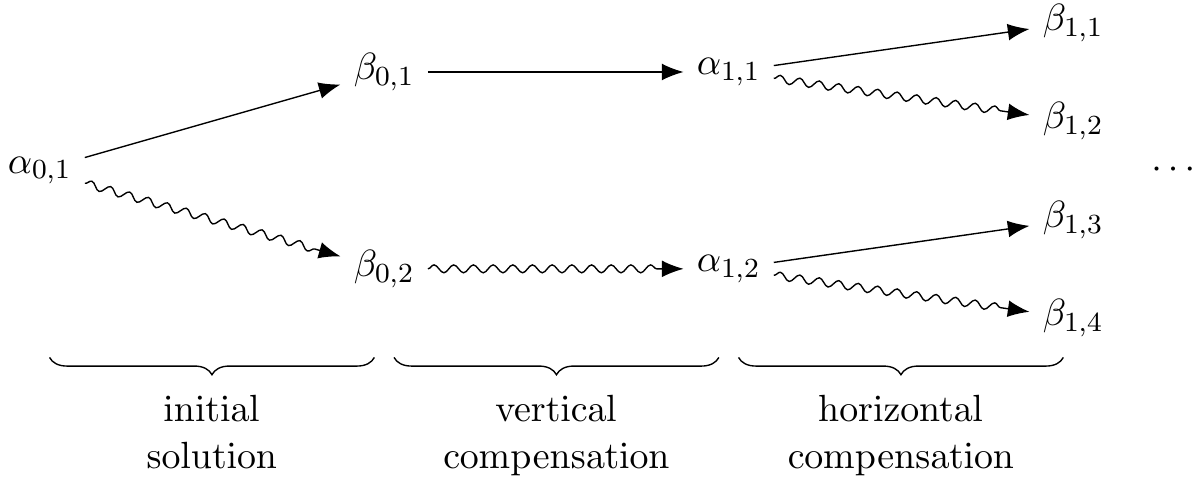}% PDF figure
\caption{For the asymmetric JSQ model, the compensation approach generates different compensation terms for the positive (straight arrow) and the negative quadrant (snaked arrow).}%
\label{fig:alpha_beta_JSQ_non-identical}%
\end{figure}%

A further extension of the compensation approach is presented in \cite{CompApproachErlangArrivals_Adan2013} for a model with two identical servers, Erlang-$s$ arrivals and JSQ routing. The state description is enhanced by adding a finite third dimension that keeps track of the number of completed arrival phases. The random walk in the positive and negative quadrant are mirror images, which permits to perform the analysis only on the positive quadrant, see Figure~\ref{fig:JSQ_various_models}(c) for the transition rate diagram. In \cite{CompApproachErlangArrivals_Adan2013} the authors extend the compensation approach to a three-dimensional setting. Due to the three-dimensional state space, each compensation term takes the form $\al^m \be^n \ibpos{\al,\be}$, where $\ibpos{\al,\be}$ is a vector of coefficients of dimension $s$ (equal to the number of arrival phases). In each HCS, $s$ different parameters $\be$ are generated instead of just one. A graphical representation of the generation of compensation terms, which has an $s$\textit{-fold tree structure}, is depicted in Figure~\ref{fig:alpha_beta_JSQ_Erlang_arrivals}.

\begin{figure}%
\centering%
\includegraphics{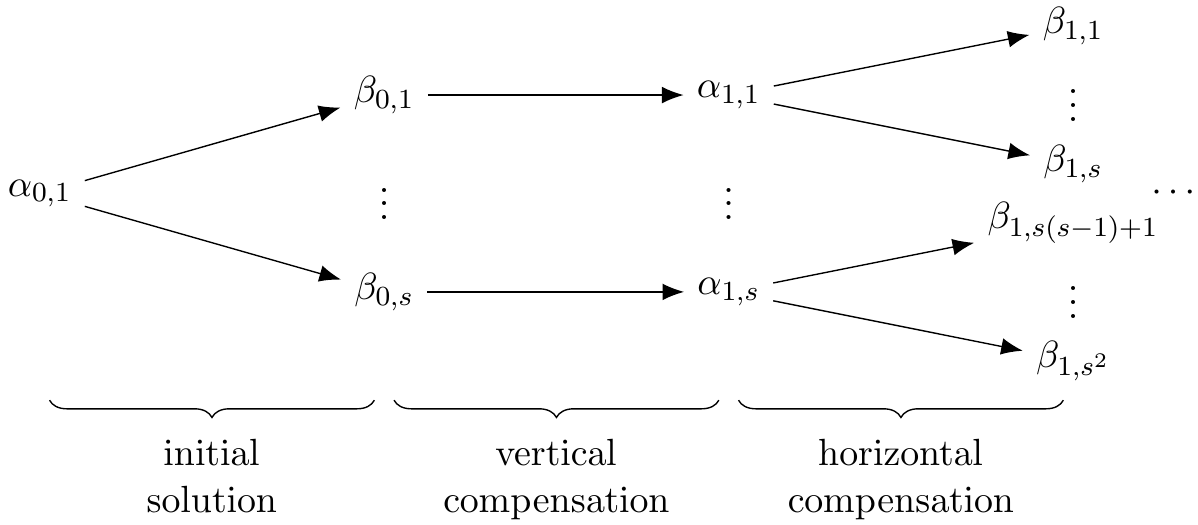}% PDF figure
\caption{For the symmetric JSQ model with Erlang-$s$ arrivals, the compensation approach generates $s$ compensation terms in a HCS and just one compensation term in a VCS.}%
\label{fig:alpha_beta_JSQ_Erlang_arrivals}%
\end{figure}%

\textbf{Our contribution.} The model at hand is defined on an $s$-layered half plane, thus requires that we further extend the compensation approach. Similarly to \cite{CompApproachAsymmetric_Adan1991}, we need to account for the two quadrants by considering two kernel functions (one for each quadrant). Furthermore, in accordance with \cite{CompApproachErlangArrivals_Adan2013}, in every HCS, a total of $s$ different parameters $\be$ are generated for the positive quadrant and a single $\be$ for the negative quadrant. This leads to a $(s + 1)$\textit{-fold tree structure} for the compensation parameters as depicted in Figure~\ref{fig:alpha_beta_SEDR}. Additionally, the product-form solutions take the form $\al^m\be^n\ibpos{\al,\be}$ or $\al^m\be^{-n}\ibneg{\al,\be}$ depending on whether they are defined in the positive or negative quadrant, respectively. The resulting solution for the equilibrium distribution is, for $m \ge 0, ~ n \ge 1$,
\begin{subequations}%
\label{eqn:equilibrium_distribution_introduction}
\begin{align}%
\pb{m,n} &= C \sum_{l = 0}^\infty \sum_{i = 1}^{(s + 1)^l} \sum_{j = 1}^s \ch_{l,d(i) + j} \al_{l,i}^m \be_{l,d(i) + j}^n \ibpos{\al_{l,i},\be_{l,d(i) + j}} \notag \\
&\quad + C \sum_{l = 0}^\infty \sum_{i = 1}^{(s + 1)^l} \sum_{j = 1}^s \cv_{l + 1,d(i) + j} \al_{l + 1,d(i) + j}^m \be_{l,d(i) + j}^n \ibpos{\al_{l + 1,i},\be_{l,d(i) + j}}. \label{eqn:equilibrium_distribution_introduction_pos}
\intertext{For $m \ge 0, ~ n = 0$,}
\pb{m,n} &= C \sum_{l = 0}^\infty \sum_{i = 1}^{(s + 1)^l} \al_{l,i}^m \hb_{l,i}. \label{eqn:equilibrium_distribution_introduction_hor}
\intertext{For $m \ge 0, ~ n \le -1$,}
\pb{m,n} &= C \sum_{l = 0}^\infty \sum_{i = 1}^{(s + 1)^l} \ch_{l,i(s + 1)} \al_{l,i}^m \be_{l,i(s + 1)}^{-n} \ibneg{\al_{l,i},\be_{l,i(s + 1)}} \notag \\
&\quad + C \sum_{l = 0}^\infty \sum_{i = 1}^{(s + 1)^l} \cv_{l + 1,i(s + 1)} \al_{l + 1,i(s + 1)}^m \be_{l,i(s + 1)}^{-n} \ibneg{\al_{l + 1,i(s + 1)},\be_{l,i(s + 1)}}, \label{eqn:equilibrium_distribution_introduction_neg}
\end{align}%
\end{subequations}%
where $C$ is the normalization constant and $d(i) \defi (i - 1)(s + 1)$. The first subscript $l$ is the level at which a parameter resides, starting at level $l = 0$ (the initial solution). Within a level, the parameters are differentiated by using an additional index $i$. A horizontal compensation step and the initial solution has coefficients $\ch$ and a vertical compensation step has coefficients $\cv$. Additionally, a vector $\hb$ is generated in each horizontal compensation step. The initial solution is described in Lemma~\ref{lem:initial_solution}, the horizontal compensation step is described in Lemma~\ref{lem:compensation_horizontal}, and the vertical compensation step is described in Lemma~\ref{lem:compensation_vertical}, see Figure~\ref{fig:alpha_beta_SEDR}.

\begin{figure}%
\centering%
\includegraphics{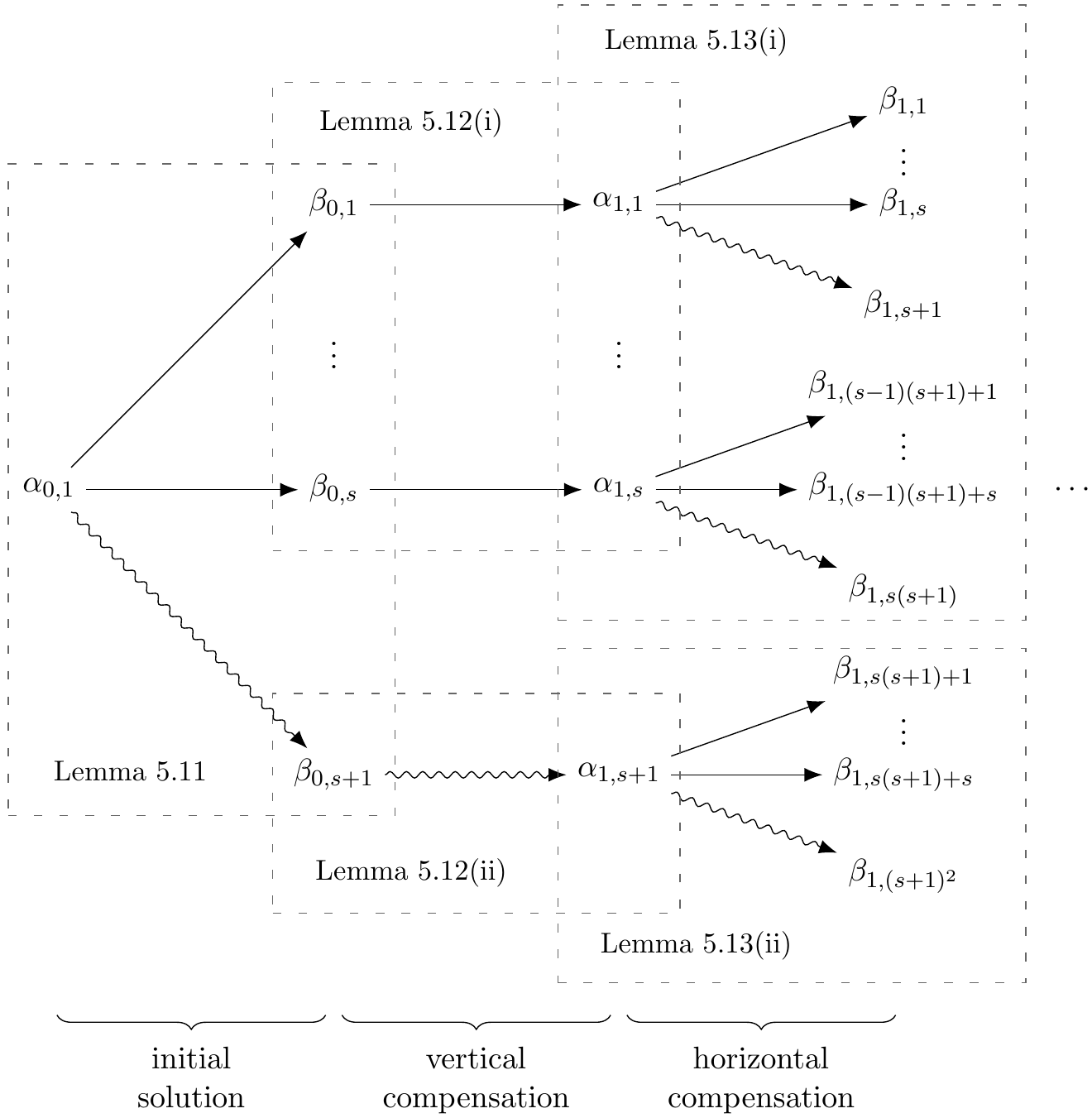}% PDF figure
\caption{For the SED system, the compensation approach generates different compensation terms for the positive (straight arrow) or the negative quadrant (snaked arrow). The number of terms generated in the positive and negative quadrant are different.}%
\label{fig:alpha_beta_SEDR}%
\end{figure}%

%%%%%%%%%%%%%%%%%%%%%%%%%%%%%%%%%%%%%%%%%%%%%%%%%%%%%%%
%%%%%%%%%%%%%%%%%%%%%%%%%%%%%%%%%%%%%%%%%%%%%%%%%%%%%%%
%%%%%%%%%%%%%%%%%%%%% NEW SECTION %%%%%%%%%%%%%%%%%%%%%
%%%%%%%%%%%%%%%%%%%%%%%%%%%%%%%%%%%%%%%%%%%%%%%%%%%%%%%
%%%%%%%%%%%%%%%%%%%%%%%%%%%%%%%%%%%%%%%%%%%%%%%%%%%%%%%

\section{Numerical results}%
\label{sec:numerical_results}%

Expression \eqref{eqn:equilibrium_distribution_introduction} is amenable for numerical calculations after applying truncation. For $m \ge 0, ~ n \ge 1$,
\begin{subequations}%
\label{eqn:equilibrium_distribution_numerical}%
\begin{align}%
\bld{p}_L(m,n) &= C \sum_{l = 0}^{\lfloor \frac{L}{2} \rfloor} \sum_{i = 1}^{(s + 1)^l} \sum_{j = 1}^s \ch_{l,d(i) + j} \al_{l,i}^m \be_{l,d(i) + j}^n \ibpos{\al_{l,i},\be_{l,d(i) + j}} \notag \\
&\quad + C \sum_{l = 0}^{\lfloor \frac{L-1}{2} \rfloor} \sum_{i = 1}^{(s + 1)^l} \sum_{j = 1}^s \cv_{l + 1,d(i) + j} \al_{l + 1,d(i) + j}^m \be_{l,d(i) + j}^n \ibpos{\al_{l + 1,i},\be_{l,d(i) + j}}. \label{eqn:equilibrium_distribution_numerical_pos}
\intertext{For $m \ge 0, ~ n = 0$,}
\bld{p}_L(m,n) &= C \sum_{l = 0}^{\lfloor \frac{L}{2} \rfloor} \sum_{i = 1}^{(s + 1)^l} \al_{l,i}^m \hb_{l,i}. \label{eqn:equilibrium_distribution_numerical_hor}
\intertext{For $m \ge 0, ~ n \le -1$,}
\bld{p}_L(m,n) &= C \sum_{l = 0}^{\lfloor \frac{L}{2} \rfloor} \sum_{i = 1}^{(s + 1)^l} \ch_{l,i(s + 1)} \al_{l,i}^m \be_{l,i(s + 1)}^{-n} \ibneg{\al_{l,i},\be_{l,i(s + 1)}} \notag \\
&\quad + C \sum_{l = 0}^{\lfloor \frac{L-1}{2} \rfloor} \sum_{i = 1}^{(s + 1)^l} \cv_{l + 1,i(s + 1)} \al_{l + 1,i(s + 1)}^m \be_{l,i(s + 1)}^{-n} \ibneg{\al_{l + 1,i(s + 1)},\be_{l,i(s + 1)}}, \label{eqn:equilibrium_distribution_numerical_neg}
\end{align}%
\end{subequations}%
where the empty sum $\sum_{l = 0}^{-1}$ is 0. Here, $L = 0$ indicates only the initial solution and for instance $L = 3$ indicates an initial solution, a vertical, horizontal and another vertical compensation. Naturally, as $L$ increases, the approximation becomes more accurate.

We perform several numerical experiments that verify that under SED routing the joint queue length process concentrates between the line where the expected delays in both queues are equal $q_1 + 1 = (q_2 + 1)/s$ and the line where the expected waiting time is equal $q_1 = q_2/s$ using the equilibrium distribution. To this end, we consider a system with service rates 1 and $s = 3$, $q = 0.4$ and set $L = 16$ in \eqref{eqn:equilibrium_distribution_numerical}. The equilibrium distribution for this model and varying $\rho$ is given in Figure~\ref{fig:equilibrium_distribution} in the form of a heat plot, supporting our claim.

\begin{figure}%
\centering%
\subfloat[$\rho = 0.6$]{%
\includegraphics{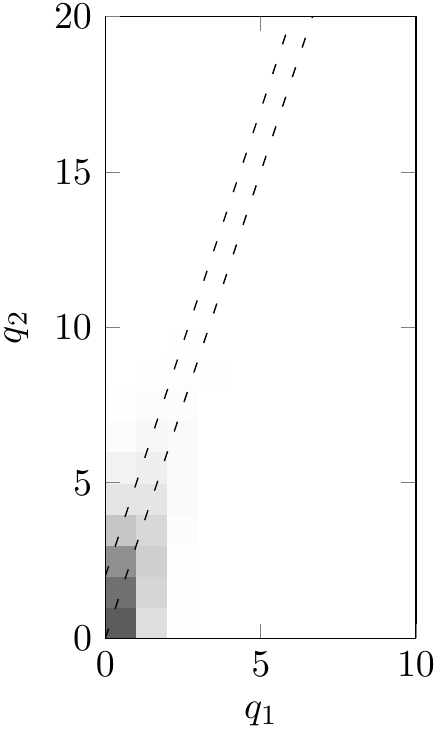}% PDF figure
}%
\subfloat[$\rho = 0.75$]{%
\includegraphics{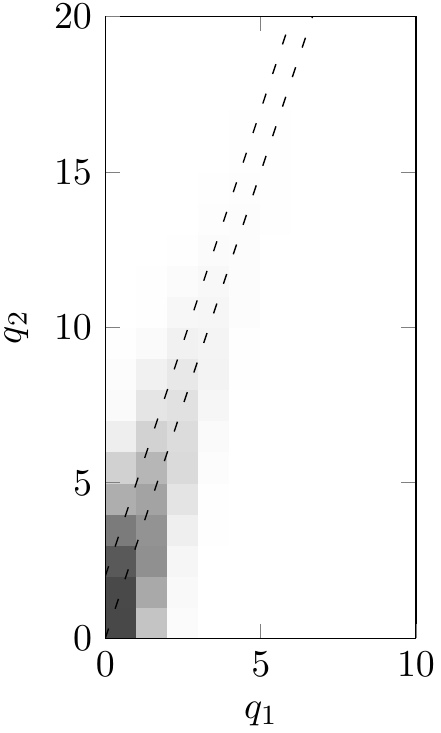}% PDF figure
}%
\subfloat[$\rho = 0.9$]{%
\includegraphics{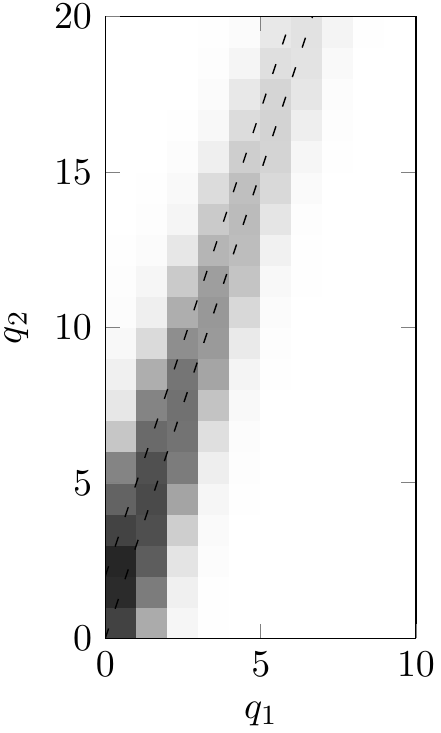}% PDF figure
}%
\caption{Heat plot of the equilibrium distribution mass for the SED system with $s = 3$, $q = 0.4$, and varying $\rho$, determined according to \eqref{eqn:equilibrium_distribution_numerical} with $L = 16$. The heat plot shows where the probability mass is located (darker colour means more mass). The dashed lines are $q_1 + 1 = (q_2 + 1)/s$ and $q_1 = q_2/s$.}%
\label{fig:equilibrium_distribution}%
\end{figure}%

Next, to demonstrate the rate of convergence of the series in \eqref{eqn:equilibrium_distribution_introduction}, we derive the number of compensation steps $L$ for which the equilibrium probabilities $\pb{m,n}$ are considered sufficiently accurate: As a measure of accuracy we compute for each state $(m,n)$ the minimum number of compensation steps $L$ such that
\begin{equation}%
\max_{r = 0,1,\ldots,s - 1} \frac{| p_L(m,n,r) - p_{L - 1}(m,n,r)|}{p_{L - 1}(m,n,r)} < 10^{-4}. \label{eqn:accuracy_condition_numerical}
\end{equation}%
Figure~\ref{fig:no_compensation_steps_needed} shows that away from the origin, the convergence of the series is very fast, but the convergence is also quite fast for states close to the origin. Note that the distance of a state $(m,n)$ to the origin is directly related to the rate of convergence of the series expression \eqref{eqn:equilibrium_distribution_introduction}. In particular, it seems to be a function of $m + |n|$: faster convergence further away from the origin. This property is formally proven in Section~\ref{subsec:absolute_convergence} and can be exploited for numerical computations.

\begin{figure}%
\centering%
\includegraphics{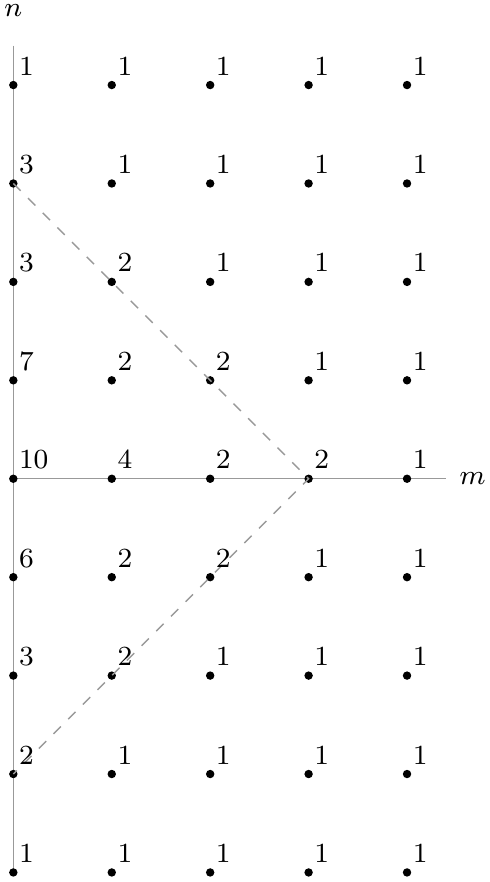}% PDF figure
\caption{The number of compensation steps $L$ for each state $(m,n)$ such that the resulting $\bld{p}_L(m,n)$ is accurate according to \eqref{eqn:accuracy_condition_numerical}. The dashed line is $m + |n| = 3$. Parameters are $s = 4$, $\rho = 0.8$, and $q = 0.4$.}%
\label{fig:no_compensation_steps_needed}%
\end{figure}%

\begin{remark}[No curse of dimensionality]\label{rem:no_curse_of_dimensionality}%
Take a triangular set of states $\mathcal{T}_{M} = \{ (m,n) \mid m \in \Nat_0, ~ n \in \Int, ~ m + |n| \le M \}$, where $M$ is some non-negative integer. Technically, $M$ needs to be strictly larger than some non-negative integer $N$, but we do not go into the details here; the lower bound $N$ is described in Section~\ref{subsec:absolute_convergence}. For states outside $\mathcal{T}_{M}$, we can use \eqref{eqn:equilibrium_distribution_numerical} to compute $\pb{m,n}$. Since the number of compensation steps $L$ required to achieve accurate results according to \eqref{eqn:accuracy_condition_numerical} decreases with $m + |n|$, the number of compensation steps $L$ for each state $(m,n) \notin \mathcal{T}_{M}$ is relatively small. For states $(m,n) \in \mathcal{T}_{M}$, a linear system of equilibrium equations needs to be solved to determine the equilibrium probabilities $\pb{m,n}$, where one uses that $\pb{m,n}, ~ (m,n) \notin \mathcal{T}_{M}$ are known.

As an example, for $s = 4$, $\rho = 0.8$ and $q = 0.4$ the choice $M = 3$ implies that $L = 1$, which gives accurate results for $\pb{m,n}, ~ (m,n) \notin \mathcal{T}_{M}$, see Figure~\ref{fig:no_compensation_steps_needed}. So it is evident that the compensation approach does not suffer from the curse of dimensionality.
\end{remark}%

%%%%%%%%%%%%%%%%%%%%%%%%%%%%%%%%%%%%%%%%%%%%%%%%%%%%%%%
%%%%%%%%%%%%%%%%%%%%%%%%%%%%%%%%%%%%%%%%%%%%%%%%%%%%%%%
%%%%%%%%%%%%%%%%%%%%% NEW SECTION %%%%%%%%%%%%%%%%%%%%%
%%%%%%%%%%%%%%%%%%%%%%%%%%%%%%%%%%%%%%%%%%%%%%%%%%%%%%%
%%%%%%%%%%%%%%%%%%%%%%%%%%%%%%%%%%%%%%%%%%%%%%%%%%%%%%%

\section{Applying the compensation approach}%
\label{sec:applying_the_compensation_approach}%
%

%%%%%%%%%%%%%%%%%%%%%%%%%%%%%%%%%%%%%%%%%%%%%%%%%%%%%%%
%%%%%%%%%%%%%%%%%%%%%%%%%%%%%%%%%%%%%%%%%%%%%%%%%%%%%%%
%%%%%%%%%%%%%%%%%%% NEW SUBSECTION %%%%%%%%%%%%%%%%%%%%
%%%%%%%%%%%%%%%%%%%%%%%%%%%%%%%%%%%%%%%%%%%%%%%%%%%%%%%
%%%%%%%%%%%%%%%%%%%%%%%%%%%%%%%%%%%%%%%%%%%%%%%%%%%%%%%

\subsection{Outline}%
\label{subsec:outling_compensation_approach}%

In the following subsections we describe the main steps in constructing the equilibrium distribution of the SED system. First, we develop some preliminary results in Section~\ref{subsec:preliminary_results}, showing that the inner equations have a product-form solution and we determine one of the two parameters of the product-form solution explicitly. Using these preliminary results, we determine the unique initial solution in Section~\ref{subsec:initial_solution}. The vertical compensation step is outlined in Section~\ref{subsec:compensation_vertical_boundary}. Section~\ref{subsec:compensation_horizontal_boundary} describes the horizontal compensation procedure. We formalize the resulting solution in terms of a sequence of product-forms in Section~\ref{subsec:constructing_equilibrium_distribution}. This sequence of compensation terms grows as a $(s + 1)$-fold tree and the problem is in showing that the sequence converges. Section~\ref{subsec:absolute_convergence} is devoted to the issue of convergence.

%%%%%%%%%%%%%%%%%%%%%%%%%%%%%%%%%%%%%%%%%%%%%%%%%%%%%%%
%%%%%%%%%%%%%%%%%%%%%%%%%%%%%%%%%%%%%%%%%%%%%%%%%%%%%%%
%%%%%%%%%%%%%%%%%%% NEW SUBSECTION %%%%%%%%%%%%%%%%%%%%
%%%%%%%%%%%%%%%%%%%%%%%%%%%%%%%%%%%%%%%%%%%%%%%%%%%%%%%
%%%%%%%%%%%%%%%%%%%%%%%%%%%%%%%%%%%%%%%%%%%%%%%%%%%%%%%

\subsection{Preliminary results}%
\label{subsec:preliminary_results}%

We conjecture that the inner equations have a product-form solution. To this end, we examine a related model that has the same behavior in the interior as the original model. We then show that the equilibrium distribution of this related model can be expressed as a product-form solution. Moreover, modeling this process as a quasi-birth--and--death (QBD) queue, we obtain a closed form expression for one of the parameters of the product-form solution. This procedure is closely related to the one in \cite[Section~4]{CompApproachErlangArrivals_Adan2013}.

The related model is constructed as follows. We start from the state space of the original model in Figure~\ref{fig:state_space_trd}(a) and bend the vertical axis as shown in Figure~\ref{fig:trd_modified_model}. Note that for this modified model, $m \in \Int$. We add an additional state $(-1,0,s-1)$ with a transition rate $\hat{A}_{0,1} = s \eb{0}$ to state $(-1,1)$ and a transition rate $\hat{B}_{0,-1} = (1 + s)\rho q \eb{s-1}$ to state $(-1,-1)$, where $\eb{i}$ is a column vector of zeros of length $s$ with a one at position $i$. The transitions from the states on the diagonal $m + n = 0, ~ n \ge 1$ are kept consistent with the transitions from a state in the positive interior of the SED model, where the downward transitions are redirected to the state $(-1,0,s - 1)$, specifically, $\hat{A}_{0,-1} = s \eb{0}^T$. Similarly, the transitions from the states on the diagonal $m - n = 0, ~ n \le -1$ are kept consistent with the ones in the negative interior of the SED model, where the upward transitions are redirected to the state $(-1,0,s-1)$, specifically, $\hat{B}_{0,1} = \oneb^T$, where $\oneb$ is a column vector of ones of size $s$.

\begin{figure}%
\centering%
\includegraphics{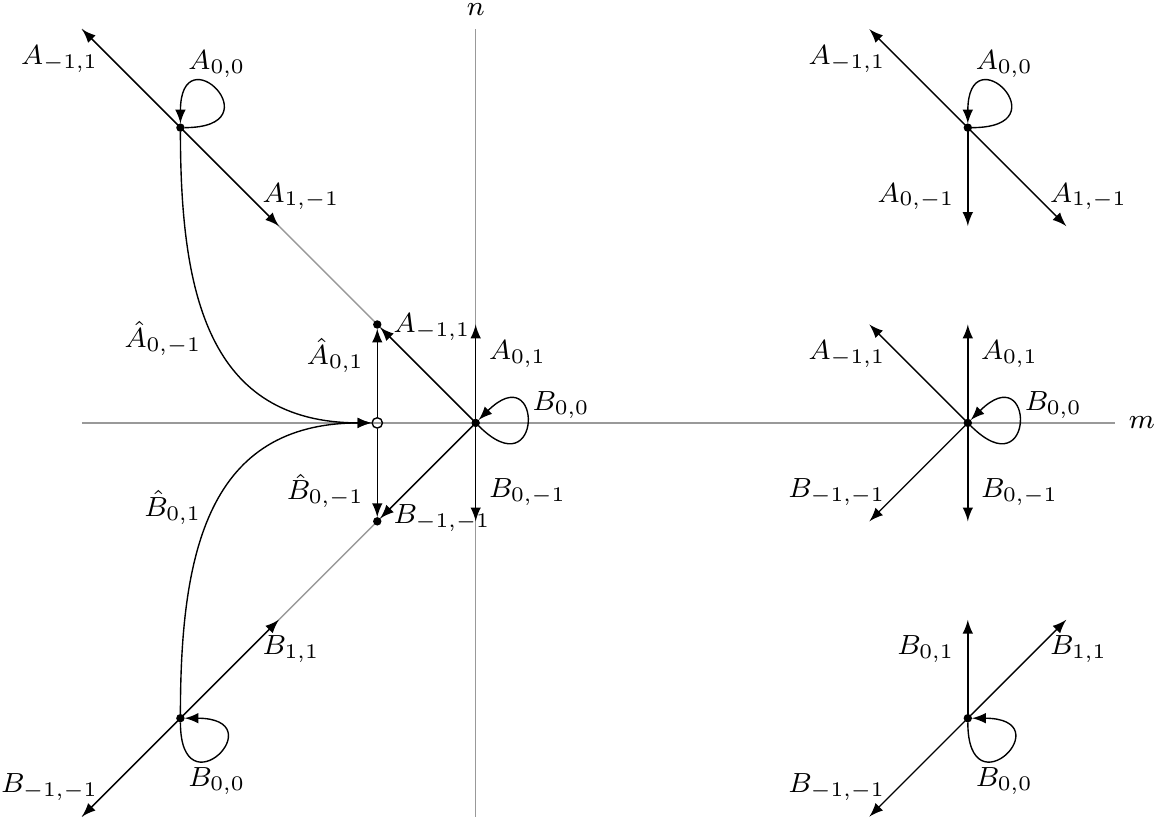}% PDF figure
\caption{$s$-layered transition rate diagram of the modified model. Note that the third dimension, i.e.~$r$, is perpendicular to the page. The non-filled dot, graphed at position $(-1,0)$, corresponds to the single state $(-1,0,s-1)$.}%
\label{fig:trd_modified_model}%
\end{figure}%

The characteristic feature of the modified model is that its equilibrium equations for $m + |n| = 0, ~ |n| \ge 2$ are exactly the same as the ones in the interior and the equilibrium equations for $n \in \{-1,0,1\}$ are exactly the same as the ones on the horizontal boundary of the original model. In this sense, the modified model has no ``vertical boundary'' equations.

We next present two lemmas. The first lemma states that a product-form solution exists for the modified model, while in the second lemma we identify the geometric term of the product form expression. Let $\pbmodi{m,n} = (\pmodi{m,n,0},\pmodi{m,n,1},\ldots,\pmodi{m,n,s-1})^T$ denote the equilibrium distribution of the modified model.

\begin{lemma} \label{lem:modified_model_product_form}%
For $\rho < 1$, the equilibrium distribution of the modified model exists and is of the form
\begin{equation}%
\pbmodi{m,n} = \al^m \qbmodi{n}, \quad m + |n| \ge 0, ~ n \in \Int \label{eqn:modified_model_product_form}
\end{equation}%
with $\al \in (0,1)$ and $\qbmodi{n} = (\qmodi{n,0},\qmodi{n,1},\ldots,\qmodi{n,s - 1})^T$ such that
\begin{equation}%
\sum_{n = -\infty}^\infty \al^{-|n|} \qmodi{n,r} < \infty, \quad r = 0,1,\ldots,s - 1.
\end{equation}%
\end{lemma}%

\begin{proof}%
First notice that the modified model is stable whenever the original model is stable and that $\pbmodi{m,n}$ satisfies the inner and horizontal boundary equations \eqref{eqn:equilibrium_eqs_I+}-\eqref{eqn:equilibrium_eqs_H} for all $m + |n| = 0$, except for state $(0,0)$. Observe that the modified model restricted to the area $\{ (m,n) \mid m \in \Nat_0, ~ n \in \Int, ~ m + |n| \ge n_0 \} \cup \{(n_0 - 1,0,s - 1)\}, ~ n_0 = 1,2,\ldots$ embarked by two lines parallel to the diagonal axes, yields the exact same process. Hence, we can conclude that
\begin{equation}%
\pbmodi{m + 1,n} = \al \pbmodi{m,n}, \quad m + |n| \ge 0, ~ n \in \Int
\end{equation}%
and therefore
\begin{equation}%
\pmodi{m,n,r} = \al^m \qmodi{n,r}, \quad m + |n| \ge 0, ~ n \in \Int, ~ r = 0,1,\ldots,s - 1.
\end{equation}%
Finally, we observe that
\begin{equation}%
\sum_{n = -\infty}^\infty \al^{-|n|} \qmodi{n,r} = \sum_{n = -\infty}^\infty \pmodi{-|n|,n,r} < 1,
\end{equation}%
which concludes the proof.
\end{proof}%

In Lemma~\ref{lem:modified_model_product_form} we have shown that the equilibrium distribution of the modified model has a product form which is unique up to a positive multiplicative constant. In the next lemma we determine the unique $\al$ in \eqref{eqn:modified_model_product_form}. More concretely, the unique $\al$ is equal to $\rho^{1 + s}$.

\begin{lemma}\label{lem:modified_model_specific_alpha}%
For $\rho < 1$, the equilibrium distribution of the modified model is of the form
\begin{equation}%
\pbmodi{m,n} = \rho^{(1+s)m} \qbmodi{n}, \quad m + |n| \ge 0, ~ n \in \Int.
\end{equation}%
\end{lemma}%

\begin{proof}%
Let $k$ denote the total number of customers in the system, i.e.
\begin{equation}%
k = \begin{cases}%
(1 + s)m + sn + r, & n \ge 0, \\
(1 + s)m - n + r, & n < 0.
\end{cases}%
\end{equation}%
Then, the number of customers in the system forms a Markov process and for all states $k > s$, the transitions are given by: (i) from state $k$ to state $k + 1$ with rate $(1 + s)\rho$; and (ii) from state $k + 1$ to state $k$ with rate $1 + s$. Let $\hat{p}_k$ denote the probability of having $k$ customers in the system. From the balance principle between states $k$ and $k + 1$ we obtain
\begin{equation}%
\hat{p}_{k + 1} = \rho \hat{p}_k, \quad k > s.
\end{equation}%
Furthermore,
\begin{align}%
\hat{p}_{k + s + 1} &= \sum_{\substack{(m,n,r) \\ (1 + s)m + sn + r = k + s + 1}} \pmodi{m,n,r} + \sum_{\substack{(m,n,r) \\ (1 + s)m - n + r = k + s + 1}} \pmodi{m,n,r} \notag \\
&= \sum_{\substack{(m,n,r) \\ (1 + r)(m - 1) + sn + r = k}} \al^m \qmodi{n,h} + \sum_{\substack{(m,n,r) \\ (1 + s)(m - 1) - n + r = k}} \al^m \qmodi{n,r} \notag \\
&= \al \sum_{\substack{(l,n,r) \\ (1 + s)l + sn + r = k}} \al^l \qmodi{n,h} + \al \sum_{\substack{(l,n,r) \\ (1 + s)l - n + r = k}} \al^l \qmodi{n,r} \notag \\
&= \al \hat{p}_k.
\end{align}%
Combining the last two results immediately yields $\al = \rho^{1 + s}$.
\end{proof}%

Combining Lemmas~\ref{lem:modified_model_product_form} and \ref{lem:modified_model_specific_alpha} yields the following result for the original model.

\begin{proposition}\label{prop:initial_product_form}%
For $\rho < 1$, the inner and horizontal boundary equations \eqref{eqn:equilibrium_eqs_I+}-\eqref{eqn:equilibrium_eqs_H} have a unique solution of the form
\begin{equation}%
\pb{m,n} = \rho^{(1 + s)m} \qb{n}, \quad m \ge 0, ~ n \in \Int \label{eqn:initial_product_form}
\end{equation}%
with $\qb{n} = (\q{n,0},\q{n,1},\ldots,\q{n,s - 1})^T$ non-zero and
\begin{equation}%
\sum_{n = -\infty}^\infty \rho^{-(1 + s)|n|} \q{n,r} < \infty, ~ r = 0,1,\ldots,s - 1. \label{eqn:initial_product_form_condition}
\end{equation}%
\end{proposition}%

The solution obtained in Proposition~\ref{prop:initial_product_form} satisfies the inner and horizontal boundary equations. However, we still need to specify the form of the vector $\qb{n}$. It will become apparent, from Lemmas~\ref{lem:solution_inner_equations}, \ref{lem:location_number_zeros_inner_pos} and \ref{lem:location_number_zeros_inner_neg}, that the form of the vector $\qb{n}$ is entirely different in the positive quadrant, on the horizontal axis, and the negative quadrant. Correctly identifying $\qb{n}$ will result in the initial solution that satisfies the equilibrium equations of the interior and the horizontal axis. In the following lemmas we describe the form of a solution satisfying the inner equations.

\begin{lemma}\label{lem:solution_inner_equations}\hspace*{1em}%
\begin{enumerate}[label = \textup{(\roman*)}]%
\item The product form $\pb{m,n} = \al^m \be^n \ibpos{\al,\be}, ~ m \ge 0, ~ n \ge 1$ is a solution of the inner equations of the positive quadrant \eqref{eqn:equilibrium_eqs_I+} if
    \begin{equation}%
    \Dpos{\al,\be} \ibpos{\al,\be} = \zerob \label{eqn:solution_inner_pos_should_satisfy}
    \end{equation}%
    with
    \begin{equation}%
    \Dpos{\al,\be} = \al \be \A{0,0} + \al \be^2 \A{0,-1} + \al^2 \A{-1,1} + \be^2 \A{1,-1}
    \end{equation}%
    and the eigenvector $\ibpos{\al,\be} = (\ipos{\al,\be,0},\ipos{\al,\be,1},\ldots,\ipos{\al,\be,s - 1})^T$ satisfies
    \begin{equation}%
    \frac{\ipos{\al,\be,r}}{\ipos{\al,\be,0}} = \Bigl( \frac{\al \be (1 + s)(\rho + 1) - \be^2 (1 + s) \rho - \al^2}{\al \be s} \Bigr)^r, \quad r = 0,1,\ldots,s - 1. \label{eqn:inner_pos_eigenvector}
    \end{equation}%
\item The product form $\pb{m,n} = \al^m \be^{-n} \ibneg{\al,\be}, ~ m \ge 0, ~ n \le -1$ is a solution of the inner equations of the negative quadrant \eqref{eqn:equilibrium_eqs_I-} if
    \begin{equation}%
    \Dneg{\al,\be} \ibneg{\al,\be} = \zerob \label{eqn:solution_inner_neg_should_satisfy}
    \end{equation}
    with
    \begin{equation}%
    \Dneg{\al,\be} = \al \be \B{0,0} + \al \be^2 \B{0,1} + \al^2 \B{-1,-1} + \be^2 \B{1,1}
    \end{equation}%
    and the eigenvector $\ibneg{\al,\be} = (\ineg{\al,\be,0},\ineg{\al,\be,1},\ldots,\ineg{\al,\be,s - 1})^T$ satisfies
    \begin{equation}%
    \hspace*{-1em}\frac{\ineg{\al,\be,r}}{\ineg{\al,\be,0}} =  \frac{\Func{\al,\be,\func{-}{\be}} \func{+}{\be}^r - \Func{\al,\be,\func{+}{\be}} \func{-}{\be}^r}{\Func{\al,\be,\func{-}{\be}} - \Func{\al,\be,\func{+}{\be}}}, \quad r = 0,1,\ldots,s-1 \label{eqn:inner_neg_eigenvector}
    \end{equation}%
    with
    \begin{equation}%
    \func{\pm}{\be} = \frac{(1 + s)(\rho + 1) - \be \pm \sqrt{(\be - (1 + s)(\rho + 1))^2 - 4s(1 + s)\rho}}{2s} \label{eqn:func_roots_eigenvectors}
    \end{equation}%
    and
    \begin{equation}%
    \hspace*{-0.8em}\Func{\al,\be,\funcparam} = \be - (1 + s)(\rho + 1) + s \funcparam + \frac{\be}{\al}(1 + s)\rho \funcparam^{s - 1} = (1 + s) \rho \funcparam^{-1} \bigl( \frac{\be}{\al} \funcparam^s - 1 \bigr). \label{eqn:Func_roots_eigenvectors}
    \end{equation}%
\end{enumerate}%
\end{lemma}%

\begin{proof}%
(i) Inserting the product form $\pb{m,n} = \al^m \be^n \ibpos{\al,\be}$ into \eqref{eqn:equilibrium_eqs_I+} and dividing by $\al^{m - 1}\be^{n - 1}$ results in \eqref{eqn:solution_inner_pos_should_satisfy}. For $\det(\Dpos{\al,\be}) = 0$, the rank of $\Dpos{\al,\be}$ is $s - 1$ and thus we have one free variable, which allows us to express $\ipos{\al,\be,r}$ in terms of $\ipos{\al,\be,0}$. The eigenvector $\ibpos{\al,\be}$ follows from the system of linear equations \eqref{eqn:solution_inner_pos_should_satisfy}, which reads
\begin{equation}%
\bigl( \al^2 + \be^2(1 + s)\rho - \al \be(1 + s)(\rho + 1) \bigr) \ipos{\al,\be,r} + \al \be s \ipos{\al,\be,r + 1} = 0, \quad r = 0,\ldots,s - 2
\end{equation}
and gives \eqref{eqn:inner_pos_eigenvector}.

\noindent (ii) Inserting the product form $\pb{m,n} = \al^m \be^{-n} \ibneg{\al,\be}$ into \eqref{eqn:equilibrium_eqs_I-} and dividing by $\al^{m-1}\be^{-n + 1}$ results in \eqref{eqn:solution_inner_neg_should_satisfy}. For $\det(\Dneg{\al,\be}) = 0$, the rank of $\Dneg{\al,\be}$ is $s-1$ and thus we have one free variable, which allows us to express $\ineg{\al,\be,r}$ in terms of $\ineg{\al,\be,0}$. The system of linear equations \eqref{eqn:solution_inner_neg_should_satisfy} reads
\begin{align}%
\bigl( \al \be^2 - \al \be(1 + s)(\rho + 1))& \ineg{\al,\be,0} + \al \be s \ineg{\al,\be,1} + \be^2 (1 + s)\rho \ineg{\al,\be,s - 1} = 0, \label{eqn:inner_neg_system_of_equations_boundary} \\
\al \be (1 + s)\rho \ineg{\al,\be,r - 1}& + \bigl( \al \be^2 - \al \be(1 + s)(\rho + 1)) \ineg{\al,\be,r} \notag \\
& + \al \be s \ineg{\al,\be,r+1} = 0, \quad r = 1,2,\ldots,s - 2. \label{eqn:inner_neg_system_of_equations_recurrence}
\end{align}%
We construct a solution of the form $\ineg{\al,\be,r} = (c_{\scriptscriptstyle +} \func{+}{\be}^r + c_{\scriptscriptstyle -} \func{-}{\be}^r)\ineg{\al,\be,0}$ and thus require $c_{\scriptscriptstyle +} + c_{\scriptscriptstyle -} = 1$. The two roots $\func{\pm}{\be}$ follow from substituting $\ineg{\al,\be,r} = \funcparam^r \ineg{\al,\be,0}$ in \eqref{eqn:inner_neg_system_of_equations_recurrence} and dividing by $\funcparam^{r - 1}\ineg{\al,\be,0}$, yielding
\begin{equation}%
s \funcparam^2 + (\be - (1 + s)(\rho + 1))\funcparam + (1 + s)\rho = 0. \label{eqn:inner_neg_funcparam_satisfies}
\end{equation}%
The constants $c_{\scriptscriptstyle +}$ and $c_{\scriptscriptstyle -}$ follow from \eqref{eqn:inner_neg_system_of_equations_boundary} and the simplification follows from the fact that $\ineg{\al,\be,r} = \funcparam^r \ineg{\al,\be,0}$ satisfies \eqref{eqn:inner_neg_funcparam_satisfies}.
\end{proof}%

\begin{remark}[Normalized eigenvectors]%
Without loss of generality we henceforth set the first elements $\ipos{\al,\be,0} = \ineg{\al,\be,0} = 1$ for all $\al$ and $\be$. Since normalization of the equilibrium distribution follows afterwards, one can arbitrarily select the value of the first element of the eigenvectors $\ibpos{\al,\be}$ and $\ibneg{\al,\be}$.
\end{remark}%

We wish to determine the $\al$'s and $\be$'s with $0 < |\al|,|\be| < 1$, for which \eqref{eqn:solution_inner_pos_should_satisfy} and \eqref{eqn:solution_inner_neg_should_satisfy} have non-zero solutions $\ibpos{\al,\be}$ and $\ibneg{\al,\be}$, respectively. Equivalently, we determine $\al$'s and $\be$'s for which $\det(\Dpos{\al,\be}) = 0$ or $\det(\Dneg{\al,\be}) = 0$. We first investigate the location and the number of zeros of $\det(\Dpos{\al,\be}) = 0$.

\begin{lemma}\label{lem:location_number_zeros_inner_pos}%
The equation $\det(\Dpos{\al,\be}) = 0$ assumes the form
\begin{equation}%
\bigl( \al\be(1 + s)(\rho + 1) - \be^2(1 + s)\rho - \al^2 \bigr)^s - \be (\al\be s)^s = 0 \label{eqn:determinant_inner_pos}
\end{equation}%
and can be rewritten to
\begin{equation}%
\frac{\al\be(1 + s)(\rho + 1) - \be^2(1 + s)\rho - \al^2}{\al \be s} = \unit_i \be^{1/s}, \quad i = 1,2,\ldots,s, \label{eqn:determinant_inner_pos_s-th_root}
\end{equation}%
where $\unit_i$ is the $i$-th root of unity of $\unit^s = 1$ and $\be^{1/s}$ is the principal root.
\begin{enumerate}[label = \textup{(\roman*)}]%
\item For every $\al$ with $|\al| \in (0,1)$, equation \eqref{eqn:determinant_inner_pos_s-th_root} has exactly one root $\be_i$ inside the open circle of radius $|\al|$ for each $i$. Furthermore, all $s$ roots $\be_i$ are distinct.
\item For every $\be$ with $|\be| \in (0,1)$, equation \eqref{eqn:determinant_inner_pos_s-th_root} has exactly one root $\al_i$ inside the open circle of radius $|\be|$ for each $i$. Furthermore, all $s$ roots $\al_i$ are distinct.
\end{enumerate}%
\end{lemma}%

\begin{proof}%
Dividing \eqref{eqn:determinant_inner_pos} by $(\al\be s)^s$ and taking the $s$-th root reduces \eqref{eqn:determinant_inner_pos} to \eqref{eqn:determinant_inner_pos_s-th_root}.

\noindent (i) The proof consists of three main steps:
\begin{enumerate}[label = (\alph*)]%
\item For every fixed $\al$ with $|\al| \in (0,1)$, equation \eqref{eqn:determinant_inner_pos} has exactly $s$ roots $\be$ inside the open circle of radius $|\al|$.
\item These $s$ roots $\be$ are distinct.
\item For every fixed $\al$ with $|\al| \in (0,1)$, equation \eqref{eqn:determinant_inner_pos_s-th_root} has at least one root $\be_i$ inside the open circle of radius $|\al|$ for each $i$.
\end{enumerate}%
Combining these three steps proves that for every fixed $\al$ with $|\al| \in (0,1)$, equation \eqref{eqn:determinant_inner_pos_s-th_root} has at exactly one root $\be_i$ inside the open circle of radius $|\al|$ for each $i$ and the $\be_i$'s are distinct.

\noindent (i)(a) Equation \eqref{eqn:determinant_inner_pos} is a polynomial of degree $2s$ in $\be$ and we will show that exactly $s$ roots are inside the open circle of radius $|\al|$. Other possible roots inside the open unit circle appear not to be useful, since these will produce a divergent solution, as follows from \eqref{eqn:initial_product_form_condition}. Divide both sides of \eqref{eqn:determinant_inner_pos} by $\al^{2s}$ and set $z = \be/\al$ to obtain
\begin{equation}%
f(z)^s - \al s^s z^{s + 1} = 0 \label{eqn:determinant_inner_pos_rewritten}
\end{equation}%
with $f(z) = (1 + s)(\rho + 1)z - (1 + s)\rho z^2 - 1$. Then $f(z)$ has the two roots
\begin{equation}%
\Lrootpos_{\scriptscriptstyle \pm} = \frac{\rho + 1 \pm \sqrt{(\rho + 1)^2 - 4\rho/(1 + s)}}{2\rho}. \label{eqn:limiting_roots_pos}
\end{equation}%
Observe that, for $0 < \rho < 1$, $\Lrootpos_{\scriptscriptstyle +} > 1$ and $0 < \Lrootpos_{\scriptscriptstyle -} < 1$. Furthermore, for $|z| = 1$ one establishes
\begin{align}
|f(z)|^s &= | (1 + s)(\rho + 1)z - (1 + s)\rho z^2 - 1 |^s \notag \\
&\ge | (1 + s)(\rho + 1)|z| - (1 + s)\rho |z|^2 - 1 |^s = s^s > |\alpha|s^s.
\end{align}
Hence, by Rouch\'e's theorem (see e.g.~\cite[Theorem~9.3.2]{Rouche_Hille1959}), equation \eqref{eqn:determinant_inner_pos_rewritten} has exactly $s$ roots inside the open unit circle. Thus, \eqref{eqn:determinant_inner_pos} has for each fixed $|\al| \in (0,1)$ exactly $s$ roots $\be$ inside the open circle of radius $|\al|$.

\noindent (i)(b) Label the left-hand side of \eqref{eqn:determinant_inner_pos_rewritten} as $g(z)$. For a root to be of (at least) multiplicity two, it must hold that $g(z) = g'(z) = 0$, which gives
\begin{equation}%
(\rho + 1)z + (s - 1)\rho z^2 - 1 = 0
\end{equation}%
with solutions
\begin{equation}%
\omega_{\scriptscriptstyle \pm} = \frac{-(\rho + 1) \pm \sqrt{(\rho + 1)^2 + 4(s - 1)\rho}}{2(s - 1)\rho}.
\end{equation}%
Note that for $0 < \rho < 1$, $\omega_{\scriptscriptstyle -} < 0$, $0 < \omega_{\scriptscriptstyle +} < 1$ and $\omega_{\scriptscriptstyle +}$ is a monotonically decreasing function of $\rho$ with $\omega_{\scriptscriptstyle +} = 1$ for $\rho = 0$ and $\omega_{\scriptscriptstyle +} = (\sqrt{s}-1)/(s-1)$ for $\rho = 1$. For $z = \omega_{\scriptscriptstyle +}$ equation \eqref{eqn:determinant_inner_pos_rewritten} reveals the contradiction
\begin{equation}%
\al = \frac{f(\omega_{\scriptscriptstyle +})^s}{s^s \omega_{\scriptscriptstyle +}^{s+1}} = \frac{\bigl( 1 + \rho (1 - 2 \omega_{\scriptscriptstyle +}) \bigr)^s}{\omega_{\scriptscriptstyle +}} \ge 1.
\end{equation}%
Thus, the $s$ roots $\be$ inside the open circle of radius $|\al|$ are distinct.

\noindent (i)(c) This step is presented in Appendix~\ref{app:proof_single_root_beta_i}. The proof was communicated to us by A.J.E.M. Janssen.

\noindent (ii) Set $z = \al/\be$, multiply \eqref{eqn:determinant_inner_pos_s-th_root} by $zs$ and rearrange to obtain
\begin{equation}%
z \bigl( (1 + s)(\rho + 1) - \unit_i \be^{1/s} s \bigr) = z^2 + (1 + s)\rho.
\end{equation}%
Label the left-hand side as $f(z)$ and the right-hand side as $g(z)$. Note that $f(z)$ has a single root within the unit circle. For $|z| = 1$ one establishes
\begin{align}%
|f(z)| &= | z \bigl( (1 + s)(\rho + 1) - \unit_i \be^{1/s} s \bigr)| = | (1 + s)(\rho + 1) - \unit_i \be^{1/s}s | \notag \\
&\ge  (1 + s)(\rho + 1) - |\unit_i| |\be^{1/s}| s = (1 + s)\rho + 1 + s(1 - |\be^{1/s}|) \notag \\
&> (1 + s)\rho + 1 \ge | z^2 + (1 + s)\rho | = |g(z)|.
\end{align}%
Hence, by Rouch\'e's theorem we establish that \eqref{eqn:determinant_inner_pos_s-th_root} has exactly one root $\al_i$ inside the open circle of radius $|\be|$ for each $i$. These roots $\al_i$ are distinct, since the $\unit_i$ are distinct.
\end{proof}%

\begin{remark}[Identical eigenvectors]\label{rem:identical_eigenvectors}%
We note that for fixed $\be$ with $|\be| \in (0,1)$ and fixed $i$, \eqref{eqn:determinant_inner_pos_s-th_root} is a quadratic equation in $\al$ and has two solutions, say $\al_{\scriptscriptstyle +}$ and $\al_{\scriptscriptstyle -}$, satisfying $|\al_{\scriptscriptstyle -}| < |\be| < |\al_{\scriptscriptstyle +}|$. Then, combining \eqref{eqn:inner_pos_eigenvector} and \eqref{eqn:determinant_inner_pos_s-th_root} we have the property that $\ibpos{\al_{\scriptscriptstyle +},\be} = \ibpos{\al_{\scriptscriptstyle -},\be}$, which we will use later.
\end{remark}%

We next provide information on the location and the number of zeros of $\det(\Dneg{\al,\be}) = 0$. The polynomial form of $\det(\Dneg{\al,\be}) = 0$ is too complicated for application of Rouch\'e's theorem. Therefore we apply the following matrix variant of Rouch\'e's theorem.

\begin{theorem}[de Smit \cite{MatrixRouche_deSmit1983}]\label{thm:deSmit}%
Let $A(z) = (a_{i,j}(z))$ and $B(z) = (b_{i,j}(z))$ be complex $n \times n$ matrices, where $B(z)$ is diagonal. The elements $a_{i,j}(z)$ and $b_{i,j}(z)$, $0 \le i,j \le n - 1$ are meromorphic functions in a simply connected region $\mathcal{S}$ in which $\mathcal{T}$ is the set of all poles of these functions. $\mathcal{C}$ is a rectifiable closed Jordan curve in $\mathcal{S} \setminus \mathcal{T}$. Let $x_B$ and $x_{A + B}$ be the number of zeros inside $\mathcal{C}$ of $\det(B(z))$ and $\det(A(z) + B(z))$, respectively, and $y_B$ and $y_{A + B}$ the number of poles inside $\mathcal{C}$ of $\det(B(z))$ and $\det(A(z) + B(z))$ (zeros and poles of higher order are counted according to this order). If $|b_{i,i}(z)| > \sum_{j = 0}^{n-1} |a_{i,j}(z)|$ on $\mathcal{C}$ for all $i = 0,1,\ldots,n - 1$, then
\begin{equation}%
x_{A + B} - y_{A + B} = x_B - y_B.
\end{equation}%
\end{theorem}%

\newpage
\begin{lemma}\label{lem:location_number_zeros_inner_neg}\hspace*{1em}%
\begin{enumerate}[label = \textup{(\roman*)}]%
\item For every $\al$ with $|\al| \in (0,1)$, the equation $\det(\Dneg{\al,\be}) = 0$ can be rewritten to
    \begin{equation}%
    \al^2 s^s + \be^2 ((1 + s)\rho)^s - \al \be s^s ( \func{+}{\be}^s + \func{-}{\be}^s ) = 0, \label{eqn:determinant_inner_neg}
    \end{equation}%
    where $\func{\pm}{\be}$ is given in \eqref{eqn:func_roots_eigenvectors}, and has exactly one root $\be$ in the open circle of radius $|\al|$.
\item For every $\be$ with $|\be| \in (0,1)$, equation \eqref{eqn:determinant_inner_neg} has exactly one root $\al$ in the open circle of radius $|\be|$.
\end{enumerate}%
\end{lemma}%

\begin{proof}%
(i) Setting $z = \be/\al$, we observe that
\begin{equation}%
\Dneg{\al,\be} = \al^2 (A(z) + B(z))
\end{equation}%
with
\begin{align}%
A(z) &= (1 + s) \rho z \Lo + s z \Lo^T + (1 + s)\rho z^2 \M{0,s-1} + s \M{s - 1,0}, \\
B(z) &= z ( \al z - (1 + s)(\rho + 1)) \I.
\end{align}%
Hence, $\det(\Dneg{\al,\be}) = \al^{2s} \det(A(z) + B(z))$. Let us take $\mathcal{C}$ to be the unit circle. Furthermore, we can verify that, for $|\al| \in (0,1)$ and $|z| = 1$,
\begin{equation}%
|b_{i,i}(z)| = |(1 + s)(\rho + 1) - \al z| > (1 + s)\rho + s = \sum_{j = 0}^{s - 1} |a_{i,j}(z)|, \quad i = 0,1,\ldots,s - 1
\end{equation}%
and the number of zeros and poles inside $\mathcal{C}$ of $\det(B(z))$ is $x_B = s$ and $y_B = 0$, respectively. The number of poles of $\det(A(z) + B(z))$ inside $\mathcal{C}$ is $y_{A + B} = 0$. By Theorem~\ref{thm:deSmit} we derive that $\det(A(z) + B(z))$ has exactly $s$ roots inside $\mathcal{C}$. Moreover,
\begin{align}%
\det(A(z) + B(z)) &= (-z)^{s - 1} \Bigl( s^s + z^2 ((1 + s)\rho)^s \notag \\
&\quad - z \sum_{i = 0}^{\lfloor s/2 \rfloor} (-1)^i \frac{s}{s - i} \binom{s - i}{i} (s(1 + s)\rho)^i ((1 + s)(\rho + 1) - \al z)^{s - 2i} \Bigr). \label{eqn:determinant_inner_neg_ratio}
\end{align}%
So, $\det(A(z) + B(z))$ has the root 0 of multiplicity $s - 1$ and exactly one non-zero root inside $\mathcal{C}$. The summation in \eqref{eqn:determinant_inner_neg_ratio} is known as the Waring formula. By using the following identity:
\begin{equation}%
x^s + y^s = \sum_{i = 0}^{\lfloor s/2 \rfloor} (-1)^i \frac{s}{s - i} \binom{s - i}{i} (xy)^i (x + y)^{s - 2i},
\end{equation}%
we can simplify the Waring formula, see e.g.~\cite[equation~(1)]{Waring_Gould1999}. So, we set $xy = s(1 + s)\rho$ and $x + y = (1 + s)(\rho + 1) - \al z$. Recall from \eqref{eqn:inner_neg_funcparam_satisfies} that $\func{+}{\al z}\func{-}{\al z} = (1 + s)\rho/s$. This allows us to establish that $x = s \func{+}{\al z}$ and $y = s \func{-}{\al z}$. Thus, \eqref{eqn:determinant_inner_neg_ratio} simplifies to
\begin{equation}%
\det(A(z) + B(z)) = (-z)^{s-1} \Bigl( s^s + z^2 ((1 + s)\rho)^s - z s^s (\func{+}{\al z}^s + \func{-}{\al z}^s) \Bigr).
\end{equation}%
Since $\det(A(z) + B(z))$ has exactly $s$ roots in the unit circle and 0 is a root of multiplicity $s - 1$ we are led to the conclusion that there exists exactly one non-zero root in the unit circle. Thus, \eqref{eqn:determinant_inner_neg} has exactly one non-zero root $\be$ in the open circle of radius $|\al|$.

\noindent (ii) The proof is identical to the proof of (i).
\end{proof}%

\begin{remark}[Notation]%
For a fixed $\al$ with $|\al| \in (0,1)$ the roots of \eqref{eqn:determinant_inner_pos} are denoted by $\be_1,\be_2,\ldots,\be_s$ with $|\be_i| < |\al|, ~ i = 1,2,\ldots,s$ and the single root of \eqref{eqn:determinant_inner_neg} is denoted by $\be_{s + 1}$ with $|\be_{s + 1}| < |\al|$. Similar indexing exists for the roots $\al$ for a fixed $\be$ with $|\be| \in (0,1)$.
\end{remark}%

%%%%%%%%%%%%%%%%%%%%%%%%%%%%%%%%%%%%%%%%%%%%%%%%%%%%%%%
%%%%%%%%%%%%%%%%%%%%%%%%%%%%%%%%%%%%%%%%%%%%%%%%%%%%%%%
%%%%%%%%%%%%%%%%%%% NEW SUBSECTION %%%%%%%%%%%%%%%%%%%%
%%%%%%%%%%%%%%%%%%%%%%%%%%%%%%%%%%%%%%%%%%%%%%%%%%%%%%%
%%%%%%%%%%%%%%%%%%%%%%%%%%%%%%%%%%%%%%%%%%%%%%%%%%%%%%%

\subsection{Initial solution}%
\label{subsec:initial_solution}%

Lemmas~\ref{lem:solution_inner_equations}, \ref{lem:location_number_zeros_inner_pos} and \ref{lem:location_number_zeros_inner_neg} characterize basic solutions satisfying the inner equations \eqref{eqn:equilibrium_eqs_I+}-\eqref{eqn:equilibrium_eqs_I-}. In Lemma~\ref{lem:initial_solution} we specify the form of $\qb{n}$.

\begin{lemma}[Initial solution]\label{lem:initial_solution}%
For $|\al| \in (0,1)$, let $\be_1,\be_2,\ldots,\be_s$ be the roots of \eqref{eqn:determinant_inner_pos} with $|\be_i| < |\al|, ~ i = 1,2,\ldots,s$ and $\ibpos{\al,\be_i}$ the corresponding non-zero vectors satisfying \eqref{eqn:inner_pos_eigenvector}. Symmetrically, let $\be_{s + 1}$ be the root of \eqref{eqn:determinant_inner_neg} with $|\be_{s + 1}| < |\al|$ and $\ibneg{\al,\be_{s + 1}}$ the corresponding non-zero vector satisfying \eqref{eqn:inner_neg_eigenvector}. Then there exists exactly one $\al$ with $|\al| \in (0,1)$ for which there exists a non-zero vector $\hb = (\h{0},\h{1},\ldots,\h{s - 1})^T$ and coefficients $\ch_1,\ch_2,\ldots,\ch_{s + 1}$ such that
\begin{equation}%
\pb{m,n} = \begin{cases}%
\al^m \sum_{i = 1}^s \ch_i \be_i^n \ibpos{\al,\be_i}, & m \ge 0, ~ n \ge 1, \\
\al^m \hb, & m \ge 0, ~ n = 0, \\
\ch_{s + 1} \al^m \be_{s + 1}^{-n} \ibneg{\al,\be_{s + 1}}, & m \ge 0, ~ n \le -1,
\end{cases} \label{eqn:solution_initial}%
\end{equation}%
satisfies \eqref{eqn:equilibrium_eqs_I+}-\eqref{eqn:equilibrium_eqs_H}. This unique $\al$ is equal to $\rho^{1 + s}$.

The vector $\hb$ is given by
\begin{equation}%
\hb = \al \bigl( \A{0,1} + \al \A{-1,1} \bigr)^{-1} \sum_{i = 1}^s \ch_i \ibpos{\al,\be_i} \label{eqn:initial_solution_horizontal_eigenvector}
\end{equation}%
and the coefficients $\ch_1,\ch_2,\ldots,\ch_s$ satisfy
\begin{align}%
&\sum_{i = 1}^s  \ch_i \Bigl( \be_i \bigl( \A{1,-1} + \al \A{0,-1} \bigr) + \al^2 \B{0,0} \bigl( \A{0,1} + \al \A{-1,1} \bigr)^{-1} \Bigr) \ibpos{\al,\be_i} \notag \\
&= - \ch_{s + 1} \be_{s + 1} \bigl( \B{1,1} + \al \B{0,1} \bigr) \ibneg{\al,\be_{s + 1}} \label{eqn:initial_solution_coefficients}
\end{align}%
and $\ch_{s + 1} = 1$.
\end{lemma}%

\begin{proof}%
Substituting \eqref{eqn:solution_initial} into \eqref{eqn:equilibrium_eqs_H+} yields
\begin{equation}%
\hb = - \frac{1}{\al} \bigl( \A{0,1} + \al \A{-1,1} \bigr)^{-1} \sum_{i = 1}^s \ch_i \bigl( \al \be_i \A{0,0} + \be_i^2 \A{1,-1} + \al \be_i^2 \A{0,-1} \bigr) \ibpos{\al,\be_i}.
\end{equation}%
Furthermore, by using \eqref{eqn:solution_inner_neg_should_satisfy} for the summands one obtains \eqref{eqn:initial_solution_horizontal_eigenvector}. Equation \eqref{eqn:initial_solution_coefficients} follows from substituting \eqref{eqn:solution_initial} into \eqref{eqn:equilibrium_eqs_H} and using \eqref{eqn:initial_solution_horizontal_eigenvector}. One can arbitrarily choose $\ch_{s + 1}$ due to the normalization that follows at the end.
\end{proof}%

%%%%%%%%%%%%%%%%%%%%%%%%%%%%%%%%%%%%%%%%%%%%%%%%%%%%%%%
%%%%%%%%%%%%%%%%%%%%%%%%%%%%%%%%%%%%%%%%%%%%%%%%%%%%%%%
%%%%%%%%%%%%%%%%%%% NEW SUBSECTION %%%%%%%%%%%%%%%%%%%%
%%%%%%%%%%%%%%%%%%%%%%%%%%%%%%%%%%%%%%%%%%%%%%%%%%%%%%%
%%%%%%%%%%%%%%%%%%%%%%%%%%%%%%%%%%%%%%%%%%%%%%%%%%%%%%%

\subsection{Compensation on the vertical boundary}%
\label{subsec:compensation_vertical_boundary}%

The initial solution \eqref{eqn:solution_initial} does not satisfy the positive and negative vertical boundary. We consider a term $\ch \al^m \be^n \ibpos{\al,\be}$ that satisfies the positive inner balance equations and stems from a solution that satisfies both the inner and horizontal boundary equations. For now, we refer to this term as the original term. In case of the initial solution \eqref{eqn:solution_initial}, this would be one of the $s$ product-form solutions of the positive quadrant. The idea behind the compensation approach is to add compensation terms $\sum_{i = 1}^s \cv_i \al_i^m \be^n \ibpos{\al_i,\be}$ to the original term, such that the linear combination
\begin{equation}%
\ch \al^m \be^n \ibpos{\al,\be} + \sum_{i = 1}^s \cv_i \al_i^m \be^n \ibpos{\al_i,\be}, \quad m \ge 0, ~ n \ge 1,  \label{eqn:compensation_vertical_pos_terms}
\end{equation}%
satisfies both \eqref{eqn:equilibrium_eqs_V+} and \eqref{eqn:equilibrium_eqs_I+}. Substituting \eqref{eqn:compensation_vertical_pos_terms} in \eqref{eqn:equilibrium_eqs_V+} gives
\begin{equation}%
\ch \be^{n - 1} V_+(\al,\be) \ibpos{\al,\be} + \sum_{i = 1}^s \cv_i \be^{n - 1} V_+(\al_i,\be) \ibpos{\al_i,\be} = \zerob, \quad n \ge 2, \label{eqn:equilibrium_eqs_V+_compensation_subs}
\end{equation}%
where $V_+(\al,\be) = \be(\A{0,0} + I) + \be^2 \A{0,-1} + \al \A{-1,1}$. Note that we indeed require that the original term and the compensation terms share the same $\be$ and therefore we obtain the $s$ roots $\al_1,\al_2,\ldots,\al_s$ with $|\al_i| < |\be|$ from \eqref{eqn:determinant_inner_pos}. Clearly, \eqref{eqn:compensation_vertical_pos_terms} now satisfies the positive inner equations \eqref{eqn:equilibrium_eqs_I+}. This leaves the $s$ coefficients $\cv_i$ to satisfy the $s$ equations \eqref{eqn:equilibrium_eqs_V+_compensation_subs}. However, from Remark~\ref{rem:identical_eigenvectors} we deduce that there exists a specific $i$ for which $\ibpos{\al,\be} = \ibpos{\al_i,\be}$ and thus there is only one coefficient that is required to be non-zero.

Let us now consider a term from the negative quadrant $\ch \al^m \be^{-n} \ibneg{\al,\be}$ that satisfies the negative inner balance equations and stems from the same solution that satisfies both the inner and horizontal boundary equations. Let us now refer to this term as the original term. In case of the initial solution \eqref{eqn:solution_initial}, this would be the solution on the negative quadrant. As before, the compensation approach dictates to add a compensation term $\cv_{s + 1} \al_{s + 1}^m \be^{-n} \ibneg{\al_{s + 1},\be}$ to the original term, such that the linear combination
\begin{equation}%
\ch \al^m \be^{-n} \ibneg{\al,\be} + \cv_{s + 1} \al_{s + 1}^m \be^{-n} \ibneg{\al_{s + 1},\be}, \quad m \ge 0, ~ n \le -1, \label{eqn:compensation_vertical_neg_terms}
\end{equation}%
satisfies both \eqref{eqn:equilibrium_eqs_V-} and \eqref{eqn:equilibrium_eqs_I-}. Substituting \eqref{eqn:compensation_vertical_neg_terms} in \eqref{eqn:equilibrium_eqs_V-} gives
\begin{equation}%
\ch \be^{-n + 1} V_-(\al,\be) \ibneg{\al,\be} + \cv_{s + 1} \be^{-n + 1} V_-(\al_{s + 1},\be) \ibneg{\al_{s + 1},\be} = \zerob, \quad n \le -2, \label{eqn:equilibrium_eqs_V-_compensation_subs}
\end{equation}%
where $V_-(\al,\be) = \be(\B{0,0} + s\M{0,0}) + \be^2 \B{0,1} + \al \B{-1,-1}$. Note that we indeed require that the original term and the compensation term share the same $\be$ and therefore we obtain the root $\al_{s + 1}$ with $|\al_{s + 1}| < |\be|$ from \eqref{eqn:determinant_inner_neg}. This ensures that the linear combination \eqref{eqn:compensation_vertical_neg_terms} satisfies the negative inner equations \eqref{eqn:equilibrium_eqs_I-}. Even though there is only one coefficient $\cv_{s + 1}$ for $s$ equations, it turns out that this provides enough freedom to satisfy \eqref{eqn:equilibrium_eqs_V-_compensation_subs}.

\begin{lemma}[Vertical compensation]\label{lem:compensation_vertical}\hspace*{1em}%
\begin{enumerate}[label = \textup{(\roman*)}]%
\item Consider the product form $\ch\al^m\be^n \ibpos{\al,\be}$ that satisfies the positive inner equations \eqref{eqn:equilibrium_eqs_I+} that stems from a solution that satisfies the inner and horizontal boundary equations. For this $\al$ and fixed $\be$, let $\al_1$ be the root that satisfies $\ibpos{\al,\be} = \ibpos{\al_1,\be}$ with $|\al_1| < |\be|$. Then there exists a coefficient $\cv_1$ such that
    \begin{equation}%
    \pb{m,n} = \ch \al^m \be^n \ibpos{\al,\be} + \cv_1 \al_1^m \be^n \ibpos{\al_1,\be}, \quad m \ge 0, ~ n \ge 1, \label{eqn:compensation_vertical_pos_solution}
    \end{equation}%
    satisfies \eqref{eqn:equilibrium_eqs_I+} and \eqref{eqn:equilibrium_eqs_V+}. The coefficient $\cv_1$ satisfies
    \begin{equation}%
    \cv_1 = - \ch \frac{1 - \frac{\be}{\al} (1 + s)\rho}{1 - \frac{\be}{\al_1} (1 + s)\rho}. \label{eqn:compensation_vertical_pos_coefficients}
    \end{equation}
\item Consider a product form $\ch\al^m\be^n \ibneg{\al,\be}$ that satisfies the negative inner equations \eqref{eqn:equilibrium_eqs_I-} that stems from a solution that satisfies the inner and horizontal boundary equations. For this $\be$, let $\al_{s + 1}$ be the root of \eqref{eqn:determinant_inner_neg} with $|\al_{s + 1}| < |\be|$ and let $\ibneg{\al_{s + 1},\be}$ be the corresponding non-zero vector satisfying \eqref{eqn:inner_neg_eigenvector}. Then there exists a coefficient $\cv_{s + 1}$ such that
    \begin{equation}%
    \pb{m,n} = \ch \al^m \be^{-n} \ibneg{\al,\be} + \cv_{s + 1} \al_{s + 1}^m \be^{-n} \ibneg{\al_{s + 1},\be}, \quad m \ge 0, ~ n \le -1, \label{eqn:compensation_vertical_neg_solution}
    \end{equation}%
    satisfies \eqref{eqn:equilibrium_eqs_I-} and \eqref{eqn:equilibrium_eqs_V-}. The coefficient $\cv_{s + 1}$ is given by
    \begin{equation}%
    \cv_{s + 1} = -\ch \frac{s - \frac{\be}{\al} (1 + s)\rho \ineg{\al,\be,s-1}}{s - \frac{\be}{\al_{s + 1}} (1 + s)\rho \ineg{\al_{s + 1},\be,s-1}}. \label{eqn:compensation_vertical_neg_coefficient}
    \end{equation}%
\end{enumerate}%
\end{lemma}%

\begin{proof}%
(i) Use \eqref{eqn:solution_inner_pos_should_satisfy} to establish that
\begin{equation}%
V_+(\al,\be)\ibpos{\al,\be} = \bigl(\be \I - \frac{\be^2}{\al} \A{1,-1} \bigr) \ibpos{\al,\be} = \be \bigl(1 - \frac{\be}{\al} (1 + s)\rho \bigr) \ibpos{\al,\be}. \label{eqn:compensation_vertical_pos_simplification_V+}
\end{equation}%
Then using Remark~\ref{rem:identical_eigenvectors} we can establish that there exists a single $\cv_i$ that is non-zero. We label the single non-zero coefficient as $\cv_1$ and find it from \eqref{eqn:equilibrium_eqs_V+_compensation_subs} with the help of \eqref{eqn:compensation_vertical_pos_simplification_V+}.

\noindent (ii) Use \eqref{eqn:solution_inner_neg_should_satisfy} to establish that
\begin{equation}%
V_-(\al,\be)\ibneg{\al,\be} = \be \bigl(s \M{0,0} - \frac{\be}{\al} (1 + s)\rho \M{0,s - 1} \bigr) \ibneg{\al,\be}.
\end{equation}%
Thus, \eqref{eqn:equilibrium_eqs_V-_compensation_subs} reduces to a single equation and $\cv_{s + 1}$ follows directly.
\end{proof}%

%%%%%%%%%%%%%%%%%%%%%%%%%%%%%%%%%%%%%%%%%%%%%%%%%%%%%%%
%%%%%%%%%%%%%%%%%%%%%%%%%%%%%%%%%%%%%%%%%%%%%%%%%%%%%%%
%%%%%%%%%%%%%%%%%%% NEW SUBSECTION %%%%%%%%%%%%%%%%%%%%
%%%%%%%%%%%%%%%%%%%%%%%%%%%%%%%%%%%%%%%%%%%%%%%%%%%%%%%
%%%%%%%%%%%%%%%%%%%%%%%%%%%%%%%%%%%%%%%%%%%%%%%%%%%%%%%

\subsection{Compensation on the horizontal boundary}%
\label{subsec:compensation_horizontal_boundary}%

The solution obtained after compensation on the vertical boundary, as outlined in Lemma~\ref{lem:compensation_vertical}, does not satisfy the horizontal boundary equations. So, we need to compensate for the error on the horizontal boundary by adding new terms. The compensation procedure on the horizontal boundary has a few differences from the one described in the previous section. We outline these differences by informally treating the compensation of a product-form term of the positive quadrant. The difference in the compensation procedure is due to the fact that adding compensation terms only for the positive quadrant does not make the solution satisfy the horizontal boundary equations. Intuitively, this originates from the fact that the horizontal boundary, i.e.~$n = 0$, is connected to both the positive and negative quadrant. Thus, for the horizontal compensation step, we need to add product-form terms for the complete positive half-plane. It turns out, that these product-form terms are nearly identical to the initial solution, which indeed satisfied both the inner and horizontal boundary equations. The same procedure holds for a product-form term of the negative quadrant.

The formal compensation on the horizontal boundary is outlined in the next lemma.

\begin{lemma}[Horizontal compensation]\label{lem:compensation_horizontal}\hspace*{1em}%
\begin{enumerate}[label = \textup{(\roman*)}]%
\item Consider the product form $\cv \al^m \be^n \ibpos{\al,\be}$ that satisfies the positive inner equations \eqref{eqn:equilibrium_eqs_I+} that stems from a solution that satisfies the inner and vertical boundary equations. For this $\al$, let $\be_1,\be_2,\ldots,\be_s$ be the roots of \eqref{eqn:determinant_inner_pos} with $|\be_i| < |\al|, ~ i = 1,2,\ldots,s$ and let $\ibpos{\al,\be_i}$ be the corresponding non-zero vectors satisfying \eqref{eqn:inner_pos_eigenvector}. Symmetrically, for this $\al$, let $\be_{s + 1}$ be the root of \eqref{eqn:determinant_inner_neg} with $|\be_{s + 1}| < |\al|$ and let $\ibneg{\al,\be_{s + 1}}$ be the corresponding non-zero vector satisfying \eqref{eqn:inner_neg_eigenvector}. Then there exists a non-zero vector $\hb$ and coefficients $\ch_1,\ch_2,\ldots,\ch_{s + 1}$ such that
    \begin{equation}%
    \pb{m,n} = \begin{cases}%
    \cv \al^m \be^n \ibpos{\al,\be} + \sum_{i = 1}^s \ch_i \al^m \be_i^n \ibpos{\al,\be_i}, & m \ge 0, ~ n \ge 1, \\
    \al^m \hb, & m \ge 0, ~ n = 0, \\
    \ch_{s + 1} \al^m \be_{s + 1}^{-n} \ibneg{\al,\be_{s + 1}}, & m \ge 0, ~ n \le -1,
    \end{cases}\label{eqn:compensation_horizontal_pos_solution}%
    \end{equation}%
    satisfies \eqref{eqn:equilibrium_eqs_I+}-\eqref{eqn:equilibrium_eqs_H}.

    The vector $\hb$ and the coefficients $\ch_1,\ch_2,\ldots,\ch_{s + 1}$ satisfy
    \begin{align}%
    \bigl(\A{0,1} + \al \A{-1,1} \bigr) \hb - \al \sum_{i = 1}^s \ch_i \ibpos{\al,\be_i} &= \cv \al \ibpos{\al,\be}, \label{eqn:compensation_horizontal_pos_h}\\
    \al \B{0,0} \hb + \sum_{i = 1}^s \ch_i \be_i \bigl( \A{1,-1} + \al \A{0,-1} \bigr) \ibpos{\al,\be_i} \notag \\
    + \ch_{s + 1} \be_{s + 1} \bigl( \B{1,1} + \al \B{0,1} \bigr) \ibneg{\al,\be_{s + 1}} &= - \cv \be \bigl( \A{1,-1} + \al \A{0,-1} \bigr) \ibpos{\al,\be}, \label{eqn:compensation_horizontal_pos_eta_i}\\
    - \ch_{s + 1} \al s + \al s \h{0} + (1 + s)\rho q \h{s-1} &= 0. \label{eqn:compensation_horizontal_pos_eta_tilde}
    \end{align}%
\item Consider the product form $\cv \al^m \be^{-n} \ibneg{\al,\be}$ that satisfies the negative inner equations \eqref{eqn:equilibrium_eqs_I+} that stems from a solution that satisfies the inner and vertical boundary equations. For this $\al$, let $\be_1,\be_2,\ldots,\be_s$ be the roots of \eqref{eqn:determinant_inner_pos} with $|\be_i| < |\al|, ~ i = 1,2,\ldots,s$ and let $\ibpos{\al,\be_i}$ be the corresponding non-zero vectors satisfying \eqref{eqn:inner_pos_eigenvector}. Symmetrically, for this $\al$, let $\be_{s + 1}$ be the root of \eqref{eqn:determinant_inner_neg} with $|\be_{s + 1}| < |\al|$ and let $\ibneg{\al,\be_{s + 1}}$ be the corresponding non-zero vector satisfying \eqref{eqn:inner_neg_eigenvector}. Then there exists a non-zero vector $\hb$ and coefficients $\ch_1,\ch_2,\ldots,\ch_{s + 1}$ such that
    \begin{equation}%
    \pb{m,n} = \begin{cases}%
    \sum_{i = 1}^s \ch_i \al^m \be_i^n \ibpos{\al,\be_i}, & m \ge 0, ~ n \ge 1, \\
    \al^m \hb, & m \ge 0, ~ n = 0, \\
    \cv \al^m \be^{-n} \ibneg{\al,\be} + \ch_{s + 1} \al^m \be_{s + 1}^{-n} \ibneg{\al,\be_{s + 1}}, & m \ge 0, ~ n \le -1,
    \end{cases}%
    \end{equation}%
    satisfies \eqref{eqn:equilibrium_eqs_I+}-\eqref{eqn:equilibrium_eqs_H}.

    The vector $\hb$ and the coefficients $\ch_1,\ch_2,\ldots,\ch_{s + 1}$ satisfy
    \begin{align}%
    \bigl(\A{0,1} + \al \A{-1,1} \bigr) \hb - \al \sum_{i = 1}^s \ch_i \ibpos{\al,\be_i} &= \zerob, \label{eqn:compensation_horizontal_neg_h}\\
    \al \B{0,0} \hb + \sum_{i = 1}^s \ch_i \be_i \bigl( \A{1,-1} + \al \A{0,-1} \bigr) \ibpos{\al,\be_i} \notag \\
    + \ch_{s + 1} \be_{s + 1} \bigl( \B{1,1} + \al \B{0,1} \bigr) \ibneg{\al,\be_{s + 1}} &= - \cv \be \bigl( \B{1,1} + \al \B{0,1} \bigr) \ibneg{\al,\be}, \label{eqn:compensation_horizontal_neg_eta_i} \\
    - \ch_{s + 1} \al s + \al s \h{0} + (1 + s)\rho q \h{s-1} &= \cv \al s. \label{eqn:compensation_horizontal_neg_eta_tilde}
    \end{align}%
\end{enumerate}%
\end{lemma}%

\begin{proof}%
(i) Equations \eqref{eqn:compensation_horizontal_pos_h} and \eqref{eqn:compensation_horizontal_pos_eta_i} follow from the exact same reasoning as outlined in the proof of Lemma~\ref{lem:initial_solution}. Equation \eqref{eqn:compensation_horizontal_pos_eta_tilde} follows from substituting \eqref{eqn:compensation_horizontal_pos_solution} in the negative horizontal boundary equations \eqref{eqn:equilibrium_eqs_H-} and using \eqref{eqn:solution_inner_neg_should_satisfy}. Note that this indeed reduces to a single equation.

\noindent (ii) The proof is identical to the proof of (i).
\end{proof}%

%%%%%%%%%%%%%%%%%%%%%%%%%%%%%%%%%%%%%%%%%%%%%%%%%%%%%%%
%%%%%%%%%%%%%%%%%%%%%%%%%%%%%%%%%%%%%%%%%%%%%%%%%%%%%%%
%%%%%%%%%%%%%%%%%%% NEW SUBSECTION %%%%%%%%%%%%%%%%%%%%
%%%%%%%%%%%%%%%%%%%%%%%%%%%%%%%%%%%%%%%%%%%%%%%%%%%%%%%
%%%%%%%%%%%%%%%%%%%%%%%%%%%%%%%%%%%%%%%%%%%%%%%%%%%%%%%

\subsection{Constructing the equilibrium distribution}%
\label{subsec:constructing_equilibrium_distribution}%

The compensation approach ultimately leads to the following expressions for the equilibrium distribution of the SED system, where the symbol $\propto$ indicates proportionality.
\begin{enumerate}[label = \textup{(\roman*)}]%
    \item For $m \ge 0, ~ n \ge 1$,
    \begin{subequations}%
    \label{eqn:equilibrium_distribution_proportional}
    \begin{align}%
    \pb{m,n} &\propto \sum_{l = 0}^\infty \sum_{i = 1}^{(s + 1)^l} \sum_{j = 1}^s \ch_{l,d(i) + j} \al_{l,i}^m \be_{l,d(i) + j}^n \ibpos{\al_{l,i},\be_{l,d(i) + j}} \notag \\
    & \quad+ \sum_{l = 0}^\infty \sum_{i = 1}^{(s + 1)^l} \sum_{j = 1}^s \cv_{l + 1,d(i) + j} \al_{l + 1,d(i) + j}^m \be_{l,d(i) + j}^n \ibpos{\al_{l + 1,i},\be_{l,d(i) + j}}.
    \label{eqn:equilibrium_distribution_proportional_pos}
    \intertext{\item For $m \ge 0$,}
    \pb{m,0} &\propto \sum_{l = 0}^\infty \sum_{i = 1}^{(s + 1)^l} \al_{l,i}^m \hb_{l,i}. \label{eqn:equilibrium_distribution_proportional_hor}
    \intertext{\item For $m \ge 0, ~ n \le -1$,}
    \pb{m,n} &\propto \sum_{l = 0}^\infty \sum_{i = 1}^{(s + 1)^l} \ch_{l,i(s + 1)} \al_{l,i}^m \be_{l,i(s + 1)}^{-n} \ibneg{\al_{l,i},\be_{l,i(s + 1)}} \notag \\
    &\quad + \sum_{l = 0}^\infty \sum_{i = 1}^{(s + 1)^l} \cv_{l + 1,i(s + 1)} \al_{l+1,i(s + 1)}^m \be_{l,i(s + 1)}^{-n} \ibneg{\al_{l+1,i(s + 1)},\be_{l,i(s + 1)}}. \label{eqn:equilibrium_distribution_proportional_neg}
    \end{align}%
    \end{subequations}%
\end{enumerate}%
We briefly describe the indexing of the compensation terms, which grows as a tree. We indicate the level at which a parameter resides with $l$, starting at level $l = 0$. Within a level, we differentiate between parameters by using an additional index $i$. The procedure for generating terms is as follows:
\begin{enumerate}%
\item[(I)] The initial solution is determined from Lemma~\ref{lem:initial_solution}.
\item[(V)] For a vertical compensation step with fixed $\be_{l,i}$:
    \begin{enumerate}[label = (\roman*)]
    \item If the index $i$ is \textit{not} a multiple of $s + 1$, then the compensation terms are determined according to Lemma~\ref{lem:compensation_vertical}(i).
    \item Otherwise, the compensation terms are determined according to Lemma~\ref{lem:compensation_vertical}(ii).
    \end{enumerate}%
\item[(H)] For a horizontal compensation step with fixed $\al_{l,i}$:
    \begin{enumerate}[label = (\roman*)]
    \item If the index $i$ is \textit{not} a multiple of $s + 1$, then the compensation terms are determined according to Lemma~\ref{lem:compensation_horizontal}(i).
    \item Otherwise, the compensation terms are determined according to Lemma~\ref{lem:compensation_horizontal}(ii).
    \end{enumerate}%
\end{enumerate}%
Recall the definition $d(i) \defi (i - 1)(s + 1)$ and note that $d(i) + s + 1 = i(s + 1)$. Informally, the indexing of subsequent terms is visualized in Figure~\ref{fig:indexing_compensation_terms}. Note that the first term in the compensation approach is $\al_{0,1} = \rho^{1 + s}$, c.f.~Lemma~\ref{lem:initial_solution}.

\begin{figure}
\centering%
\includegraphics{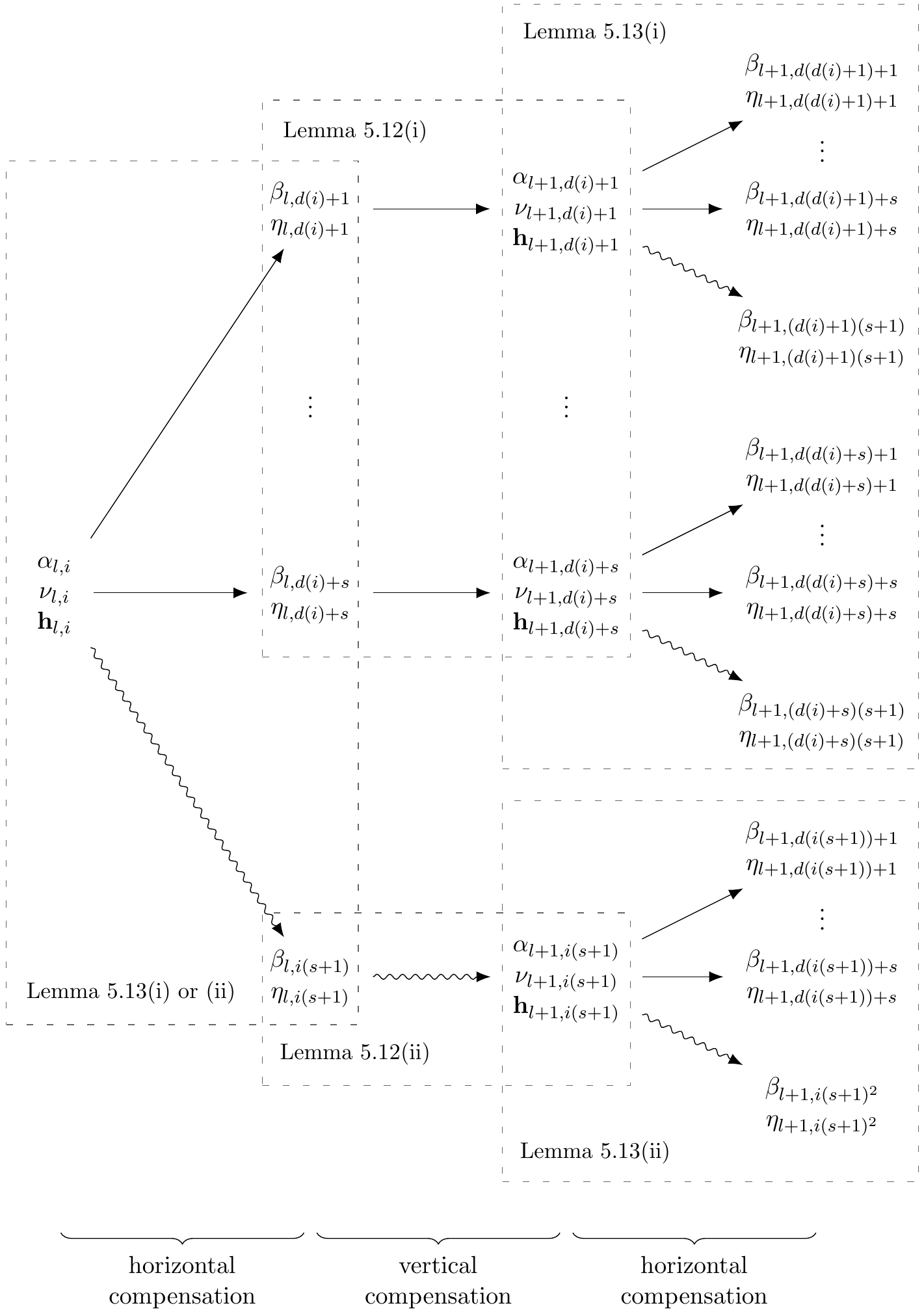}% PDF figure
\caption{Indexing of the compensation terms for a subtree. A dashed rectangle indicates according to which lemma the compensation terms are generated. A straight arrow indicates that the roots $\al$ or $\be$ are obtained through \eqref{eqn:determinant_inner_pos} and a snaked arrow through \eqref{eqn:determinant_inner_neg}.}%
\label{fig:indexing_compensation_terms}%
\end{figure}

%%%%%%%%%%%%%%%%%%%%%%%%%%%%%%%%%%%%%%%%%%%%%%%%%%%%%%%
%%%%%%%%%%%%%%%%%%%%%%%%%%%%%%%%%%%%%%%%%%%%%%%%%%%%%%%
%%%%%%%%%%%%%%%%%%% NEW SUBSECTION %%%%%%%%%%%%%%%%%%%%
%%%%%%%%%%%%%%%%%%%%%%%%%%%%%%%%%%%%%%%%%%%%%%%%%%%%%%%
%%%%%%%%%%%%%%%%%%%%%%%%%%%%%%%%%%%%%%%%%%%%%%%%%%%%%%%

\subsection{Absolute convergence}%
\label{subsec:absolute_convergence}%

In this subsection we prove that the series for the equilibrium probabilities, as formulated in \eqref{eqn:equilibrium_distribution_proportional}, are absolutely convergent. Before we are able to do that, we need some preliminary results. Specifically, we establish that the sequences $\{\al_{l,i}\}_{l \in \Nat_0, ~ i = 1,2,\ldots,(s + 1)^l}$ and $\{\be_{l,i}\}_{l \in \Nat_0, ~ i = 1,2,\ldots,(s + 1)^{l + 1}}$ decrease \textit{exponentially} fast and \textit{uniformly} in the levels of the parameter tree. Furthermore, we investigate the asymptotic behavior of (ratios of) the parameters $\al_{l,i}$, $\be_{l,i}$, the coefficients, and the elements of the (eigen)vectors.

\begin{corollary}\label{cor:exponential_convergence_alpha_beta}%
Let $\bar{\al}_l \defi \max_{i = 1,2,\ldots,(s + 1)^l} |\al_{l,i}|$ and $\bar{\be}_l \defi \max_{i = 1,2,\ldots,(s + 1)^{l + 1}} |\be_{l,i}|$. Then,
\begin{enumerate}[label = \textup{(\roman*)}]%
\item $\displaystyle \rho^{1 + s} = |\al_{0,1}| > \bar{\be}_0 > \bar{\al}_1 > \bar{\be}_1 > \bar{\al}_2 > \bar{\be}_2 > \cdots$
\item There exists $c \in (0,1)$, such that $\displaystyle 0 < \bar{\al}_l,\bar{\be}_l < c^l$.
\end{enumerate}%
\end{corollary}

\begin{proof}%
(i) Follows immediately from Lemmas~\ref{lem:location_number_zeros_inner_pos} and \ref{lem:location_number_zeros_inner_neg}.

\noindent (ii) For a fixed $\al$, let $\be_1,\be_2,\ldots,\be_s$ be the roots of \eqref{eqn:determinant_inner_pos} with $|\be_i| < |\al|, ~ i = 1,2,\ldots,s$ and let $\be_{s + 1}$ be the root of \eqref{eqn:determinant_inner_neg} with $|\be_{s + 1}| < |\al|$. We define
\begin{equation}%
t(\al) \defi \max_{i = 1,2,\ldots,s + 1} |\be_i/\al|
\end{equation}%
and let $\mathcal{C} = \{ \al \in \Complex \mid |\al| \le \al_{0,1} = \rho^{1 + s} \}$. Using Lemmas~\ref{lem:location_number_zeros_inner_pos} and \ref{lem:location_number_zeros_inner_neg} we have $t(\al) < 1$ and in particular, since $\rho \in (0,1)$, we have that $\mathcal{C}$ is a closed proper subset of the unit disc and thus $\Lboundal \defi \max_{\al \in \mathcal{C}} t(\al) < 1$. One can reason along the same lines and obtain a bound $\Lboundbe$ for a fixed $\be$. The exponential decrease in terms of the level $l$ follows from the fact that each $\be$ that is generated from an $\al$ satisfies $|\be| < |\al| \Lboundal$ and each $\al$ that is generated from a $\be$ satisfies $|\al| < |\be|\Lboundbe$. Thus, we define $c \defi \Lboundal \Lboundbe$.
\end{proof}%

In Corollary~\ref{cor:exponential_convergence_alpha_beta} we have established that both sequences $\{\al_{l,i}\}_{l \in \Nat_0, ~ i = 1,2,\ldots,(s + 1)^l}$ and $\{\be_{l,i}\}_{l \in \Nat_0, ~ i = 1,2,\ldots,(s + 1)^{l + 1}}$ tend to zero as $l \to \infty$. Thus, letting $\al \downarrow 0$ or $\be \downarrow 0$ is equivalent to letting $l \to \infty$. In what follows, whenever the specific dependence on the level $l$ is not needed, we opt to use the simpler notation of Sections~\ref{subsec:preliminary_results}-\ref{subsec:compensation_horizontal_boundary}. The next lemma presents the limiting behavior of the compensation parameters and their associated eigenvectors.

\begin{lemma}\label{lem:convergence_alpha_beta}\hspace*{1em}%
\begin{enumerate}[label = \textup{(\roman*)}]%
\item For a fixed $\al$, let $\be_1,\be_2,\ldots,\be_s$ be the roots of \eqref{eqn:determinant_inner_pos} with $|\be_i| < |\al|, ~ i = 1,2,\ldots,s$ and let $\be_{s + 1}$ be the root of \eqref{eqn:determinant_inner_neg} with $|\be_{s + 1}| < |\al|$. Then, as $\al \downarrow 0$,
    \begin{enumerate}[label = \textup{(\alph*)}]%
    \item The ratio $\be_i/\al \to \Lal, ~ i = 1,2,\ldots,s$ with $\Lal < 1$, where $\Lal$ is the root of
    \begin{equation}%
    \Lrootpos^2(1 + s)\rho - \Lrootpos(1 + s)(\rho + 1) + 1 = 0. \label{eqn:limiting_ratio_be_i/al_satisfies}
    \end{equation}%
    \item The eigenvector $\ibpos{\al,\be_i} \to \eb{0}, ~ i = 1,2,\ldots,s$.
    \item The ratio $\be_{s + 1}/\al \to \Lalt$ with $\Lalt < 1$, where $\Lalt$ is the root of
    \begin{equation}%
    \Lrootneg^2((1 + s)\rho)^s - \Lrootneg s^s (\func{+}{0}^s + \func{-}{0}^s) + s^s = 0, \label{eqn:limiting_ratio_bet/al_satisfies}
    \end{equation}%
    where $\func{\pm}{0}$ is defined in \eqref{eqn:func_roots_eigenvectors}.
    \item The elements of the eigenvector $\ineg{\al,\be_{s + 1},r} \to \func{+}{0}^r, ~ r = 0,1,\ldots,s-1$.
    \end{enumerate}%
\item For a fixed $\be$, let $\al_1,\al_2,\ldots,\al_s$ be the roots of \eqref{eqn:determinant_inner_pos} with $|\al_i| < |\be|, ~ i = 1,2,\ldots,s$ and let $\al_{s + 1}$ be the root of \eqref{eqn:determinant_inner_neg} with $|\al_{s + 1}| < |\be|$. Then, as $\be \downarrow 0$,
    \begin{enumerate}[label = \textup{(\alph*)}]%
    \item The ratio $\al_i/\be \to 1/\Lbe, ~ i = 1,2\ldots,s$ with $\Lbe > 1$, where $\Lbe$ is the root of \eqref{eqn:limiting_ratio_be_i/al_satisfies}.
    \item The eigenvector $\ibpos{\al_i,\be} \to \eb{0}, ~ i = 1,2\ldots,s$.
    \item The ratio $\al_{s + 1}/\be \to 1/\Lbet$ with $\Lbet > 1$, where $\Lbet$ is the root of \eqref{eqn:limiting_ratio_bet/al_satisfies}.
    \item The elements of the eigenvector $\ineg{\al_{s + 1},\be,r} \to \func{-}{0}^r, ~ r = 0,1,\ldots,s - 1$.
    \end{enumerate}%
\end{enumerate}%
\end{lemma}%

\begin{proof}%
(i)(a) Set $\Lrootpos = \be_i/\al$ and let $\al \downarrow 0$ in \eqref{eqn:determinant_inner_pos_rewritten} to establish \eqref{eqn:limiting_ratio_be_i/al_satisfies}. The roots of \eqref{eqn:limiting_ratio_be_i/al_satisfies} are defined in \eqref{eqn:limiting_roots_pos} and satisfy $0 < \Lal < 1$ and $\Lbe > 1$.

\noindent(i)(b) We have that $\ipos{\al,\be,0} = 1$ for all $\al$ and for $\al \downarrow 0$ the other elements in \eqref{eqn:inner_pos_eigenvector} go to $((1+s)(\rho + 1) - \Lal(1 + s)\rho - 1/\Lal)/s$, which is equal to zero by (i)(a).

\noindent(i)(c) Set $\Lrootneg = \be_{s + 1}/\al$, divide \eqref{eqn:determinant_inner_neg} by $\al^2$ and let $\al \downarrow 0$ to establish \eqref{eqn:limiting_ratio_bet/al_satisfies}. We note that $\func{\pm}{\cdot}$ is a continuous function, so that indeed for $\al \downarrow 0$, $\func{\pm}{\al \Lrootneg} \to \func{\pm}{0}$. Furthermore, \eqref{eqn:limiting_ratio_bet/al_satisfies} is a quadratic equation in $\Lrootneg$ with roots $|\Lalt| < 1$ and $|\Lbet| > 1$, where the inequalities follow from Lemma~\ref{lem:location_number_zeros_inner_neg}, and also satisfy $\Lalt,\Lbet \in \Real_+$.

\noindent(i)(d) Again, set $\Lrootneg = \be_{s + 1}/\al$. Recall the definition of the eigenvector in \eqref{eqn:inner_neg_eigenvector}. We aim at showing that for $\al \downarrow 0$, $\Func{\al,\be_{s + 1},\func{-}{\al \Lrootneg}} \to 0$ and $\Func{\al,\be_{s + 1},\func{+}{\al \Lrootneg}}$ tends to some non-zero constant. As $\al \downarrow 0$, we have that
\begin{equation}%
\Func{\al,\be_{s + 1},\funcparam} \to (1 + s)\rho \funcparam^{-1} \bigl( \Lalt \funcparam^s - 1 \bigr). \label{eqn:limiting_Func}
\end{equation}%
Note that $\func{-}{0} \in (0,1)$ and $\func{+}{0} > 1$. Since $\Lalt < 1$, the limiting value in \eqref{eqn:limiting_Func} can only equal zero in the case that $\Lalt = 1/\func{+}{0}^s$. Thus, we have established that $\Func{\al,\be_{s + 1},\func{-}{\al \Lrootneg}}$ tends to a non-zero constant as $\al \downarrow 0$. We next verify that $\Lalt = 1/\func{+}{0}^s$ is a solution to \eqref{eqn:limiting_ratio_bet/al_satisfies}, which proves that $\Func{\al,\be_{s + 1},\func{+}{\al \Lrootneg}} \to 0$ as $\al \downarrow 0$. Substituting $\Lrootneg = 1/\func{+}{0}^s$ in the left-hand side of \eqref{eqn:limiting_ratio_bet/al_satisfies} yields
\begin{align}%
&\frac{1}{\func{+}{0}^{2s}} ((1 + s)\rho)^s - \frac{1}{\func{+}{0}^s} s^s ( \func{+}{0}^s + \func{-}{0}^s ) + s^s, \notag \\
&= \frac{1}{\func{+}{0}^{2s}} \Bigl( ((1 + s)\rho)^s - s^s \func{+}{0}^s ( \func{+}{0}^s + \func{-}{0}^s ) + s^s \func{+}{0}^{2s} \Bigr), \notag \\
&= \frac{1}{\func{+}{0}^{2s}} \Bigl( ((1 + s)\rho)^s - s^s ( \func{+}{0} \func{-}{0} )^s \Bigr) = 0,
\end{align}%
where we used $\func{+}{0}\func{-}{0} = (1 + s)\rho/s$, as found in \eqref{eqn:inner_neg_funcparam_satisfies}.

\noindent(ii) The proof is identical to the proof of (i).
\end{proof}%

Finally, we describe the limiting behavior of the coefficients. In the following lemma we introduce the variable $\Lcoeffpos$ associated with a product-form solution that satisfies the positive inner equations, and the variable $\Lcoeffneg$ associated with a product-form solution that satisfies the negative inner equations.

\begin{lemma}\label{lem:convergence_coefficients}\hspace*{1em}%
\begin{enumerate}[label = \textup{(\roman*)}]%
\item Consider the setting of \textup{Lemma~\ref{lem:compensation_vertical}(i)}. Then, as $\be \downarrow 0$,
    \begin{equation}%
    \frac{\cv_1}{\ch} \to - \frac{1 - \Lal (1 + s)\rho}{1 - \Lbe (1 + s)\rho} \ifed \Lcv.
    \end{equation}%
\item Consider the setting of \textup{Lemma~\ref{lem:compensation_vertical}(ii)}. Then, as $\be \downarrow 0$,
    \begin{equation}%
    \frac{\cv_{s + 1}}{\ch} \to - \frac{s - \Lalt (1 + s) \rho \func{+}{0}^{s - 1}}{s - \Lbet(1 + s)\rho \func{-}{0}^{s - 1}} \ifed \Lcvt.
    \end{equation}%
\item Consider the setting of \textup{Lemma~\ref{lem:compensation_horizontal}(i)}. Then, as $\al \downarrow 0$,
    \begin{enumerate}[label = \textup{(\alph*)}]%
    \item %
        \begin{equation}%
        \frac{\hb}{\cv} \to \zerob.
        \end{equation}%
    \item %
        \begin{equation}%
        \frac{\ch_{s + 1}}{\cv} \to \frac{\Lal - \Lbe}{\Lal s \frac{1-q}{q} + \Lalt \func{+}{0}^{s - 1}} \ifed \Lchtpos.
        \end{equation}%
    \item %
        \begin{equation}%
        \lim_{\al \downarrow 0} \Bigl| \frac{\ch_i}{\cv} \Bigr| \le \max_{1 \le j,r+1 \le s} |a_j(r)| \ifed \Lchpos,
        \end{equation}%
        where $\bld{a}_j = (a_j(0),a_j(1),\ldots,a_j(s - 1))^T$ is the solution to the following linear system of equations for a specific $j$, with $\bld{b}_j$ a column vector of size $s$ with unknowns,
        \begin{align}%
        \Lal W \bld{a}_j + L \bld{b}_j &= - \Lbe \bld{w}_j - \Lchtpos \Lalt \func{+}{0}^{s - 1}\eb{0}, \\
        - W \bld{a}_j + (1 + s)\rho(1 - q) \M{0,s - 1} \bld{b}_j  &= \bld{w}_j
        \end{align}%
        with
        \begin{align}%
        W &\defi \begin{pmatrix}%
        1 & 1 & \cdots & 1 \\
        \Lal^{1/s} \unit_1 & \Lal^{1/s} \unit_2 & \cdots & \Lal^{1/s} \unit_s \\
        \vdots & \vdots & & \vdots \\
        \Lal^{(s - 1)/s} \unit_1^{s - 1} & \Lal^{(s - 1)/s} \unit_2^{s - 1} & \cdots & \Lal^{(s - 1)/s} \unit_s^{s - 1}
        \end{pmatrix}, \label{eqn:convergence_horizontal_definition_W} \\
        \bld{w}_j &\defi
        \begin{pmatrix}%
        1 & \Lbe^{1/s} \unit_j & \cdots & \Lbe^{(s - 1)/s} \unit_j^{s - 1}
        \end{pmatrix}^T. \label{eqn:convergence_horizontal_definition_w_j}%
        \end{align}%
    \end{enumerate}%
    \item Consider the setting of \textup{Lemma~\ref{lem:compensation_horizontal}(ii)}. Then, as $\al \downarrow 0$,
\begin{enumerate}[label = \textup{(\alph*)}]%
    \item %
        \begin{equation}%
        \frac{\hb}{\cv} \to \zerob.
        \end{equation}%
    \item %
        \begin{equation}%
        \frac{\ch_{s + 1}}{\cv} \to - \frac{\Lal s \frac{1 - q}{q} + \Lbet \func{-}{0}^{s - 1}}{\Lal s \frac{1 - q}{q} + \Lalt \func{+}{0}^{s - 1}} \ifed \Lchtneg.
        \end{equation}%
    \item %
        \begin{equation}%
        \lim_{\al \downarrow 0} \Bigl| \frac{\ch_i}{\cv} \Bigr| \le \max_{0 \le r \le s - 1} |c(r)| \ifed \Lchneg,
        \end{equation}%
        where $\bld{c} = (c(0),c(1),\ldots,c(s - 1))^T$ is the solution to, with $\bld{d}$ a column vector of size $s$ with unknowns,
        \begin{align}%
        \Lal W \bld{c} + \Lo \bld{d} &= -\bigl( \Lbet \func{-}{0}^{s - 1} + \Lchtneg \Lalt \func{+}{0}^{s - 1} \bigr) \eb{0}, \\
        - W \bld{c}+ (1 + s) \rho (1 - q) \M{0,s - 1} \bld{d} &= \zerob.
        \end{align}%
    \end{enumerate}%
\end{enumerate}%
\end{lemma}%

\begin{proof}%
(i) Divide both sides of \eqref{eqn:compensation_vertical_pos_coefficients} by $\ch$ and let $\al \downarrow 0$.

\noindent(ii) The proof is identical to the proof of (i).

\noindent(iii) Due to the length of this proof, the proof has been relegated to Appendix~\ref{app:proof_convergence_coefficients_horizontal}.

\noindent(iv) The proof is identical to the proof of (iii).
\end{proof}%

We have established the limiting behavior of ratios of compensation parameters, coefficients, eigenvectors and the vector on the horizontal axis. As we have seen in Lemma~\ref{lem:convergence_alpha_beta}, the limiting values of the eigenvectors $\ibpos{\al,\be}$ and $\ibneg{\al,\be}$ are finite. So, we can bound the absolute value of both of them by a constant not depending on $\al$ or $\be$. In doing so, one does not need to take into account the eigenvectors $\ibpos{\al,\be}$ and $\ibneg{\al,\be}$ to establish absolute convergence. Based on this observation and the asymptotic results derived above, we now show that the series appearing in \eqref{eqn:equilibrium_distribution_proportional} are absolutely convergent.

\begin{theorem}\label{thm:absolute_convergence_series}%
There exists a positive integer $N$ such that
\begin{enumerate}[label = \textup{(\roman*)}]%
\item The series
    \begin{equation}%
    \sum_{l = 0}^\infty \sum_{i = 1}^{(s+1)^l} |\al_{l,i}^m| \sum_{j = 1}^{s + 1} |\ch_{l,d(i) + j}| |\be_{l,d(i) + j}^{|n|}| \label{eqn:convergence_horizontal}
    \end{equation}%
    converges for all $m \ge 0, ~ |n| \ge 1$ with $m + |n| > N$.
\item The series
    \begin{equation}%
    \sum_{l = 0}^\infty \sum_{i = 1}^{(s + 1)^{l + 1}} |\cv_{l + 1,i}| |\al_{l + 1,i}^m| |\be_{l,i}^{|n|}| \label{eqn:convergence_vertical}
    \end{equation}%
    converges for all $m \ge 0, ~ |n| \ge 1$ with $m + |n| > N$.
\item The series
    \begin{equation}%
    \sum_{l = 0}^\infty \sum_{i = 1}^{(s+1)^l} |\al_{l,i}^m| |\hb_{l,i}| \label{eqn:convergence_boundary}
    \end{equation}%
    converges for all $m \ge N$.
\item The series
    \begin{equation}%
    \sum_{m + |n| > N} p(m,n,r), \quad r = 0,1,\ldots,s-1 \label{eqn:convergence_sum_over_all_states}
    \end{equation}%
    converges absolutely, where $p(m,n,r)$ is given in \eqref{eqn:equilibrium_distribution_proportional}.
\end{enumerate}%
\end{theorem}%

\begin{proof}%
(i) We can view the series as an infinite tree with $s + 1$ roots and each term has $s + 1$ children. We define the ratios of a term with each of it's descendants as
\begin{equation}%
R^{(1)}_{l,i,j \to k}(m,n) \defi \frac{|\ch_{l + 1,d(d(i) + j) + k}| |\al_{l + 1,d(i) + j}^m| |\be_{l + 1,d(d(i) + j) + k}^{|n|}|}{|\ch_{l,d(i) + j}| |\al_{l,i}^m| |\be_{l,d(i) + j}^{|n|}|}.
\end{equation}%
We multiply the ratio $R^{(1)}_{l,i,j \to k}(m,n)$ by $\Bigl| \frac{\cv_{l,d(i) + j}}{\cv_{l,d(i) + j}} \frac{\al_{l + 1,d(i) + j}^{|n|}}{\al_{l + 1,d(i) + j}^{|n|}} \frac{\al_{l,i}^{|n|}}{\al_{l,i}^{|n|}} \frac{\be_{l,d(i) + j}^{m + |n|}}{\be_{l,d(i) + j}^{m + |n|}} \Bigr|$, yielding
\begin{equation}%
R^{(1)}_{l,i,j \to k}(m,n) = \Bigl| \frac{\ch_{l + 1,d(d(i) + j) + k}}{\cv_{l,d(i) + j}} \frac{\cv_{l,d(i) + j}}{\ch_{l,d(i) + j}} \frac{\al_{l + 1,d(i) + j}^{m + |n|}}{\be_{l,d(i) + j}^{m + |n|}} \frac{\be_{l,d(i) + j}^{m + |n|}}{\al_{l,i}^{m + |n|}} \frac{\be_{l + 1,d(d(i) + j) + k}^{|n|}}{\al_{l + 1,d(i) + j}^{|n|}} \frac{\al_{l,i}^{|n|}}{\be_{l,d(i) + j}^{|n|}} \Bigr|.
\end{equation}%
As $l \to \infty$ we obtain by Lemmas~\ref{lem:convergence_alpha_beta} and \ref{lem:convergence_coefficients} that $\lim_{l \to \infty} R^{(1)}_{l,i,j \to k}(m,n) \le R^{(1)}_{j \to k}(m,n)$, where the inequality is element-wise, and
\begin{equation}%
R^{(1)}_{j \to k}(m,n) = \begin{cases}%
|\Lchpos| |\Lcv| |\Lal / \Lbe|^{m + |n|}, & j = 1,2,\ldots,s, ~ k = 1,2,\ldots,s, \\
|\Lchneg| |\Lcvt| |\Lal|^{|n|} |\Lalt/\Lbet|^m |1/\Lbet|^{|n|}, & j = s + 1, ~ k = 1,2,\ldots,s, \\
|\Lchtpos| |\Lcv| |\Lal/\Lbe|^m |1/\Lbe|^{|n|} |\Lalt|^{|n|}, & j = 1,2,\ldots,s, ~ k = s + 1, \\
|\Lchtneg| |\Lcvt| |\Lalt / \Lbet|^{m + |n|}, & j = s + 1, ~ k = s + 1.
\end{cases}%
\end{equation}%
The convergence of the series is determined by the spectral radius of the corresponding matrix of ratios $R^{(1)}(m,n) \defi (R^{(1)}_{j \to k}(m,n))_{j,k = 1,2,\ldots,s+1}$. If the spectral radius of the matrix $R^{(1)}(m,n)$ is less than one, the series converges. For more details, see e.g.~\cite[Section~9]{CompApproachAsymmetric_Adan1991} or \cite[proof of Theorem~5.2]{CompApproachErlangArrivals_Adan2013}. Observe that $|\Lal|,|1/\Lbe|,|\Lalt|,|1/\Lbet| < 1$, but the values of the limiting ratios $|\Lcv|,|\Lcvt|,|\Lchpos|,|\Lchneg|,|\Lchtpos|,|\Lchtneg|$ can be larger than one. The spectral radius of the matrix $R^{(1)}(m,n)$ depends on $m + |n|$ and thus we can find an integer $N$ with $m + n > N$ for which the spectral radius is less than one, ensuring the series converges for $m + |n| > N$.

\noindent(ii) Using a similar analysis as in (i), we define the ratio as
\begin{equation}%
R^{(2)}_{l,i,j \to k} (m,n) \defi \frac{|\cv_{l + 2,d(d(i) + j) + k}| |\al_{l + 2,d(d(i) + j) + k}^m| |\be_{l + 1,d(d(i) + j) + k}^{|n|}|}{|\cv_{l + 1,d(i) + j}| |\al_{l + 1,d(i) + j}^m| |\be_{l,d(i) + j}^{|n|}|}.
\end{equation}%
For $l \to \infty$ we have that $\lim_{l \to \infty} R^{(2)}_{l,i,j \to k}(m,n) \le R^{(2)}_{j \to k}(m,n)$ and
\begin{equation}%
\hspace*{-0.7em}R^{(2)}_{j \to k}(m,n) = \begin{cases}%
|\Lchpos| |\Lcv| |\Lal/\Lbe|^{m + |n|}, & j = 1,2,\ldots,s, ~ k = 1,2,\ldots,s, \\
|\Lchneg| |\Lcv| |\Lal|^{|n|} |\Lal/\Lbe|^m |1/\Lbet|^{|n|}, & j = s + 1, ~ k = 1,2,\ldots,s, \\
|\Lchtpos| |\Lcvt| |1/\Lbe|^{|n|} |\Lalt|^{|n|} |\Lalt/\Lbet|^m, & j = 1,2,\ldots,s, ~ k = s + 1, \\
|\Lchtneg| |\Lcvt| |\Lalt/\Lbet|^{m + |n|}, & j = s + 1, ~ k = s + 1.
\end{cases}%
\end{equation}%
The spectral radius of the matrix $R^{(2)}(m,n) \defi (R^{(2)}_{j \to k}(m,n))_{j,k = 1,2,\ldots,s + 1}$ depends on $m + |n|$ and thus we can find an integer $N$ with $m + n > N$ for which the spectral radius is less than one.

\noindent(iii) We can rewrite \eqref{eqn:convergence_boundary} as
\begin{equation}%
\sum_{l = 0}^\infty \sum_{i = 1}^{(s + 1)^l} |\al_{l,i}^m| |\hb_{l,i}|  = |\al_{0,1}^m| |\hb_{0,1}| + \sum_{l = 1}^\infty \sum_{i = 1}^{(s + 1)^l} |\cv_{l,i}| |\al_{l,i}^m| \frac{|\hb_{l,i}|}{|\cv_{l,i}|}.
\end{equation}%
Lemmas~\ref{lem:convergence_coefficients}(iii)(a) and (iv)(a) show that $|\hb_{l,i}|/|\cv_{l,i}| \to \zerob$ and thus we can bound it from above and need not consider it when proving convergence of the series. Proving convergence of
\begin{equation}%
\sum_{l = 0}^\infty \sum_{i = 1}^{(s + 1)^{l + 1}} |\cv_{l + 1,i}| |\al_{l + 1,i}^m|
\end{equation}%
establishes convergence of \eqref{eqn:convergence_boundary}. Exploiting the similarity with the series \eqref{eqn:convergence_vertical}, we define the ratio $R^{(3)}_{l,i,j \to k} (m) \defi R^{(2)}_{l,i,j \to k}(m,0)$. Hence, the limiting ratios can be bound from above: $\lim_{l \to \infty} R^{(3)}_{l,i,j \to k} (m) \le R^{(3)}_{j \to k} (m) = R^{(2)}_{j \to k} (m,0)$. The spectral radius of the matrix $R^{(3)}(m) \defi (R^{(3)}_{j \to k}(m))_{j,k = 1,2,\ldots,s + 1}$ depends on $m$ and thus we can find an integer $N$ with $m \ge N$ for which the spectral radius is less than one.

\noindent(iv) This follows straightforwardly, along the same lines as in \cite[Section~9]{CompApproachAsymmetric_Adan1991}.
\end{proof}%

\begin{remark}\label{rem:choosing_N}%
We note that the series in Theorem~\ref{thm:absolute_convergence_series}(i) (without the absolute values) corresponds to the sum of the first series in \eqref{eqn:equilibrium_distribution_proportional_pos} and the first series in \eqref{eqn:equilibrium_distribution_proportional_neg}. The series in Theorem~\ref{thm:absolute_convergence_series}(ii) (without the absolute values) corresponds to the sum of the second series in \eqref{eqn:equilibrium_distribution_proportional_pos} and the second series in \eqref{eqn:equilibrium_distribution_proportional_neg}. So, proving convergence of the series in Theorem~\ref{thm:absolute_convergence_series} establishes convergence of the series in \eqref{eqn:equilibrium_distribution_proportional}.
\end{remark}%

\begin{remark}%
The index $N$ is the \textit{minimal non-negative integer} for which the spectral radii of the matrices $R^{(1)}(m,n)$ and $R^{(2)}(m,n)$ are both less than one for $m + |n| > N$ and the spectral radius of $R^{(3)}(m)$ is less than one for $m \ge N$. So, for all states $(m,n)$ with $m + |n| > N$ and $(N,0)$ the series in \eqref{eqn:equilibrium_distribution_proportional} converges. In general, the index $N$ is small. In Table~\ref{tbl:numerical_N} we list the index $N$ for fixed $q$, while varying the values of $s$ and $\rho$. Note that the area of convergence might be bigger than the one based on $N$, since $N$ is based on $R^{(1)}(m,n)$, $R^{(2)}(m,n)$ and $R^{(3)}(m)$ which are \textit{upper bounds} on the rate of convergence.

\begin{table}%
\centering%
\begin{tabular}{c|ccccc|ccccc}%
$s$    & 2   &     &     &     &     & 5   &     &     &     &     \\
$\rho$ & 0.1 & 0.3 & 0.5 & 0.7 & 0.9 & 0.1 & 0.3 & 0.5 & 0.7 & 0.9 \\
\hline
$N$    & 1   & 1   & 1   & 1   & 1   & 1   & 1   & 1   & 1   & 1
\end{tabular}%
\caption{The index $N$ for fixed $q = 0.4$ and varying $s$ and $\rho$.}%
\label{tbl:numerical_N}%
\end{table}%
\end{remark}%

We are now in the position to formulate the main result of the paper.

\begin{theorem}\label{thm:equilibrium_distribution}%
For all states $(m,n,r), ~ m \in \Nat_0, ~ n \in \Int, ~ r = 0,1,\ldots,s - 1$ and $m + |n| > N$ including $m = N, ~ n = 0$, the equilibrium probabilities $\pb{m,n}$ are proportional, up to a multiplicative constant $C$, to the expressions in \eqref{eqn:equilibrium_distribution_proportional}, see \eqref{eqn:equilibrium_distribution_introduction}, where $C$ is the normalization constant of the equilibrium distribution. The remaining $\pb{m,n}, ~ m + |n| \le N$ are determined by the finite system of equilibrium equations for the states $(m,n)$ with $m + |n| \le N$, where $N$ is the minimal non-negative integer for which the spectral radii of the matrices $R^{(1)}(m,n)$ and $R^{(2)}(m,n)$ are both less than one for $m + |n| > N$ and the spectral radius of $R^{(3)}(m)$ is less than one for $m \ge N$.
\end{theorem}%

\begin{proof}%
This proof is similar to the proof in \cite[proof of Theorem~5.3]{CompApproachErlangArrivals_Adan2013} and \cite[Section~11]{CompApproachAsymmetric_Adan1991}, but we include it here for completeness. Define $\mathcal{L}_{N} = \{ (m,n) \mid m \in \Nat_0, ~ n \in \Int, m + |n| > N \} \cup \{(N,0,s - 1)\}$ and note the similarity with the set of states defined in the proof of Lemma~\ref{lem:modified_model_product_form} and the associated transition rate diagram in Figure~\ref{fig:trd_modified_model}. Then $\mathcal{L}_{N}$ is a set of states for which the series in \eqref{eqn:equilibrium_distribution_proportional} converge absolutely. The restricted stochastic process on the set $\mathcal{L}_{N}$ is an irreducible Markov process, whose associated equilibrium equations are identical to the equations of the original unrestricted process on the set $\mathcal{L}_{N}$, expect for the equilibrium equation of state $(N,0,s - 1)$. Hence, the process restricted to $\mathcal{L}_{N}$ is ergodic so that the series in \eqref{eqn:equilibrium_distribution_proportional} can be normalized to produce the equilibrium distribution of the restricted process on $\mathcal{L}_{N}$. Since the set $\Nat_0 \times \Int \times \{0,1,\ldots,s - 1\} \setminus \mathcal{L}_{N}$ is finite, it follows that the original process is ergodic and relating appropriately the equilibrium probabilities of the unrestricted and restricted process completes the proof.
\end{proof}%

The following remark is in line with Remark~\ref{rem:no_curse_of_dimensionality}.

\begin{remark}[An efficient numerical scheme]%
\begin{figure}%
\centering%
\includegraphics{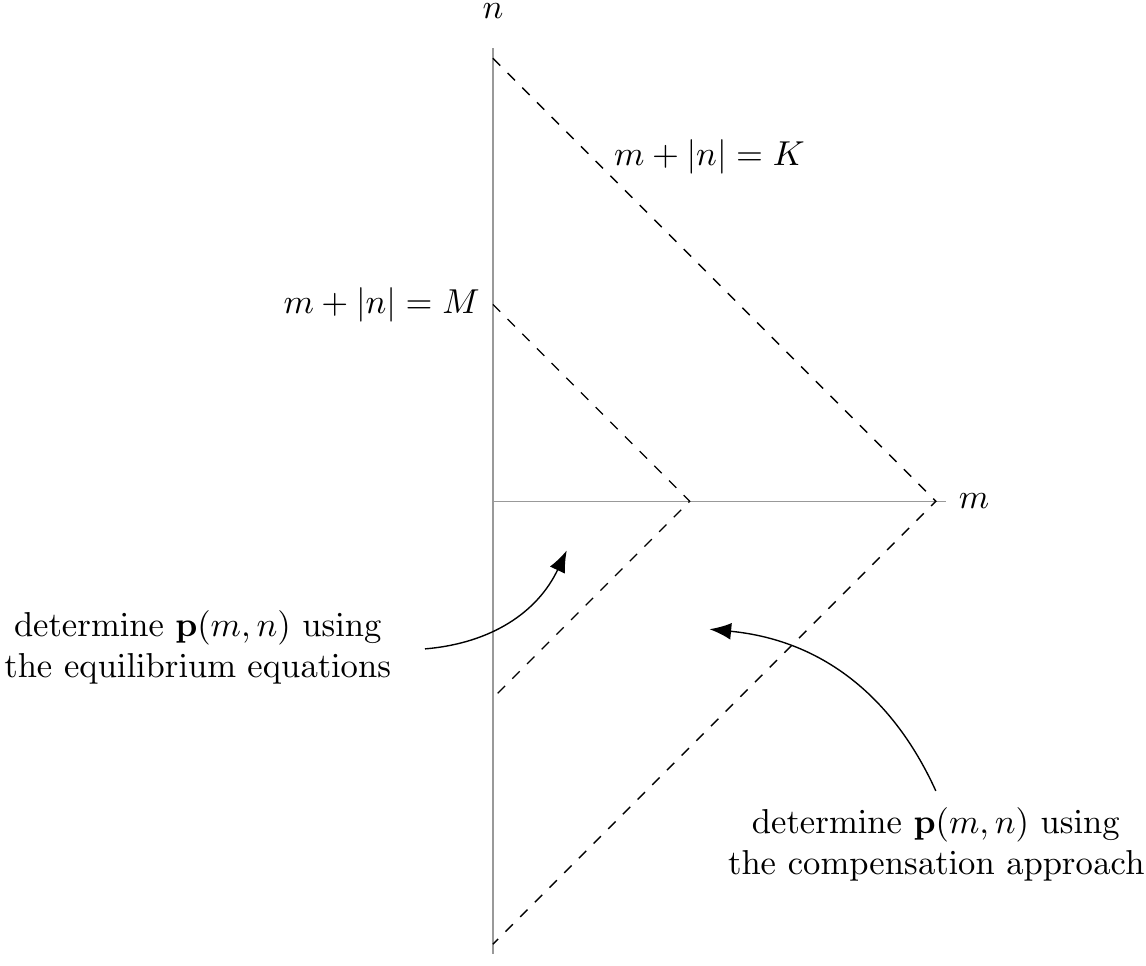} % PDF figure
\caption{Regions of the efficient numerical scheme.}%
\label{fig:numerical_scheme_choices_integers}%
\end{figure}%

The following scheme exploits the fact that the rate of convergence of the series in \eqref{eqn:equilibrium_distribution_proportional} increases as $m + |n|$ increases. So, further away from the origin, fewer compensation steps $L$ are needed to achieve the same accuracy according to \eqref{eqn:equilibrium_distribution_numerical} and \eqref{eqn:accuracy_condition_numerical}. This property was seen in Section~\ref{sec:numerical_results} and in particular Figure~\ref{fig:no_compensation_steps_needed}. Denote a triangular set of states as $\mathcal{T}_x \defi \{ (m,n) \mid m \in \Nat_0, ~ n \in \Int, ~ m + |n| \le x \}$. Figure~\ref{fig:numerical_scheme_choices_integers} serves as a visual aid.
\begin{enumerate}[label = (\roman*)]%
\item Determine the minimal non-negative integer $N$ for which the spectral radii of the matrices $R^{(1)}(m,n)$ and $R^{(2)}(m,n)$ are both less than one for $m + |n| > N$ and the spectral radius of $R^{(3)}(m)$ is less than one for $m \ge N$.
\item Select integers $M$ and $K$ such that $N < M < K$.
\item Determine $\pb{m,n}, ~ (m,n) \in \mathcal{T}_K \setminus \mathcal{T}_M$ according to \eqref{eqn:equilibrium_distribution_numerical} with $C = 1$ and $L$ the minimal integer such that \eqref{eqn:accuracy_condition_numerical} holds.
\item Determine $\pb{m,n}, ~ (m,n) \in \mathcal{T}_M$ from the equilibrium equations for the states $(m,n) \in \mathcal{T}_M$.
\item Normalize the equilibrium distribution by dividing each equilibrium probability by the sum $\sum_{(m,n) \in \mathcal{T}_K} \pb{m,n} \oneb$.
\end{enumerate}%
The integer $L$ in step (iii) depends on $M$. As $M$ increases, $L$ decreases or stays constant, but the size of the system of equilibrium equations in step (iv) increases. This tradeoff is clearly in favor of selecting a larger $M$: decreasing $L$ decreases the number of computation steps exponentially, whereas the size of the system of equilibrium equations increases polynomially with $M$. As $K$ increases, the equilibrium probabilities become more accurate. $K$ can be chosen arbitrarily large; it has little to no impact on the performance.
\end{remark}%

%%%%%%%%%%%%%%%%%%%%%%%%%%%%%%%%%%%%%%%%%%%%%%%%%%%%%%%
%%%%%%%%%%%%%%%%%%%%%%%%%%%%%%%%%%%%%%%%%%%%%%%%%%%%%%%
%%%%%%%%%%%%%%%%%%%%% NEW SECTION %%%%%%%%%%%%%%%%%%%%%
%%%%%%%%%%%%%%%%%%%%%%%%%%%%%%%%%%%%%%%%%%%%%%%%%%%%%%%
%%%%%%%%%%%%%%%%%%%%%%%%%%%%%%%%%%%%%%%%%%%%%%%%%%%%%%%

\section{Conclusion}%
\label{sec:conclusion}%

We have studied a queueing system with two non-identical servers with service rates $1$ and $s \in \Nat$, respectively, Poisson arrivals and the SED routing policy. This policy assigns an arriving customer to the queue with the smallest expected delay, i.e.~waiting time plus service time. The SED routing policy is a natural and simple routing policy that balances the load for the two non-identical servers. Although not always optimal, SED performs well at both ends of system utilization range. Moreover, SED routing is asymptotically optimal in the heavy traffic regime.

The SED system can be modeled as an inhomogeneous random walk in the quadrant. By appropriately transforming the state space, we mapped the two-dimensional state space into a half-plane with a finite third dimension. The random walks on each quadrant are different, yet homogeneous inside each quadrant. Extending the compensation approach to this three-dimensional setting, we showed in this paper that the equilibrium distribution of the joint queue length can be represented as a series of product-form solutions. These product-form solutions are generated iteratively to compensate for the error introduced by its preceding product-form term.

The analysis presented in this paper proves that the compensation approach can be applied in the context of a three-dimensional state space. We believe that a similar analysis can be used to investigate general conditions for applicability of the compensation approach for three-dimensional Markov processes. These conditions will be comparable to the three conditions in Section~\ref{sec:introduction}. Furthermore, the compensation approach can possibly be extended to the case where all three dimensions are infinite, paving the way for performance analysis of higher-dimensional Markov processes.

The insights gained for the SED system with two servers, specifically, the series expressions for the equilibrium probabilities, can be used to develop approximations of the performance of heterogeneous multi-server systems with a SED routing protocol. These approximations can be derived along the same lines as in \cite{JSQ_WebServerFarms_Gupta2007}. An approximate performance analysis for a system with two servers can be found in \cite{Selen2016_GJSQ}.

Another interesting direction for future research is to study rare events or tail probabilities in the SED system, in a similar way as done for JSQ systems in \cite{GeometricDecay_Li2007}. Since the compensation approach determines the complete expansion of each equilibrium probability, one can approximate rare events with arbitrary precision.

%%%%%%%%%%%%%%%%%%%%%%%%%%%%%%%%%%%%%%%%%%%%%%%%%%%%%%%
%%%%%%%%%%%%%%%%%%%%%%%%%%%%%%%%%%%%%%%%%%%%%%%%%%%%%%%
%%%%%%%%%%%%%%%%%%% NEW SUBSECTION %%%%%%%%%%%%%%%%%%%%
%%%%%%%%%%%%%%%%%%%%%%%%%%%%%%%%%%%%%%%%%%%%%%%%%%%%%%%
%%%%%%%%%%%%%%%%%%%%%%%%%%%%%%%%%%%%%%%%%%%%%%%%%%%%%%%

\subsection*{Acknowledgements}%
\label{subsec:acknowledgement}%

The authors would like to thank A.J.E.M.~Janssen for proving step (c) of Lemma~\ref{lem:location_number_zeros_inner_pos}(i). This work was supported by an NWO free competition grant, an ERC starting grant and the NWO Gravitation Project NETWORKS.

%%%%%%%%%%%%%%%%%%%%%%%%%%%%%%%%%%%%%%%%%%%%%%%%%%%%%%%
%%%%%%%%%%%%%%%%%%%%%%%%%%%%%%%%%%%%%%%%%%%%%%%%%%%%%%%
%%%%%%%%%%%%%%%%%%%%%% APPENDIX %%%%%%%%%%%%%%%%%%%%%%%
%%%%%%%%%%%%%%%%%%%%%%%%%%%%%%%%%%%%%%%%%%%%%%%%%%%%%%%
%%%%%%%%%%%%%%%%%%%%%%%%%%%%%%%%%%%%%%%%%%%%%%%%%%%%%%%

\appendix%

%%%%%%%%%%%%%%%%%%%%%%%%%%%%%%%%%%%%%%%%%%%%%%%%%%%%%%%
%%%%%%%%%%%%%%%%%%%%%%%%%%%%%%%%%%%%%%%%%%%%%%%%%%%%%%%
%%%%%%%%%%%%%%%%%%%%% NEW SECTION %%%%%%%%%%%%%%%%%%%%%
%%%%%%%%%%%%%%%%%%%%%%%%%%%%%%%%%%%%%%%%%%%%%%%%%%%%%%%
%%%%%%%%%%%%%%%%%%%%%%%%%%%%%%%%%%%%%%%%%%%%%%%%%%%%%%%

\section{Proof of step (c) of Lemma~\ref{lem:location_number_zeros_inner_pos}(i)}%
\label{app:proof_single_root_beta_i}%

The following argument was communicated to us by A.J.E.M.~Janssen. Our goal is to show that for every $\al$ with $|\al| \in (0,1)$ the equation
\begin{equation}%
\frac{\al\be(1 + s)(\rho + 1) - \be^2(1 + s)\rho - \al^2}{\al\be s} = \unit_i \be^{1/s}, \quad i = 1,2,\ldots,s, \label{eqn:app_proof_single_beta_i_original}
\end{equation}%
has at least one root $\be_i$ with $|\be_i| < |\al|$ for each $i$. To that end, we consider the more general formulation
\begin{equation}%
\sigma = f(z) = \frac{E}{z^t} (\Lbe - z) (z - \Lal) \label{eqn:app_proof_single_beta_i_general_formulation}
\end{equation}%
with $z = \be/\al$, $E = \frac{(1 + s)\rho}{s} > 0$, $t = \frac{s + 1}{s} \in (1,2]$ and $\Lrootpos_{\scriptscriptstyle \pm}$ is defined in \eqref{eqn:limiting_roots_pos} with the properties
\begin{equation}%
0 < \Lal < 1, \quad \Lbe > 1, \quad \Lbe + \Lal = \frac{1 + \rho}{\rho}, \quad \Lbe \Lal = \frac{1}{(1 + s)\rho}.
\end{equation}%
One retrieves the original form \eqref{eqn:app_proof_single_beta_i_original} by setting $\sigma = \unit_i \al^{1/s}$.

The proof consists of three main steps:
\begin{enumerate}[label = (\alph*)]%
\item We derive the inverse function of $\sigma = f(z)$ on a neighborhood of $z = \Lal$ using the Lagrange inversion theorem \cite[Section~2.2]{LagrangeInversionTheorem_Bruijn1970}. We establish that $z = g(\sigma)$ where $g(\sigma)$ is a power series that is analytic on a neighbourhood of $\sigma = 0$ and has positive power series coefficients.
\item Employing Pringsheim's theorem \cite[Theorem~1]{PringsheimTheorem_Bateman2000}, we are able to extend the radius of convergence for the power series $g(\sigma)$ to $\sigma_{\textup{max}} > 1$.
\item Finally, we establish that for $\sigma = \unit_i \alpha^{1/s}, ~ i = 1,2,\ldots,s$ satisfying \eqref{eqn:app_proof_single_beta_i_general_formulation}, the corresponding $z$ has $|z| < 1$.
\end{enumerate}%

\noindent (a) We note that $f(z)$ is analytic near $z = \Lal$. So, by the Lagrange inversion theorem, we have in a neighborhood of $\sigma = f(\Lal) = 0$
\begin{align}%
z = g(\sigma) &= \Lal + \sum_{n = 1}^\infty \frac{\sigma^n}{n!} \frac{\dinf^{n - 1}}{\dinf z^{n - 1}} \Bigl( \frac{z - \Lal}{\frac{E}{z^t}(\Lbe - z) (z - \Lal)} \Bigr)^n \bigg\vert_{z = \Lal}  \notag \\
&= \Lal + \sum_{n = 1}^\infty \frac{(\sigma/E)^n}{n!} \frac{\dinf^{n - 1}}{\dinf z^{n - 1}} \frac{z^{n t}}{(\Lbe - z)^n} \bigg\vert_{z = \Lal}.
\end{align}%
Note that $g(\cdot)$ is analytic in a neighborhood of $\sigma = f(\Lal) = 0$. Let $g_1(z) = z^{nt}$ and $g_2(z) = (\Lbe - z)^{-n}$ so that we can write
\begin{equation}%
\frac{\dinf^{n - 1}}{\dinf z^{n - 1}} g_1(z) g_2(z) = \sum_{k = 0}^{n - 1} \binom{n - 1}{k} \frac{\dinf^{n - k - 1}}{\dinf z^{n - k - 1}} g_1(z) \frac{\dinf^k}{\dinf z^k} g_2(z),
\end{equation}%
where
\begin{align}%
\frac{\dinf^{n - k - 1}}{\dinf z^{n - k - 1}} g_1(z) &= \Bigl( \prod_{i = 0}^{n - k - 2} (nt - i) \Bigr) z^{nt - (n - k - 1)}, \\
\frac{\dinf^k}{\dinf z^k} g_2(z) &= \frac{(n + k - 1)!}{(n - 1)!} (\Lbe - z)^{-(n + k)}.
\end{align}%
Thus, we can conclude that
\begin{equation}%
\frac{\dinf^{n - 1}}{\dinf z^{n - 1}} \frac{z^{n t}}{(\Lbe - z)^n} \bigg\vert_{z = \Lal} = \sum_{k = 0}^{n - 1} \frac{(n + k - 1)!}{k! (n - 1 - k)!} \Bigl( \prod_{i = 0}^{n - k - 2} (nt - i) \Bigr) \frac{\Lal^{nt - (n - k - 1)}}{(\Lbe - \Lal)^{n + k}} > 0.
\end{equation}%
Hence, the power series $g(\sigma)$ has positive power series coefficients.

\noindent (b) We now extend the convergence range of $g(\sigma)$ from a neighbourhood of 0 to the entire unit disc. This is achieved by using the Pringsheim theorem. We are allowed to do so because the coefficients of the power series of $g(\sigma)$ are positive. According to the Pringsheim theorem we can limit attention to the positive real line. With reference to Figure~\ref{fig:app_proof_single_beta_i_graph_f(z)}, we let $\sigma_{\textup{max}} = \max \{ f(z) \mid z \ge \Lal \}$. Note that for $z \in [\Lal,g(\sigma_{\textup{max}}))$ the function $f(z)$ is strictly increasing and is analytic by \eqref{eqn:app_proof_single_beta_i_general_formulation}, thus its inverse, $g(\sigma)$, is analytic for $\sigma \in [f(\Lal),f(g(\sigma_{\textup{max}}))) = [0,\sigma_{\textup{max}})$. Clearly, $\sigma_{\textup{max}} > 1$, hence by the Pringsheim theorem $g(\sigma)$ is analytic for $|\sigma| < \sigma_{\textup{max}}$ and thus inside the unit disc.

\begin{figure}%
\centering%
\includegraphics{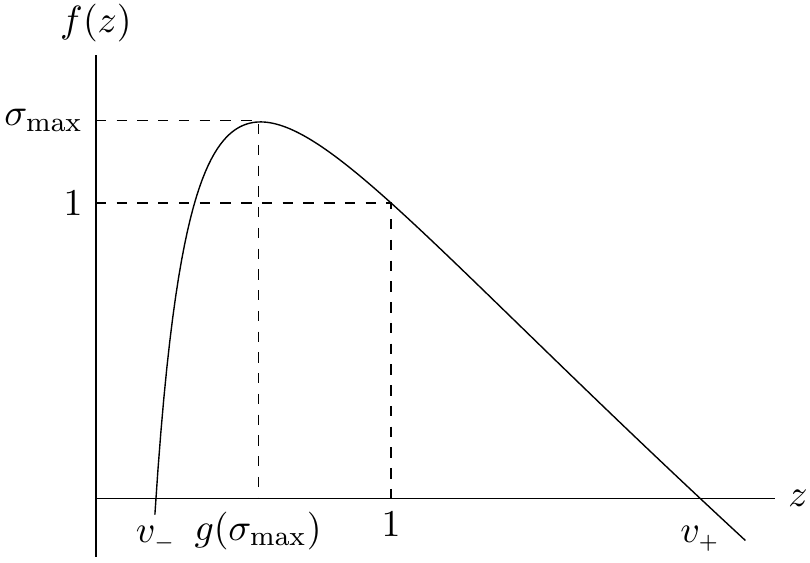}% PDF figure
\caption{The function $\sigma = f(z)$ for $z > 0$ with $\rho = 0.8$ and $s = 2$.}
\label{fig:app_proof_single_beta_i_graph_f(z)}
\end{figure}%

\noindent (c) It is important to see where $g(\sigma_{\textup{max}})$ lies. We find $z = g(\sigma_{\textup{max}})$ from
\begin{align}%
0 &= \frac{\dinf}{\dinf z} \Bigl( \frac{1}{z^t}(\Lbe - z)(z - \Lal) \Bigr) \notag \\
&= \frac{1}{z^{t + 1}} \Bigl( \Lbe \Lal t - (t - 1)(\Lbe + \Lal) z - (2 - t)z^2 \Bigr) \notag \\
&= \frac{1}{z^{t + 1}} \Bigl( \frac{t}{(1 + s)\rho} - (t - 1) \frac{1 + \rho}{\rho} z - (2 - t) z^2 \Bigr). \label{eqn:app_proof_single_beta_i_maximum}
\end{align}%
For $t \in (1,2)$ we take the positive solution of the quadratic equation in \eqref{eqn:app_proof_single_beta_i_maximum} and obtain after some rewriting
\begin{equation}%
z = g(\sigma_{\textup{max}}) = \frac{2}{1 + \rho + \sqrt{(1 + \rho)^2 + 4\rho(s - 1)}} < 1,
\end{equation}%
which also holds for $t = 2$. So, for a $\sigma$ with $|\sigma| < \sigma_{\textup{max}}$, we have from the positivity of the power series coefficients that
\begin{equation}%
|g(\sigma)| \le g(|\sigma|) \le g(\sigma_{\textup{max}}) < 1.
\end{equation}%
By taking $\sigma_i = \unit_i \alpha^{1/s}, ~ i = 1,2,\ldots,s$ with $|\sigma_i| < 1$ we see that the power series $g(\sigma_i)$ converges and $|z| = |g(\sigma_i)| < 1$. Finally, we can conclude that for every $\al$ with $|\al| \in (0,1)$, \eqref{eqn:app_proof_single_beta_i_original} has at least one root $\be_i$ with $|\be_i| < |\al|$ for each $i$.

%%%%%%%%%%%%%%%%%%%%%%%%%%%%%%%%%%%%%%%%%%%%%%%%%%%%%%%
%%%%%%%%%%%%%%%%%%%%%%%%%%%%%%%%%%%%%%%%%%%%%%%%%%%%%%%
%%%%%%%%%%%%%%%%%%%%% NEW SECTION %%%%%%%%%%%%%%%%%%%%%
%%%%%%%%%%%%%%%%%%%%%%%%%%%%%%%%%%%%%%%%%%%%%%%%%%%%%%%
%%%%%%%%%%%%%%%%%%%%%%%%%%%%%%%%%%%%%%%%%%%%%%%%%%%%%%%

\section{Proof of Lemma~\ref{lem:convergence_coefficients}(iii)}%
\label{app:proof_convergence_coefficients_horizontal}%

In this appendix we prove the convergence of the ratio of coefficients used in the horizontal compensation step for a solution that satisfies the positive inner equations. Specifically, we consider the setting of Lemma~\ref{lem:compensation_vertical}(i). We wish to establish the limiting values, or an upper bound on the absolute value of the following ratios, as $\al \downarrow 0$,
\begin{enumerate}[label = (\alph*)]%
\item $\hb/\cv$,
\item $\ch_{s + 1}/\cv$,
\item $\ch_i/\cv, ~ i = 1,2,\ldots,s$.
\end{enumerate}%
We reiterate below the equations that determine these ratios, which can be found in the main body of the paper in \eqref{eqn:compensation_horizontal_pos_h}-\eqref{eqn:compensation_horizontal_pos_eta_tilde},
\begin{align}%
\bigl(\A{0,1} + \al \A{-1,1} \bigr) \hb - \al \sum_{i = 1}^s \ch_i \ibpos{\al,\be_i} &= \cv \al \ibpos{\al,\be}, \label{eqn:app_compensation_horizontal_pos_h}\\
\al \B{0,0} \hb + \sum_{i = 1}^s \ch_i \be_i \bigl( \A{1,-1} + \al \A{0,-1} \bigr) \ibpos{\al,\be_i} \notag \\
+ \ch_{s + 1} \be_{s + 1} \bigl( \B{1,1} + \al \B{0,1} \bigr) \ibneg{\al,\be_{s + 1}} &= - \cv \be \bigl( \A{1,-1} + \al \A{0,-1} \bigr) \ibpos{\al,\be}, \label{eqn:app_compensation_horizontal_pos_eta_i}\\
- \ch_{s + 1} \al s + \al s \h{0} + (1 + s)\rho q \h{s-1} &= 0. \label{eqn:app_compensation_horizontal_pos_eta_tilde}
\end{align}%
As it turns out, when we let $\al \downarrow 0$ in the system of linear equations \eqref{eqn:app_compensation_horizontal_pos_h}-\eqref{eqn:app_compensation_horizontal_pos_eta_tilde}, we obtain a degenerate system of equations. By properly scaling the system, one obtains a non-degenerate system of equations.

In this appendix we adopt the following notation to indicate the limiting value of a ratio
\begin{equation}%
x^\textup{lim} \defi \lim_{\al \downarrow 0} \frac{x}{\cv}.
\end{equation}%

\noindent(a) Divide \eqref{eqn:app_compensation_horizontal_pos_eta_tilde} by $\cv$ and rearrange to obtain
\begin{equation}%
\frac{\h{s-1}}{\cv} = \al \frac{s}{q} \frac{1}{(1 + s)\rho} \Bigl( \frac{\ch_{s + 1}}{\cv} - \frac{\h{0}}{\cv} \Bigr). \label{eqn:app_proof_h(s-1)}
\end{equation}%
By letting $\al \downarrow 0$ in \eqref{eqn:app_proof_h(s-1)} we get $h^\textup{lim}(s - 1) = 0$. Next, divide \eqref{eqn:app_compensation_horizontal_pos_eta_i} by $\al$ and $\cv$, which yields
\begin{align}%
\B{0,0} \frac{\hb}{\cv} + \sum_{i = 1}^s \frac{\ch_i}{\cv} \frac{\be_i}{\al} \bigl( \A{1,-1} + \al \A{0,-1} \bigr) \ibpos{\al,\be_i} \notag \\
+ \frac{\ch_{s + 1}}{\cv} \frac{\be_{s + 1}}{\al} \bigl( \B{1,1} + \al \B{0,1} \bigr) \ibneg{\al,\be_{s + 1}} &= - \frac{\be}{\al} \bigl( \A{1,-1} + \al \A{0,-1} \bigr) \ibpos{\al,\be}. \label{eqn:app_proof_eta_i}
\end{align}%
We let $\al \downarrow 0$ and use that $\A{1,-1} = (1 + s)\rho \I$ and $\B{1,1} = (1 + s)\rho \M{0,s - 1}$. This gives
\begin{equation}%
\B{0,0} \hb^\textup{lim} + \Lal (1 + s)\rho \sum_{i = 1}^s \ch_i^\textup{lim} \eb{0} + \ch_{s + 1}^\textup{lim} \Lalt (1 + s)\rho \func{+}{0}^{s - 1} \eb{0} = - \Lbe (1 + s)\rho \eb{0}. \label{eqn:app_proof_eta_i_limit}
\end{equation}%
The matrix $\B{0,0}$ is a tri-diagonal matrix and together with $h^\textup{lim}(s - 1) = 0$ we find that $\hb^\textup{lim} = 0$.

\noindent(b) Together with $\hb^\textup{lim} = 0$, equation \eqref{eqn:app_proof_eta_i_limit} reduces to the single equation
\begin{equation}%
\Lal \sum_{i = 1}^s \ch_i^\textup{lim} + \ch_{s + 1}^\textup{lim} \Lalt \func{+}{0}^{s - 1} = - \Lbe. \label{eqn:app_proof_eta_i_limit_single_equation}
\end{equation}%
Now, divide \eqref{eqn:app_compensation_horizontal_pos_h} by $\al$ and $\cv$, yielding
\begin{equation}%
\A{0,1} \frac{\hb}{\al \cv} + \A{-1,1} \frac{\hb}{\cv} - \sum_{i = 1}^s \ch_i \ibpos{\al,\be_i} = \ibpos{\al,\be}. \label{eqn:app_proof_h}
\end{equation}%
Note that $\A{0,1} \hb/(\al \cv) = (1 + s)\rho (1 - q) \h{s - 1}/(\al \cv) \eb{0}$, so that, together with \eqref{eqn:app_proof_h(s-1)} and $\A{-1,1} = \I$, \eqref{eqn:app_proof_h} reduces to
\begin{equation}%
s \frac{1 - q}{q} \Bigl( \frac{\ch_{s + 1}}{\cv} - \frac{\h{0}}{\cv} \Bigr)\eb{0} + \frac{\hb}{\cv} - \sum_{i = 1}^s \ch_i \ibpos{\al,\be_i} = \ibpos{\al,\be}. \label{eqn:app_proof_h_reduced}
\end{equation}%
Moreover, letting $\al \downarrow 0$ again gives us a single equation, namely
\begin{equation}%
- \sum_{i = 1}^s \ch_i^\textup{lim} + s \frac{1 - q}{q}  \ch_{s + 1}^\textup{lim} = 1. \label{eqn:app_proof_h_limit_single_equation}
\end{equation}%
One can solve the linear system of equations \eqref{eqn:app_proof_eta_i_limit_single_equation} and \eqref{eqn:app_proof_h_limit_single_equation} for the two unknowns $\sum_{i = 1}^s \ch_i^\textup{lim}$ and $\ch_{s + 1}^\textup{lim}$, where the solution of $\ch_{s + 1}^\textup{lim}$ is given in Lemma~\ref{lem:convergence_coefficients}(iii)(b).

\noindent(c) We wish to establish a scaling of each equation in \eqref{eqn:app_compensation_horizontal_pos_h}-\eqref{eqn:app_compensation_horizontal_pos_eta_tilde} so that, in the limit for $\al \downarrow 0$, we obtain a non-degenerate system of equations. From \eqref{eqn:inner_pos_eigenvector} and \eqref{eqn:determinant_inner_pos_s-th_root} we know that $\ipos{\al,\be_i,r} = \unit_i \be_i^{r/s}, r = 0,1,\ldots,s - 1$ and from Remark~\ref{rem:identical_eigenvectors} it is clear that there exists a $j$ for which $\ipos{\al,\be,r} = \unit_j \be^{r/s}, ~ r = 0,1,\ldots,s-1$. Thus, the correct scaling to establish a non-degenerate limit of a term $\ipos{\al,\be,r}$ is $\al^{-r/s}$.

We are going to multiply both \eqref{eqn:app_proof_eta_i} and \eqref{eqn:app_proof_h} element-wise by the column vector
\begin{equation}%
\bm{\al} \defi \begin{pmatrix} 1 & \al^{-1/s} & \al^{-2/s} & \cdots & \al^{-(s - 1)/s} \end{pmatrix}^T.
\end{equation}%
To be able to do that, we need some new notation. Let us introduce the column vectors
\begin{equation}%
\bm{\ch} \defi \begin{pmatrix} \ch_1 & \ch_2 & \cdots & \ch_s \end{pmatrix}^T, \quad \text{and} \quad
\tilde{\hb} \defi \begin{pmatrix} \frac{\h{0}}{\al^{1/s}} & \frac{\h{1}}{\al^{2/s}} & \cdots & \frac{\h{s - 1}}{\al} \end{pmatrix}^T,
\end{equation}%
and the matrices
\begin{align}%
W(\al) &\defi \begin{pmatrix}%
\ibpos{\al,\be_1} & \ibpos{\al,\be_2} & \cdots & \ibpos{\al,\be_s}
\end{pmatrix}, \\
\tilde{W}(\al) &\defi \begin{pmatrix}%
\frac{\be_1}{\al} \ibpos{\al,\be_1} & \frac{\be_2}{\al} \ibpos{\al,\be_2} & \cdots & \frac{\be_s}{\al} \ibpos{\al,\be_s}
\end{pmatrix}.
\end{align}%
Using the introduced vectors and matrices we can write
\begin{equation}%
\sum_{i = 1}^s \frac{\ch_i}{\cv} \ibpos{\al,\be_i} = W(\al) \frac{\bm{\ch}}{\cv}, \quad \text{and} \quad
\sum_{i = 1}^s \frac{\ch_i}{\cv} \frac{\be_i}{\al}  \ibpos{\al,\be_i} = \tilde{W}(\al) \frac{\bm{\ch}}{\cv}.
\end{equation}%
We introduce the symbol $\circ$ for element-wise (Hadamard) multiplication, i.e., for column vectors $\bld{x}$ and $\bld{y}$, we have that $\bld{z} = \bld{x} \circ \bld{y}$ means that element $z(r) = x(r) y(r)$. Below we list some useful element-wise multiplications that we will use
\begin{align}%
\bm{\al} \circ  W(\al) &= \begin{pmatrix}%
1 & 1 & \cdots & 1 \\
\bigl( \frac{\be_1}{\al} \bigr)^{1/s} \unit_1 & \bigl( \frac{\be_2}{\al} \bigr)^{1/s} \unit_2 & \cdots & \bigl( \frac{\be_s}{\al} \bigr)^{1/s} \unit_s \\
\vdots & \vdots & & \vdots \\
\bigl( \frac{\be_1}{\al} \bigr)^{(s - 1)/s} \unit_1^{s - 1} & \bigl( \frac{\be_2}{\al} \bigr)^{(s - 1)/s} \unit_2^{s - 1} & \cdots & \bigl( \frac{\be_s}{\al} \bigr)^{(s - 1)/s} \unit_s^{s - 1}
\end{pmatrix}, \\
\exists j : \bm{\al} \circ \ibpos{\al,\be} &= \begin{pmatrix}%
1 &
\bigl( \frac{\be}{\al} \bigr)^{1/s} \unit_j &
\cdots &
\bigl( \frac{\be}{\al} \bigr)^{(s - 1)/s} \unit_j^{s - 1}
\end{pmatrix}^T, \\
\bm{\al} \circ \hb &= \al^{1/s} \tilde{\hb}.
\end{align}%
Moreover, for $\al \downarrow 0$, we determine that $\bm{\al} \circ  W(\al) \to W$, $\bm{\al} \circ  \tilde{W}(\al) \to \Lal W$, and $\bm{\al} \circ \ibpos{\al,\be} \to \bld{w}_j$, where $W$ and $\bld{w}_j$ are given in \eqref{eqn:convergence_horizontal_definition_W}-\eqref{eqn:convergence_horizontal_definition_w_j}.

We are now in a position to multiply both \eqref{eqn:app_proof_eta_i} and \eqref{eqn:app_proof_h} element-wise by $\bm{\al}$. For the element-wise multiplication of \eqref{eqn:app_proof_eta_i} we get
\begin{align}%
&\bm{\al} \circ \B{0,0} \frac{\hb}{\cv} + \bm{\al} \circ \sum_{i = 1}^s \frac{\ch_i}{\cv} \frac{\be_i}{\al} \bigl( \A{1,-1} + \al \A{0,-1} \bigr) \ibpos{\al,\be_i} + \frac{\ch_{s + 1}}{\cv} \frac{\be_{s + 1}}{\al} \bm{\al} \circ \bigl( \B{1,1} + \al \B{0,1} \bigr) \ibneg{\al,\be_{s + 1}} \notag \\
 &= - \frac{\be}{\al} \bm{\al} \circ \bigl( \A{1,-1} + \al \A{0,-1} \bigr) \ibpos{\al,\be}. \label{eqn:app_proof_eta_i_element-wise_multiplication}
\end{align}%
Combining \eqref{eqn:app_proof_eta_i_element-wise_multiplication} with the definition of the transition matrices $\B{0,0} = (1 + s)\rho \Lo - (1 + s)(\rho + 1)\I + s\Lo^T$, $\A{1,-1} = (1 + s)\rho \I$ and $\B{1,1} = (1 + s)\rho \M{0,s-1}$, and the matrix-vector notation described above, we obtain
\begin{align}%
&\Bigl( (1 + s)\rho \Lo - \al^{1/s} (1 + s)(\rho + 1) \I + \al^{2/s} s \Lo^T \Bigr) \frac{\tilde{\hb}}{\cv} + (1 + s)\rho \bm{\al} \circ \tilde{W}(\al) \frac{\bm{\ch}}{\cv} \notag \\
&+ \bm{\al} \circ \sum_{i = 1}^s \frac{\ch_i}{\cv} \frac{\be_i}{\al} \al \A{0,-1} \ibpos{\al,\be_i} + \frac{\ch_{s + 1}}{\cv} \frac{\be_{s + 1}}{\al} \bigl( (1 + s)\rho \M{0,s - 1} + \al \bm{\al} \circ \B{0,1} \bigr) \ibneg{\al,\be_{s + 1}} \notag \\
&= - \frac{\be}{\al} \bm{\al} \circ \bigl( (1 + s)\rho \I + \al \A{0,-1} \bigr) \ibpos{\al,\be}. \label{eqn:app_proof_eta_i_element-wise_multiplication_rewritten}
\end{align}%
Multiplying \eqref{eqn:app_proof_h} element-wise by $\bm{\al}$ gives
\begin{align}%
\Bigl( (1 + s) \rho (1 - q) \M{0,s - 1} + \al^{1/s} \I \Bigr) \frac{\tilde{\hb}}{\cv} - \bm{\al} \circ W(\al) \frac{\bm{\ch}}{\cv} = \bm{\al} \circ \ibpos{\al,\be}. \label{eqn:app_proof_h_element-wise_multiplication}
\end{align}%
Finally, if we let $\al \downarrow 0$ in \eqref{eqn:app_proof_eta_i_element-wise_multiplication_rewritten} and \eqref{eqn:app_proof_h_element-wise_multiplication} we obtain the limiting system of linear equations
\begin{align}%
L \tilde{\hb}^\textup{lim} + \Lal W \bm{\ch}^\textup{lim} &= - \Lbe \bld{w}_j - \Lchtpos \Lalt \func{+}{0}^{s - 1}\eb{0}, \label{eqn:app_proof_non-degenerate_1}\\
(1 + s)\rho(1 - q) \M{0,s-1} \tilde{\hb}^\textup{lim} - V \bm{\ch}^\textup{lim} &= \bld{w}_j. \label{eqn:app_proof_non-degenerate_2}
\end{align}%
We have constructed $2s$ equations, \eqref{eqn:app_proof_non-degenerate_1}-\eqref{eqn:app_proof_non-degenerate_2}, for the $2s$ unknowns $\tilde{\hb}^\textup{lim}$ and $\bm{\ch}^\textup{lim}$. The solution to this system of equations depends on $j$. So, let $\tilde{\hb}^\textup{lim}_j$ and $\bm{\ch}^\textup{lim}_j$ be the solution to \eqref{eqn:app_proof_non-degenerate_1}-\eqref{eqn:app_proof_non-degenerate_2} for a specific $j$. Then,
\begin{equation}%
\lim_{\al \downarrow 0} \Bigl| \frac{\ch_i}{\cv} \Bigr| \le \max_{1 \le j,r + 1 \le s} |\ch^\textup{lim}_j(r)|.
\end{equation}%
This concludes the proof of part (c).

%%%%%%%%%%%%%%%%%%%%%%%%%%%%%%%%%%%%%%%%%%%%%%%%%%%%%%%
%%%%%%%%%%%%%%%%%%%%%%%%%%%%%%%%%%%%%%%%%%%%%%%%%%%%%%%
%%%%%%%%%%%%%%%%%%%% BIBLIOGRAPHY %%%%%%%%%%%%%%%%%%%%%
%%%%%%%%%%%%%%%%%%%%%%%%%%%%%%%%%%%%%%%%%%%%%%%%%%%%%%%
%%%%%%%%%%%%%%%%%%%%%%%%%%%%%%%%%%%%%%%%%%%%%%%%%%%%%%%

\bibliographystyle{plain}%
\bibliography{ShortestExpectedDelayRoutingBibliography}%

\end{document}% 